\documentclass[12pt]{article}

\usepackage{multibib}
\usepackage{times}

\usepackage{lineno}
\usepackage{lipsum}
\usepackage{dsfont,eucal,mathrsfs}
\usepackage{bm}
\usepackage[normalem]{ulem}

\usepackage{amssymb}
\usepackage{amsmath}
\usepackage{mwe}

\renewcommand{\mathbf}{\boldsymbol}
\renewcommand{\mathcal}{\mathscr}

\usepackage{bm}
\usepackage{lipsum}

\usepackage{color}

\usepackage{epsfig,amsmath,amsfonts,amsthm,graphicx}
\usepackage{todonotes}
\usepackage{siunitx}

  \usepackage{algorithmicx}
    \usepackage[ruled]{algorithm}

\usepackage{algorithm}
\usepackage[noend]{algpseudocode}

\algnewcommand\algorithmicforeach{\textbf{for each}}
\algdef{S}[FOR]{ForEach}[1]{\algorithmicforeach\ #1\ \algorithmicdo}

\newcommand{\beq}{\begin{equation}}
\newcommand{\eeq}{\end{equation}}
\newcommand{\beqarr}{\begin{eqnarray}}
\newcommand{\eeqarr}{\end{eqnarray}}

\theoremstyle{plain}
\newtheorem{thr}{Theorem}[section]

\newtheorem{lem}[thr]{Lemma}

\newtheorem{prop}[thr]{Proposition}
\theoremstyle{definition}
\newtheorem{defi}[thr]{Definition}
\newtheorem{ex}[thr]{Example}
\theoremstyle{remark}
\newtheorem{remk}[thr]{Remark}
\theoremstyle{remark}

\DeclareMathOperator{\re}{Re}

\usepackage{ragged2e}
\usepackage{hyperref}

\usepackage[utf8]{inputenc}

\usepackage{amssymb, mathrsfs, mhchem, enumerate}
\usepackage{tikz,  titlesec, mathtools}
\usetikzlibrary{shapes.arrows}

\title{Emergent hypernetworks in weakly coupled oscillators  
}

\author
{Eddie Nijholt$^{1}$, Jorge Luis  Ocampo-Espindola$^{2}$,  Deniz Eroglu$^{3}$, \\ Istv\'{a}n Z. Kiss$^{2}$,  Tiago Pereira$^{1,4\ast}$\\
\\
\normalsize{$^{1}$Instituto de Ci\^encias Matem\'aticas e Computa\c{c}\~ao, Universidade de S\~ao Paulo, S\~ao Carlos, Brazil}\\
\normalsize{$^{2}$Department of Chemistry, Saint Louis University, St. Louis, USA}\\
\normalsize{$^{3}$Faculty of Engineering and Natural Sciences, Kadir Has University, Istanbul, Turkey}\\
\normalsize{$^{4}$Department of Mathematics, Imperial College London, London, UK}\\
\\
\normalsize{$^\ast$To whom correspondence should be addressed; E-mail:  tiago.pereira@imperial.ac.uk}
}

\date{}

\begin{document}

\baselineskip24pt

\maketitle 


\begin{abstract}
{

Networks of weakly coupled oscillators had a profound impact on our understanding of complex systems.  Studies on model reconstruction from data have shown prevalent contributions from hypernetworks with triplet and higher interactions among oscillators, in spite that such models were originally  defined as oscillator networks with pairwise interactions. Here, we show that hypernetworks  can spontaneously emerge even in the presence of pairwise albeit nonlinear coupling given certain triplet frequency resonance conditions. The results are demonstrated in experiments with electrochemical oscillators and in simulations with integrate-and-fire neurons. By developing a comprehensive theory, we uncover the mechanism for emergent hypernetworks by identifying appearing and forbidden frequency resonant conditions. Furthermore, it is shown that microscopic linear (difference) coupling among units results in coupled mean fields, which have sufficient nonlinearity to facilitate hypernetworks. Our findings shed light on the apparent abundance of hypernetworks and provide a constructive way to predict and engineer their emergence.}
\end{abstract}

\vspace{10pt}
\begin{center}
\small Published in \href{doi.org/10.1038/s41467-022-32282-4}{\textit{Nature Communications} \textbf{13}, 4849} (2022)
\end{center}

\baselineskip18pt

\clearpage

\phantomsection

\section*{Introduction}
Networks of weakly coupled oscillators are prolific models for a variety of natural systems ranging from biology \cite{Watts2011,kralemann2013vivo} and chemistry \cite{sebek2016complex, bick2017robust} to neuroscience \cite{Nat2,Ermentrout:10} via ecology \cite{blasius1999complex} to engineering \cite{matheny2019exotic}. Such networks serve as stepping stones to understand collective dynamics \cite{Smeal:10,Omel:12,Hong:11,kuramoto2003chemical} and other  emergent phenomena in networks \cite{Stankovski_RMP_2017,rodrigues2016kuramoto}. In these models, the interactions are described in a pairwise manner and the collective dynamics of a network can be predicted by the superposition of such pairwise interactions. 

Recent work, however, suggests that many networks described as pairwise interactions can be better described in terms of hypernetworks with triplet and quadruplet interactions among nodes \cite{RalfNC, giusti2015clique,  reimann2017cliques, bassett2018nature}. In fact, hypernetworks appear as suitable representations of certain dynamical processes found in physics \cite{millan2020explosive,bick2016chaos}, chemistry \cite{kori2014clustering} and neuroscience \cite{giusti2016two,bassett2017network}. This has ignited research aimed at understanding the impact of higher-order interactions on the dynamical behavior of complex systems \cite{grilli2017higher, skardal2019abrupt,mulas2020coupled, bilal2014synchronization}. 
 Moreover, besides considering hypernetworks as a good description of such models, we observed that hypernetworks could be revealed in data-driven model reconstructions when the original model is a network. Therefore,  a major puzzle  is why hypernetworks emerge as the fitting description of actual network data.

Here, we show that hypernetworks can describe experimental data of networks of electrochemical oscillators with nonlinear coupling. We uncover a mechanism that generates higher-order interactions as a model to describe oscillator networks from data. First, we show that sparse model recovery from data reveals higher-order interactions. We then develop a theory for the emergence of such higher-order interactions when the isolated system is close to a Hopf bifurcation.  
We provide an algorithm to reveal emergent hypernetwork and its emergent coupling functions for any network in disciplines ranging from neuroscience to chemistry. {\color{black} The emergent hypernetworks provide a dimension reduction that allows the characterization of critical transitions. }

\section*{Results}

\subsection*{Emergent hypernetworks in electrochemical experiments} 

We designed an experimental system with four oscillatory chemical reactions coupled with nonlinear feedback and delay arranged in a ring network (see Fig.~\ref{Fig1}~(a)). The set-up consists of a multichannel potentiostat interfaced with a real-time  controller and connected to a Pt counter, a Hg/Hg$_2$SO$_4$ sat K$_2$SO$_4$ reference, and four Ni working electrodes in 3.0 M sulfuric acid electrolyte.  At a constant circuit potential ($V_0$=1100 mV with respect to the reference electrode) and with an external resistance ($R_{ind}$=1.0 kohm) attached to each nickel wire, the electrochemical dissolution of nickel exhibits periodic current and electrode potential oscillations with a natural frequency of 0.385 Hz. 

Without coupling, we adjusted the natural frequency of each oscillator to have a ratio with respect to oscillator 1 as 
$\omega_2/\omega_1$ = 2.53 ($\approx 2.5$), $\omega_3/\omega_1$ =1.56 ($\approx 1.5$)  and $\omega_4/\omega_1$ = 2.53 ($\approx 2.5$)  with a set of resistors and capacitors ($C_{ind}$), see {Supplementary Note 1}.) The natural frequencies create opportunities for triplet resonances, as there are small detunings for $\omega_1 - \omega_2 + \omega_3 $ and $\omega_1 - \omega_4 + \omega_3 $, {\color{black} as well as pairwise resonances $\omega_2 \approx \omega_4$. }

The individual electrode potentials ($E_k$) were recorded and rescaled and offset corrected 
\begin{equation}
\tilde{E}_k=O_k[E_k - o_k],
\end{equation}
where $o_k$ and $O_k$ are the time-averaged electrode potential and amplitude rescaling factor, respectively. (The rescaling factors, $O_k = {0.5 ,1 , 0.5 , 1 }$ V were applied to counter the  different amplitudes of the slow oscillators.) A ring-coupling can be introduced with  external feedback (see Fig.~\ref{Fig1}~(b,c)) according to
\begin{linenomath}
\begin{eqnarray}\label{exp_feed}
V_k(t) = V_{0, k} + K \sum_{\ell=1}^{4}A_{k \ell}h[\tilde{E}_k(t), \tilde{E}_{\ell}(t-\tau)],
\end{eqnarray} 
\end{linenomath}
where $V_k(t)$ and $V_{0, k}$ are the applied and the offset circuit potential of the  $k$th  electrode, respectively, $K$ is the coupling strength, $A_{k \ell}$ is the adjacency matrix, $\tau$ is a time delay, and
\begin{linenomath} 
\begin{eqnarray}\label{non_feed}
h[\tilde E_k(t), \tilde E_{\ell}(t-\tau)]=(\tilde E_k(t)+\tilde E_k(t)^2)\tilde E_{\ell}(t-\tau).
\end{eqnarray} 
\end{linenomath}
This delayed nonlinear feedback modulates the impact of the coupled units with a bias towards  positive values (similar to a diode operation  in the $(-1,1)$ interval).  Note that this form of feedback is fundamentally different from previously applied nonlinear schemes \cite{bick2017robust} in that it does not produce obvious synchronization patterns, for example, one and multi-cluster states.

Figure~\ref{Fig1}~(d) shows the time series of the electrode potential for $K$=5.2 and $\tau$ =1.65 s. The slow  oscillators (1 and 3) have larger amplitudes and the time series exhibit nonlinear waveform modulations  without any obvious synchronization pattern (one-cluster state).  
\begin{figure}[htp]
    \centering
    \includegraphics[width=1\columnwidth]{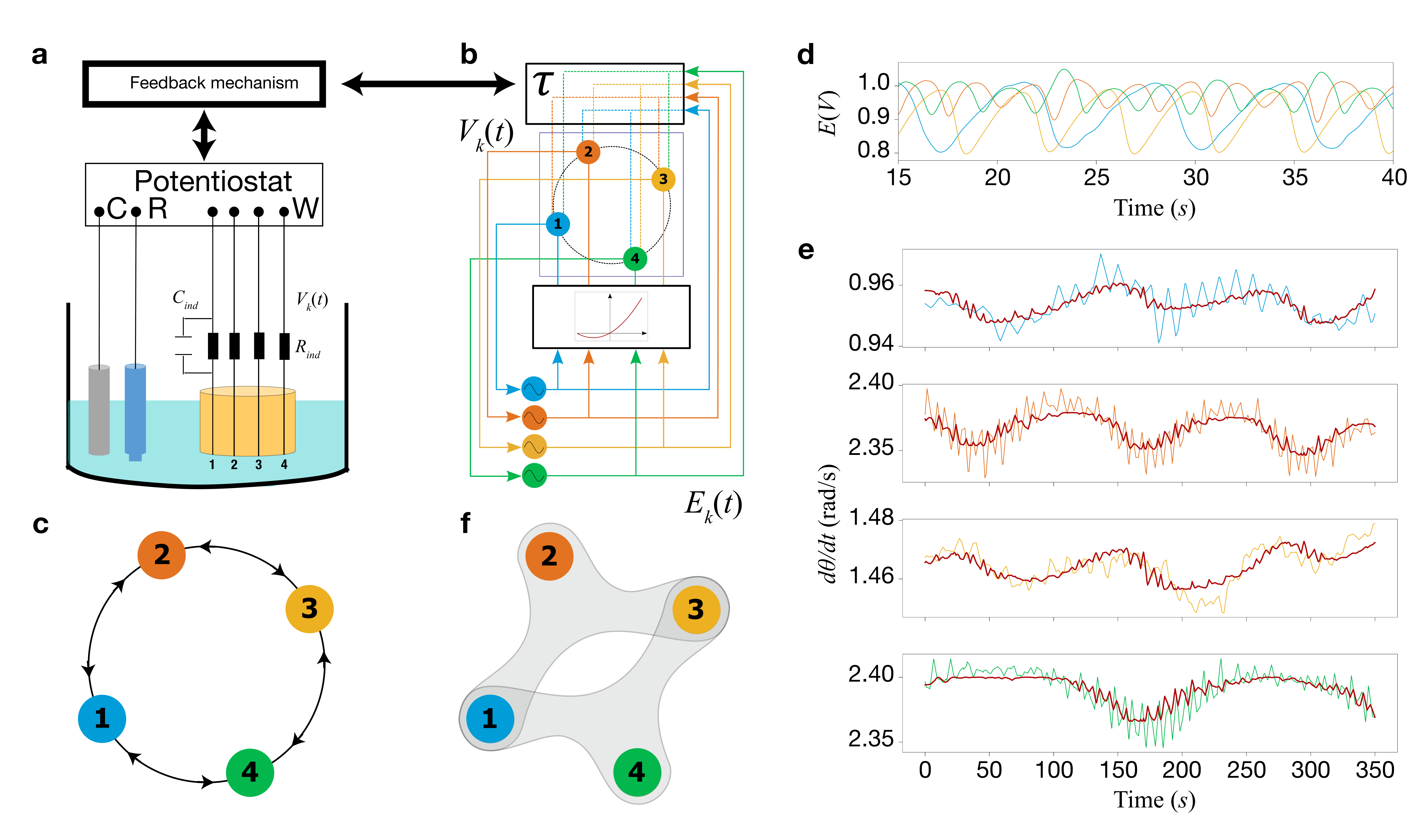}
    \caption{{\bf Emergent hypernetworks in an  electrochemical network experiment.} (a) Experimental setup. (b) Schematic illustration of the electrochemical experiment with the nonlinear feedback. The blue, orange, yellow, and green lines represent the elements 1 to 4, respectively. {\color{black}The electrode potential signals ($E_k$) of the four (nearly) isolated electrodes are nonlinearly modulated and fed back with a delay $\tau$ to the corresponding circuit potential ($V_k$), 
    which drives the metal dissolution. (The delay is implemented by storing the past data in the memory of 
    the computer.)  }  (c) Representation of the in a ring network topology used in the experiment.  (d) Electrode potential  time series.  (e) Filtered and fitted (dark red line) instantaneous frequency using LASSO for hypernetwork reconstruction corresponding from top to bottom to oscillators 1 to 4, respectively. (f) Experimental recovery of the phase interactions given by a hypernetwork.
}
\label{Fig1}
\end{figure}

From the potentials $\tilde E_k$ we  extract the  frequencies $\dot{\theta}_k$ and apply a first-order Savitzky-Golay filter with a time window of 45~s to remove the in-cycle  and short-range phase fluctuation, as shown in Figure~\ref{Fig1}~(e) (solid line). For each oscillator, a slow variation is seen as the oscillators slow down and speed up on a timescale of about 100 s (or 40 cycles); notably, the elements  1 and 3 exhibit similar  $\dot{\theta}_k$ oscillations, which are different from those in elements 2 and 4.
 
To describe the nature of the phase dynamics, we consider the slow triplet phase differences
\begin{linenomath}
\begin{eqnarray}
\begin{split}\label{slow_phases}
\phi_1 &= \theta_1  - \theta_2 + \theta_3 \\
\phi_2 &= \theta_1 - \theta_4 + \theta_3,
\end{split}
\end{eqnarray}
\end{linenomath}
which correspond to the triplet frequency detunings. 

The impact of triplet interactions on the dynamics can be extracted with a LASSO fit to 
\begin{linenomath}
\begin{equation}\label{eq:exp}
\dot{\theta}_k = \hat{\omega}_k(t) +   \sum_{j=1}^{2} C_j^k \sin (\phi_j) 
+ D_j^k \cos (\phi_j)
\end{equation}
\end{linenomath}
where $\hat{\omega}_k(t) = \hat{\omega}^0_k + \hat{\omega}^1_k t + \hat{\omega}^2_k t^2 $  is the fitted, slowly drifting (up to quadratic variation in time) natural frequency, and $C_j^k$ and $D_j^k$ are the amplitudes  of the sin and cos phase  coupling functions corresponding to the appropriate triplet phase differences. The strength of the triplet interactions $j=1,2$ (for $\phi_j$) on oscillator $k$ is given by the amplitudes $H_j^k = \sqrt{(C^k_j)^2  + (D^k_j)^2}$. 

The dynamics of oscillators 1 and 3 are  impacted by both triplet interactions; $\phi_1$ impacts oscillators 1 and 3 with amplitudes \num{4.9e-3} and \num{4.4e-3}, and $\phi_2$ with  \num{2.3e-3} and  \num{3.2e-3}, respectively. However, the dynamics of oscillators 2 and 4 are only impacted by triplet interactions $\phi_1$ (with amplitude \num{1.33e-2} ) and $\phi_2$ (\num{1.7e-2}), respectively.  These triplet interactions describe phase fluctuations  over the long time scale (red curves in Fig.~\ref{Fig1}~(e)). Therefore, we can conclude that the phase dynamics of the oscillators coupled in a ring can be described by a hypernetwork shown in Fig.~\ref{Fig1}~(f). 

{\color{black} The fact that model recovery provides triplets as the best description is rather puzzling. Also given that the resonant behavior $\omega_2 \approx \omega_4$ did not appear in the model recovery from data. This suggests an interplay between the resonant frequencies and the network topology. The question arises, which resonances/triplet interactions emerge from a large number of possibilities in a given network, natural frequencies, and nonlinear coupling? } An outstanding question is  what is the origin of these triplet interactions that were generated by pairwise physical coupling?

\subsection*{A theory for emergent higher-order interactions}

To answer these questions, we develop a theory that captures  the important characteristics  of the experiments: nonlinear coupling and triplet  resonance conditions. We consider the networks 
\begin{linenomath}
\begin{equation}\label{Eq1}
\dot{z}_k = f_k(z_k) + \alpha \sum_{\ell=1}^n A_{k\ell} h_k(z_k, z_{\ell})
\end{equation}
\end{linenomath}
where $z_k \in \mathbb{C}$ is the state of the $k$th oscillator, $h_k:  \mathbb{C} \times  \mathbb{C} \rightarrow  \mathbb{C}$  is the pairwise coupling function, $A_{k \ell}$ is the  adjacency matrix, and $\alpha > 0 $ is the coupling strength. When the  isolated system is close to a Hopf bifurcation, the dynamics is described by $f_k (z_k) = \gamma_k z_k + \beta_k z_k |z_k|^2$ \cite{shil2001methods}. The Hopf bifurcation is a common route to oscillations in nonlinear systems and describes  the appearance of oscillations in applications \cite{kralemann2013vivo,Nat2,Ermentrout:10,matheny2019exotic,sebek2016complex}.  Our proofs are valid  for $\gamma_k= \lambda+ i \omega_k$ with small  $\lambda$ and $\omega_k$ satisfying resonance conditions. We  fix $\beta_k=-1$, but this value is immaterial. We develop a normal form theory to  eliminate unnecessary terms of $h(z_k, z_{\ell})$ and to expose higher-order ones that   predict the dynamics.  To a network of the form of Eq. (\ref{Eq1})  we  associate non-resonance conditions that allow us to get rid of the leading interaction terms in $\alpha$.  

Since $h(z_k, z_{\ell})$ is a linear combination of monomials and the theory can be applied to each monomial independently, we assume first that $h(z_k, z_{\ell})$ is a single monomial of the form 
\begin{linenomath}
\begin{equation}\label{ResMono}
h(z_k, z_{\ell})  = z_k^{d_1}\bar{z}_k^{d_2}z_{\ell}^{d_3}\bar{z}_{\ell}^{d_4}
\end{equation}
\end{linenomath}
for non-negative numbers $d_1, \dots, d_4$. {\color{black} Our major theoretical result is a formulation}  of a non-resonance condition   given by
\begin{linenomath}
\begin{equation}\label{ResC}
(d_1 - d_2 -1)\omega_k + (d_3 - d_4)\omega_{\ell} \not= 0.
\end{equation}
\end{linenomath}
\textcolor{black}{This condition shows up naturally in our approach, as a monomial Equation \eqref{ResMono} can only be eliminated by a transformation that divides by the left-hand side of Equation \eqref{ResC}. Hence, an interaction term in the coupling function $h$ given by Equation \eqref{ResMono} can only be removed if the non-resonance condition is satisfied.}
The non-resonance condition is defined as the union over all non-resonance conditions of its monomial terms. The network non-resonance conditions are given by the union over all non-resonance conditions of $h(z_k, z_{\ell})$
for which $A_{{k \ell}} \not= 0$. Our result is the following:

{In Methods, we show that given Eq. (\ref{Eq1}) with  $h: \mathbb{C} \times \mathbb{C} \rightarrow \mathbb{C}$ a smooth map with vanishing constant terms, under the network non-resonance conditions, there is a coordinate transformation that eliminates pairwise interaction terms and reveals  the higher-order  interactions.} 
The proof consists of two main steps:

i) Existence of a polynomial change of variables.
 Consider
\begin{linenomath}
\begin{equation}
u_k = z_k - \alpha P_k
\end{equation}
\end{linenomath}
for some polynomials $P_k$. The goal is to design $P_k$  such that in the variables $u_k$ interaction terms linear in $\alpha$ vanish. We obtain higher-order interactions of order $\alpha^2$. For Eq. (\ref{Eq1}) we use
\begin{linenomath}
\begin{equation}\label{Trans}
P_k(z) = \sum_{\ell =1}^n A_{k \ell}  \tilde{h}_{k \ell} (z_k, z_{\ell})\, ,
\end{equation}
\end{linenomath}
where $\tilde{h}_{k \ell}(z, w)$ is the function obtained from $h(z, w)$ by  transforming each monomial  according to the following replacement rule:
\begin{linenomath}
\begin{equation}\label{Rr}
z^{d_1}\bar{z}^{d_2}w^{d_3}\bar{w}^{d_4} \mapsto\frac{z^{d_1}\bar{z}^{d_2}w^{d_3}\bar{w}^{d_4}}{(d_1-1)\gamma_k + d_2\bar{\gamma}_k + d_3\gamma_{\ell} + d_4\bar{\gamma}_{\ell} }\,
\end{equation}
\end{linenomath}

\textcolor{black}{Note that the imaginary part of the denominator in Equation \eqref{Rr} is precisely the left-hand side of Equation \eqref{ResC}. }While bringing the equations to the new form, we face a major challenge  to understand the combinatorial behavior of the Taylor coefficients during the transformation. We define a bracket on the space of polynomials to track these coefficients. 

ii) Dealing with transformed isolated dynamics. The  second major challenge lies in the fact that another coordinate transformation is needed to eliminate terms coming from the isolated dynamics $f_k$. Indeed, as we eliminate coupling terms linear in $\alpha$, other terms linear in $\alpha$ appear due to the isolated dynamics. A remarkable fact is that the same non-resonance conditions also ensure that the second transformation exists.

Our theorem is applicable to a much broader class of coupling functions and network formalisms than what is described by Eq. (\ref{Eq1}). A rich variety of new interaction rules can emerge, depending on the specifics of the set-up (see Supplementary Note 2). 

Applying the replacement rule Eq. (\ref{Rr}) we obtain  
\begin{linenomath}
\begin{align}\label{Eq1-transf-1}
\dot{u}_k = f_k(u_k) - \alpha^2 \left\{ \sum_{\ell = 1}^n  \sum_{p = 1}^n  \left[  A_{k \ell} A_{k p} \prescript{1}{}{G}_{k}^{\ell p}(u_k, u_{\ell}, u_p) 
-  A_{k \ell} A_{\ell p} \prescript{2}{}{G}_{k}^{\ell p}(u_k, u_{\ell}, u_p) \right]\right\},
\end{align}
\end{linenomath}
\noindent
up to higher-order terms in $\alpha$ and $u$. 
In Methods, we discuss the new coupling functions $\prescript{1}{}{G}_{k}$ and $\prescript{2}{}{G}_{k}$ some their properties. The coupling  is now $\alpha^2$ explaining anomalous synchronization transitions that appears in networks (see {Supplementary Note 3}). 

\subsection*{Emergent hypernetworks explain experimental data}

Similar to the experiments we consider a ring of four oscillators with coupling function
\begin{linenomath}
\begin{equation}\label{degree3}
h(z,w) = z \bar w + z^2 \bar w. 
\end{equation} 
\end{linenomath}
\noindent 
Instead of delay, the oscillators are coupled through a conjugate variable that enables a streamlined theoretical treatment. Close to a Hopf bifurcation, the delay would have an effect of advancing the oscillations over half a period. {\color{black}As before, we consider $\omega_1 - \omega_2 + \omega_3$ and $\omega_1 - \omega_4 + \omega_3$ to be close to zero, so, capturing the triplet resonance in the experiments.} We apply our theory to this case to unravel how higher-order interactions appear in the data. 

The coupling function is a combination of $z \bar w$ and $z^2 \bar w$, providing $d_1 = 1$ and $d_4 = 1$ for the first monomial and $d_1 = 2$ and $d_4  = 1$ for the latter. The resonance condition Eq. (\ref{ResC}) is satisfied for both. Using the replacement rule Eq. (\ref{Rr}), we find 
\begin{linenomath}
\begin{eqnarray}\label{Trans}
u_k=z_k+\alpha \left( \frac{z_{k-1}z_k}{\bar \gamma_{k-1}} + \frac{z_{k}z_{k+1}}{\bar \gamma_{k+1}} + \frac{z_{k-1}^2\bar{z}_k}{\gamma_{k-1} + \bar{\gamma}_{k}} 
+
\frac{z_k^2\bar{z}_{k+1}}{\gamma_k + \bar{\gamma}_{k+1}}
\right) 
\end{eqnarray}
\end{linenomath}

Each node equation contains  $16$ interaction terms as in Eq. (\ref{Eq1-transf-1}).  We discuss some of these terms for the first node. $^2 G^{2 3}_1$ appears as  node $1$ is connected to node $2$ and $2$ to $3$. This interaction is resonant, see Figure \ref{Fig2} (a).  $^2 G^{4 3}_1$ appears because node $1$ is connected to $4$ and node $4$ to $3$.  This term is also resonant, see Figure \ref{Fig2} (b). $^1 G^{2 4}_1$ is nonzero and nonresonant. This term appear as $1$ is  directed connected to $2$ and $4$, see Figure \ref{Fig2} (c). Finally,  the term $^2 G^{2 4}_1$ is a forbidden, the term would appear from an interaction of $1$ to $2$ and from $2$ to $4$, however, in the original network the later interaction is absent, see Figure \ref{Fig2} (d).  Remarkably, not all interactions are relevant when the goal is to describe slow oscillations in the phases.

\begin{figure}[htp]
    \centering
    \includegraphics[width=1\columnwidth]{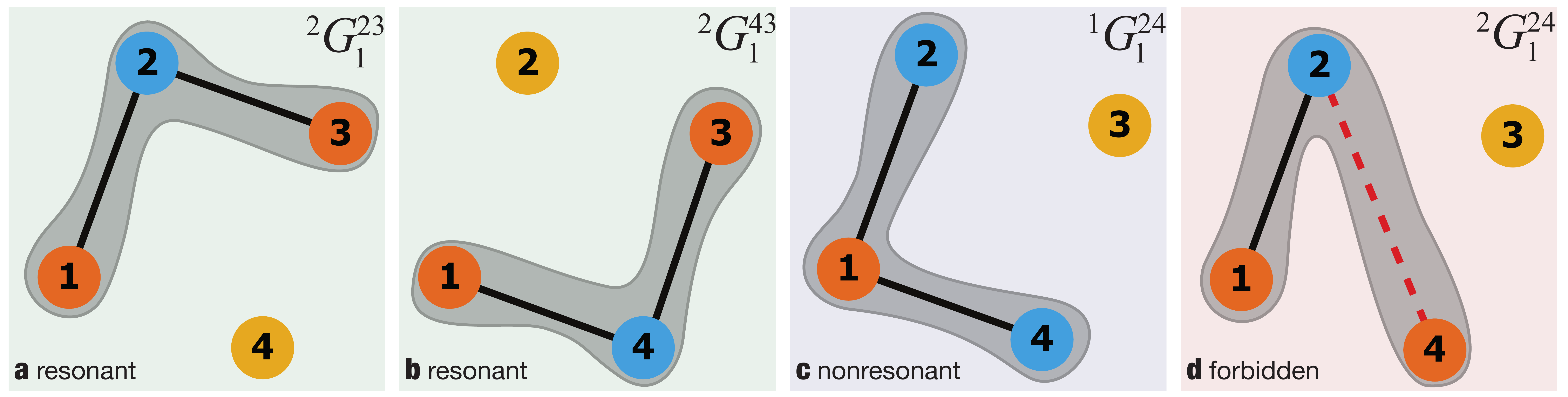}
    \caption{{\bf Emergent higher-order interactions from the original ring network}. 
Coupling functions appearing in Eq. (\ref{Eq1-transf-1}) of node $1$.  Colors correspond to signs in the phase combination with blue standing for positive and orange for negative. (a) resonant interaction term appearing as $^2G^{2 3}_1$.  (b) resonant interaction term appearing as $^2 G^{4 3}_1$. Finally, (c) is a nonresonant term and (d) $^2 G^{2 4}_1$ is a forbidden term (it does not appear). These new interaction terms can be predicted from the combinatorics of the original network and coupling function. }
    \label{Fig2}
\end{figure}

{\color{black} 
Indeed, once we analyse the phases in the new equations, the coupling term coming from $^2G_1^{23}$ will lead to oscillations with frequency close to $\omega_1 - \omega_2 + \omega_3$ while the term coming from $^2G_1^{43}$ leads to a frequency close to $\omega_1 - \omega_4 + \omega_3$. This implies that both terms are slowly varying. In contrast, the term coming from $^2G_1^{24}$ leads to oscillations with frequency $\omega_1 - \omega_2 + \omega_4 \approx \omega_1$ and is fast oscillating in comparison to the slow terms with small frequencies. In virtue of the averaging theory, such fast oscillating terms can be neglected. In fact, only resonant terms  connected by local trees in the original graph will survive such as the resonant ones involving $\omega_1 - \omega_2+ \omega_3$ and $\omega_1 - \omega_4+ \omega_3$.} This yields  
\begin{linenomath}
\begin{eqnarray}
\begin{split}
\label{hn}
\dot u_1 &= f_1(u_1)
- \alpha^2 \eta_{12} 
u_1^2 \bar u_2 u_3 - \alpha^2  \eta_{14}
u_1^2 \bar u_4 u_3 \label{eqr} \\
\dot u_2 &= f_2( u_2) 
- \alpha^2 \zeta_{231}
 u_2^2 \bar u_1 \bar u_3 \\
\dot u_3 &= f_3(u_3) 
- \alpha^2 \eta_{32}
u_3^2 \bar u_2 u_1 - \alpha^2 \eta_{34}
u_3^2 \bar u_4 u_1 \\
\dot u_4 &= f_4(u_4)
- \alpha^2 \zeta_{431}u_4^2 \bar u_1 \bar u_3 
\end{split}
\end{eqnarray}
\end{linenomath}
\noindent
where
$\eta_{pq} = \frac{1}{\gamma_p + \bar \gamma_q}$
and  $\zeta_{pqr} = \frac{2}{\gamma_p + \bar \gamma_q }
+\frac{2}{\gamma_p + \bar \gamma_r } 
+ \frac{1}{\bar \gamma_q } + \frac{1}{\bar \gamma_r }$.
Writing $u = re^{i\theta}$ we obtain equations for the phases $\theta$.  The averaging theorem  gives
\begin{linenomath}
\begin{eqnarray}\label{EqsPhaseRed}
\begin{split}
\dot \theta_{1} &= \omega_{1} - \alpha^2 r_0^3 \left[ \rho_{12}(\phi_1) + \rho_{14} (\phi_2) \right],  \\
\dot \theta_{2} &= \omega_{2} - \alpha^2 r_0^3    \sigma_{231}(\phi_1)  \\
\dot \theta_{3} &= \omega_{3} - \alpha^2 r_0^3 \left[ \rho_{32}(\phi_1) + \rho_{34} (\phi_2) \right],  \\
\dot \theta_{4} &= \omega_{4} - \alpha^2 r_0^3    \sigma_{431}(\phi_2), 
\end{split}
\end{eqnarray} 
\end{linenomath}
where the phases $\phi_1$ and $\phi_2$ are given in Eq. (\ref{slow_phases}). 
The functions $\rho$ and $\sigma$ are provided in the {Supplementary Note 4}. The emergent hypernetwork explains the experimental fitting found in Eq. (\ref{eq:exp}). These functions represent hyperlinks as shown in Figure \ref{Fig1} (f).

The phase triplets $\phi_1$ and  $\phi_2$ are  revealed from phase reduction in the normal form and  they are not obvious from the original Eq. (\ref{Eq1}). We confirm these  predictions by direct simulations of  Eq. (\ref{Eq1}) ({Supplementary Note 5}). We present examples for a three-node path in Supplementary Note 6 and a six-node network in Supplementary Note 7.

\subsection*{Predicting the slow phase interactions in experiments}

In {Supplementary Note 3}, we show that the experimental recovery of a hypernetwork is not an artifact. \textcolor{black}{Rather, we prove that imposing sparsity unavoidably leads to the recovery of the normal form instead. Indeed, as the recovery allows for a small least square deviation between the data and the model, the recovery finds the hypernetwork as a simpler description of the system. So,} by measuring the original variables and attempting a model recovery while imposing sparsity,  model recovery learns only the higher-order interactions. We now use the emergent network prediction for the ring network with the corresponding resonance conditions as in the experiment to explain the slow phase dynamics.

From the data we  extract the slow phases $\phi_1$ and $\phi_2$ as shown in Figure \ref{Fig3} in solid lines. Using our theory, from Eq. (\ref{EqsPhaseRed}), we obtain that
\begin{linenomath}
\begin{eqnarray}\label{Slow}
\dot{\phi}_i = \Omega_i + \sum_{j=1}^2 a_{ij} \cos \phi_j +  b_{ij} \sin \phi_j
\end{eqnarray}
\end{linenomath}
where  $a$'s and $b$'s are given in terms of the functions $\sigma$ and $\rho$ in Eq. (\ref{EqsPhaseRed}) see {Supplementary Note 5}.  We treat $a$'s  and $b$'s as fitting parameters from the vector field in Eq. (\ref{Slow}) obtained from first principles, since the corresponding coupling parameter and amplitudes are unknown. The resulting solutions agree  with the experimental data as seen in Figure \ref{Fig3}.  {\color{black} Our findings are not strictly limited to  electrochemical oscillators. As shown in  Supplementary Note 9, we detected the same hypernetworks in  nonlinearly coupled integrate-and-fire neuron models.}

\begin{figure}[h]
    \centering
    \includegraphics[width=0.7\columnwidth]{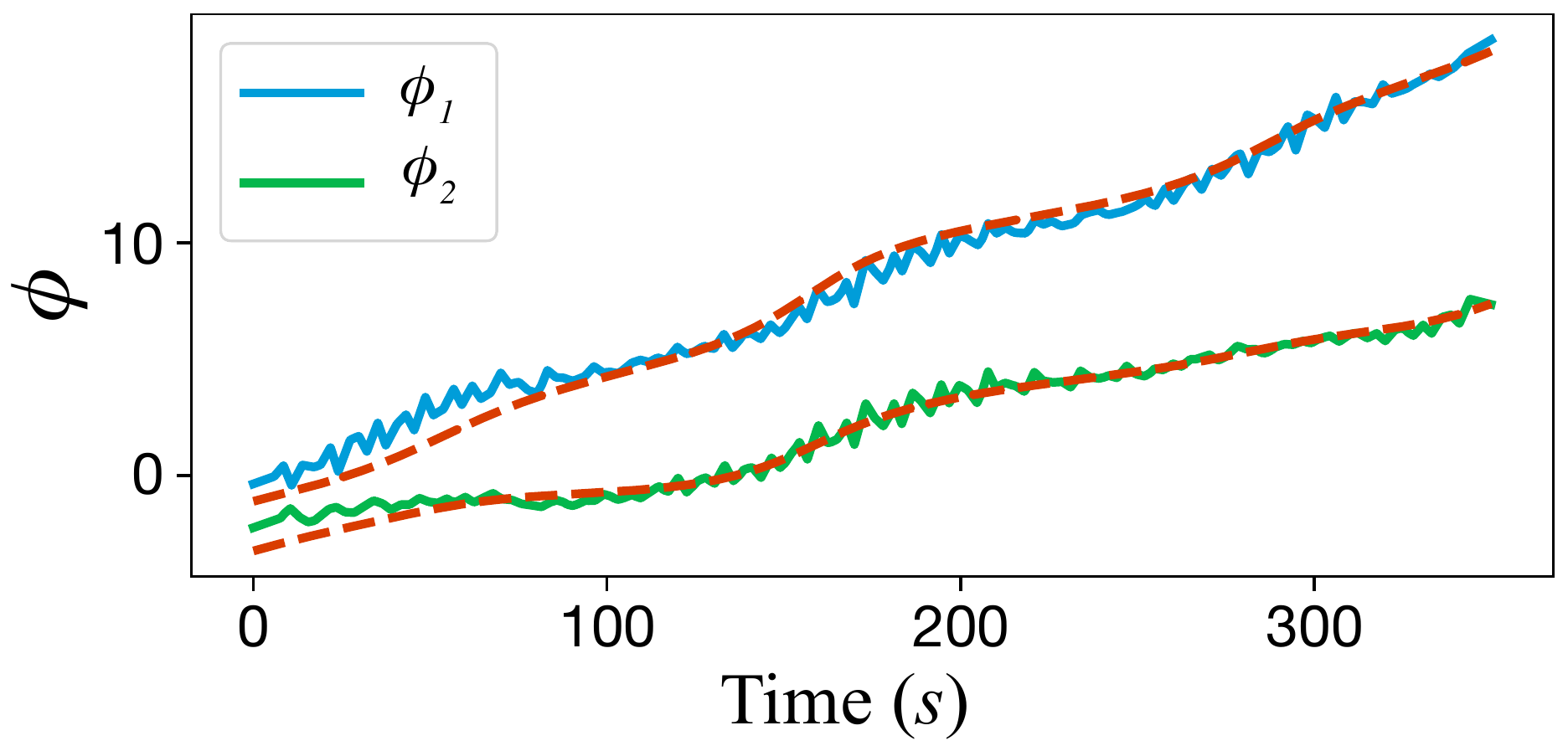}
    \caption{{\bf Normal form theory explains the experimental results}. We show the time series of the slow phase $\phi_1$  and $\phi_2$ from experimental data (solid) and the prediction of the emergent hypernetwork (dashed) capturing higher-order interactions. The vector field describing the phase interaction is obtained from first principles. The coefficients of the vector field are obtained by least-square minimization. 
}
    \label{Fig3}
\end{figure}

\subsection*{Emergent hypernetworks among network modules coupled through mean-fields}

{\color{black} The requirement of a nonlinear coupling, at first sight, seems to be a limitation for practical applications. However, here we analyze how  hypernetworks emerge in modular networks with microscopic pairwise coupling through phase differences.

We consider four subpopulations of $N$ interacting Kuramoto oscillators \cite{Stankovski_RMP_2017}. Nodes in each subpopulation interact strongly among themselves with coupling strength $\mu$ and weakly between subgroups with coupling strength $\alpha$, see Figure \ref{OAE}. As we will show at the macroscopic mean-field level, the interaction is nonlinear. According to our theory, although the mean-fields have a pairwise interaction, their model recovery will be in terms of hypernetworks. We first consider the microscopic description; each oscillator is described by 
\begin{linenomath}
\begin{equation}
\dot{\psi}_{km} = \omega_{km} + \frac{\mu}{N}\sum_{n=1}^N\sin(\psi_{kn} - \psi_{km}) +  \sum_{\ell=1}^4 A_{kl} \left( \frac{\alpha}{N} \sum_{n=1}^N\sin( \psi_{ln} - \psi_{km}) \right)
\end{equation}
\end{linenomath}
or in terms of mean-fields
$
\dot{\psi}_{km} = \omega_{km} + {\rm Im} \left( \mu  z_{k}  + \alpha  \sum_{} A_{kl} z_{l}  \right) e^{-i \psi_{km}}
$
where
\begin{linenomath}
\begin{equation}
z_k = \frac{1}{N}\sum_{m=1}^N e^{i \psi_{km}}
\end{equation}
\end{linenomath}
is the mean-field of the subpopulation $k$.  The frequencies $\omega_{km}$ are distributed according to a Lorenzian $\rho(\omega,\Omega_k,\sigma_k)$  where $\Omega_k$ is the mean subpopulation frequency and $\sigma_k$ is the frequency dispersion.  Applying the Ott-Antonsen ansatz \cite{RalfNC}, we obtain the macroscopic equations describing the mean-fields in the limit $N\rightarrow \infty$ as 
\begin{linenomath}
\begin{equation}
\dot{z}_k = f_k(z_k) + \sum_{l=1}^4 A_{kl} h(z_k,z_l) 
\end{equation}
\end{linenomath}
where $f_k$ is the Hopf normal form with constants $\gamma_k = (i \Omega_k +\mu - \sigma_k)$ and $\beta_k = -\mu$ and 
\begin{linenomath}
\begin{equation}
h(z_k,z_l) = \alpha z_l + \alpha \bar z_l z_k^2, 
\end{equation}
\end{linenomath}
thus, in the macroscopic description the coupling is nonlinear. 
We interpret $\alpha$ as a bifurcation parameter and deal with $\alpha z_l$ as a nonlinear term as in bifurcation theory. We consider the ensemble frequencies to satisfy the resonance conditions $\Omega_1 + \Omega_3 \approx 2\Omega_2$ and $\Omega_2 + \Omega_4 \approx 2\Omega_1$. At $\alpha=0$ each subpopulation will have an order parameter behaving as 
$
z_k(t) = r_k e^{i \theta_k(t)}
$
where
$
r_k = \sqrt{\frac{\mu- \sigma_k}{\mu}}
$
and $\dot \theta_k = \Omega_k$. To obtain the phase model, we bring the network to its normal form and apply the phase reduction.  In  Supplementary Note 10, we perform the calculations of such resonance conditions to obtain the new normal form equations.  After discarding  nonresonant terms the phase equations of the mean-fields read as

\begin{eqnarray}\label{thetaOA}
\begin{split}
\dot{\theta}_{1,3} &= \Omega_{1,3} +F_{1,3}(\varphi_1)\\ 
\dot{\theta}_{2,4} &= \Omega_{2,4} +F_{2,4}(\varphi_2)  
\end{split}
\end{eqnarray}
where $F_i$ is a linear combination of sine and cosine.

Next, we fix the ensemble frequencies as $\Omega_1 = 2, \Omega_2 = 3, \Omega_3 = 4$ and $\Omega_4 = 1$ as well as the coupling strengths $\mu=0.5$, $\sigma_k = 0.48$ yielding $r_k = 0.15$ and $\alpha=0.1$ for all subpopulations.  We numerically integrate the mean-field equations and obtain the complex fields $z_1(t),z_2(t),z_3(t)$ and $z_4(t)$ which enables us to extract the phase dynamics $\theta_1(t),\theta_2(t), \theta_3(t)$ and $\theta_4(t)$. Performing a Lasso regression we recover the vector fields of Eq. (\ref{thetaOA}) confirming the theoretical prediction of higher order interactions,  {see Supplementary Note 10}.

\begin{figure}[htp!]
    \centering
    \includegraphics[width=0.6\columnwidth]{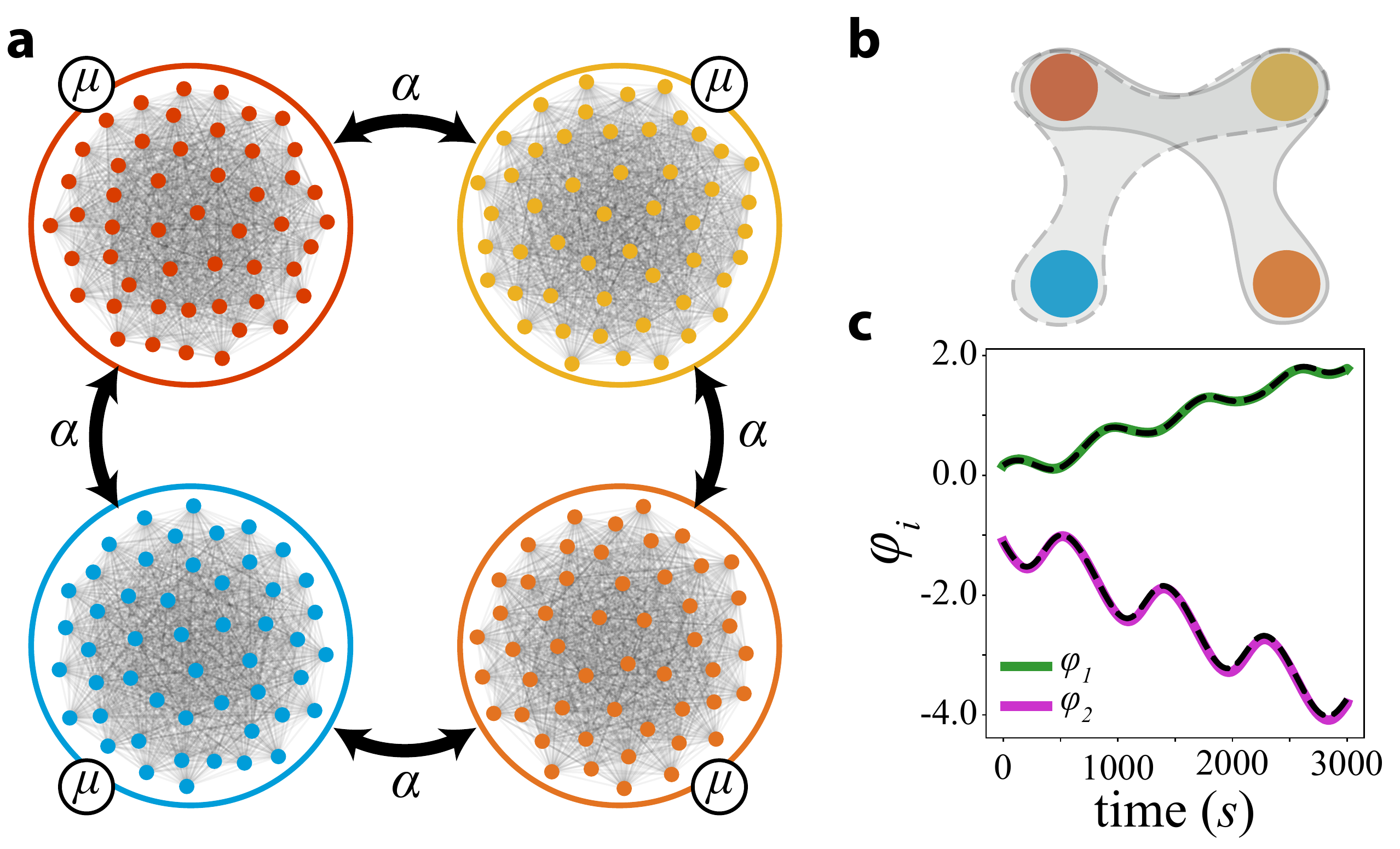}
    \caption{{\color{black}{\bf Interacting subpopulations lead to higher order interaction of mean-fields.}  a) The original network of coupled subpopulations (with four distinct colours, namely, red, yellow, blue and orange). Oscillators are interacting by an internal coupling constant $\mu$ and inter-subpopulations coupling constant $\alpha$. b) Higher order phase interaction of the  mean-fields represented with the same colors as in a) (red, yellow, blue and orange).  Applying our approach we uncover that the phase interaction between the mean-fields is described by a hypernetwork. c) The mean-field slow phase variables $\varphi_{1}$ (green) and $\varphi_{2}$ (purple) were computed from the data collected from the simulations of mean fields on the associated network. The  dashed curve is the simulation of the vector field of the slow phases $\varphi_{1,2}$ reconstructed from data using the Lasso method.}}
\label{OAE}
\end{figure}

As before, we introduce the slow phases
\begin{linenomath}
\begin{eqnarray}
\begin{split}
\varphi_1 = \theta_1 -2 \theta_2  +\theta_3,  \\
\varphi_2 = \theta_2 -2 \theta_1 + \theta_4.
\end{split}
\end{eqnarray}
\end{linenomath}
The theory predicts the higher order interaction between the slow phases as $\dot{\varphi}_k = \varepsilon_k + G_k(\varphi_1,\varphi_2)$, as shown in {Supplementary Note 10}. The fitting the predicted vector field of $\varphi$ to the data is excellent as can be observed in {Figure \ref{OAE} c).}

For these four subpopulation on a ring, the condition on the frequencies is close to the subspace 

$
V_{res} = \{ \Omega_1 + \Omega_3 = 2\Omega_2, \, \Omega_2 + \Omega_4 =  2\Omega_1\}\,, 
$
forming a co-dimension $2$ resonance surface.  That is, the emergence of hypernetworks is generic in a two parameter family of frequencies.
}

\section*{Discussion} 
We have uncovered a mechanism by which nonlinear pairwise interactions with triplet resonance conditions result in nontrivial phase dynamics on a hypernetwork. { \color{black} Such interactions traditionally were attributed in brain dynamics  to synaptic transmission between two neurons mediated by chemical messengers from a third neuron (heterosynaptic plasticity) \cite{Chistiakova:2014aa}. Our findings 
provide an alternative mechanism. On one hand, this finding shows that phase dynamics can be mediated through `virtual' interactions not physically present in the system. On the other hand, such  a mechanism could be leveraged to design interactions between remote components not directly connected but instead having correlations in natural frequencies. }

The experimental system with {a generic network motif with a ring of four} electrochemical oscillators presented here was an example, where a relatively simple nonlinear modulation of the  coupling induced a  hypernetwork driven phase dynamics.  {Networks with a ring topology are selected for the experiment since they are common for many network based complex systems, e.g., in lasers, biological systems, neuronal dynamics and many disciplines \cite{popovych2011delay, takamatsu2001spatiotemporal}}.  Such nonlinear modulation of the coupling  can be quite general in  gene expressions; for example, it was used to describe the coupling among circadian cells through  Michaelis-Menten mechanism where  coupling from one cell modulated the maximum gene expression rate in the other \cite{Schroder:2012bb}. 

{\color{black} Strikingly,  we showed that the coupling resulting in mean-field coupling  among network modules has sufficient nonlinearity to facilitate hypernetwork interactions. } {\color{black} In particular, event related modulation of spectral responses of  magnetoencephalogram (MEG) recordings  (i.e., modulation of frequency-specific oscillations in the motor network established by a handgrip task) have shown very strong evidence for nonlinear, between-frequency coupling  of remote brain regions\cite{friston2010}. Our results strongly suggest that in these MEG recordings, given the  appropriate  resonances and nonlinearities, hypernetwork description could facilitate the long-range modulation of frequencies.} In conclusion, the findings  open new avenues for hypernetwork based description and engineering of complex systems with heterogeneous frequencies and nonlinear interactions.

\section*{Methods}

Our results give an algorithmic procedure for obtaining a hypernetwork that accurately describes the observed behavior of the original system. This emergent higher order system depends on details of the given network, the original coupling function and the resonance relations among the phases. 
{\color{black}{
\subsection*{Normal form calculations}
In {Supplementary Note 2},  we consider ODEs of the general form
\begin{align}\label{theODE00methods}
\dot{z}_k &= \gamma_kz_k - \beta_k z_k|z_k|^2 + \alpha H_k(z_1, \dots, z_n)\, , \quad k \in \{1, \dots, n\}\, ,
\end{align}
with $z_k \in \mathbb{C}$ and $\alpha \in \mathbb{R}$. The numbers $\beta_k, \gamma_k \in \mathbb{C}$ are assumed non-zero, and we furthermore write $\gamma_k = \lambda + i\omega_k$. Here $\lambda \in \mathbb{R}$ is seen as the bifurcation parameter for a Hopf bifurcation, and we assume the interaction functions $H_k: \mathbb{C}^n \rightarrow \mathbb{C}$ to be smooth (i.e. $C^{\infty}$) for convenience. Moreover, we initially assume each $H_k$ satisfies $H_k(0) = 0$ and $DH_k(0) = 0$, though the condition on its  derivative is later dropped.

\noindent Our main result shows that the ODE \eqref{theODE00methods} can be put in a normal form that allows us to predict the phase dynamics of the oscillators. We do this by using two successive  transformations:
\begin{align}
  w_k &= z_k - \alpha P_k(z); \label{transforrr1}\\
  u_k &= w_k - \alpha Q_k(w), \label{transforrr2}\,
\end{align}
\noindent
for some appropriately chosen polynomials $P_k$ and $Q_k$. The first of these coordinate transformations is used to remove the term $\alpha H_k(z)$ from the Equation \eqref{theODE00methods}. This will generate additional terms in $\alpha^2$ that may be expressed in the coefficients of $H_k$ and $P_k$ following certain combinatorial rules. We manage this combinatorial behavior by introducing a special bracket $[\bullet|| \bullet]$ on the space of polynomials. In addition to these new interaction terms, the transformation will also produce terms in $\alpha$ involving $P_k$ and $\beta_k z_k|z_k|^2$, which obscure an interpretation of the system as a (hyper) network. We therefore remove these additional terms using the second coordinate transformation. A crucial observation here is that the non-resonance conditions needed for the first transformation are sufficient to ensure the second. We are able to prove this using the precise bookkeeping enabled by the aforementioned bracket. 

When dealing with the case where $DH_k(0) \not= 0$, we instead remove only the non-linear terms in $H_k$ using the transformations \eqref{transforrr1} and \eqref{transforrr2}. This reveals higher order terms as before. Even though $DH_k(0)$ accounts only for  nonresonant terms by assumption, this linear term will nevertheless cause an overall frequency shift that has to be accounted for. More precisely, if we denote by $\Omega$ the diagonal matrix with entries the frequencies $\omega_1, \dots, \omega_n$,  then the natural frequencies in the coupled case will be given by the imaginary part of the eigenvalues of $i\Omega + \alpha DH(0)$. Here we have set $H = (H_1, \dots, H_n)$. These new frequencies can be approximated by standard eigenvalue perturbation techniques. 
}}

\subsection*{Properties of the coupling functions $\prescript{1}{}{G}_{k}^{\ell p}$ and $\prescript{2}{}{G}_{k}^{\ell p}$}

Applying the transformation of the theorem to Eq. (\ref{Eq1}) yields a new system of the form Eq. (\ref{Eq1-transf-1}). In {Supplementary Note 2}, we show that 
\begin{linenomath}
\begin{align}\label{g1gen}\
\begin{split}
\prescript{1}{}{G}_{k}^{\ell p}(u_k, u_{\ell}, u_p) &= \frac{\partial \tilde{h}_{k \ell}(u_k, u_{\ell})}{\partial u_{k}} h(u_{k}, u_{p}) + \frac{\partial \tilde{h}_{k \ell}(u_k, u_{\ell})}{\partial \bar{u}_{k}} \overline{h(u_{k}, u_p)} \, \\ 
\prescript{2}{}{G}_{k}^{\ell p}(u_k, u_{\ell}, u_p) &= \frac{\partial \tilde{h}_{k \ell}(u_k, u_{\ell})}{\partial u_{\ell}} h(u_{\ell}, u_{p}) + \frac{\partial \tilde{h}_{k \ell}(u_k, u_{\ell})}{\partial \bar{u}_{\ell}} \overline{h(u_{\ell}, u_p)}\, .
\end{split}
\end{align}
\end{linenomath}

In Eq. (\ref{g1gen}) a term of degree $d$ in $h$ and a term of degree $\tilde{d}$ in $ \tilde{h}_{k \ell}$ combine to form a term of degree $d+\tilde{d} - 1$ in $\prescript{1}{}{G}_{k}^{\ell p}$. As both $h$ and  $\tilde{h}_{k \ell}$  have terms of degree 2 and higher,  we see that $\prescript{1}{}{G}_{k}^{\ell p}$ only has terms of degree 3 and higher. The same holds true for $\prescript{2}{}{G}_{k}^{\ell p}$, which means that a classical network description involving directed edges is  no longer possible. 

The third order terms are moreover easily found by replacing $h$ and  $\tilde{h}_{k \ell}$ in Eq. (\ref{g1gen}) by their quadratic terms. Likewise, the fourth order terms are found by replacing  $h$ by its quadratic terms and $\tilde{h}_{k \ell}$ by its cubic terms and vice versa in Eq.  (\ref{g1gen}). We may also argue that these higher order terms in $\prescript{1}{}{G}_{k}^{\ell p}$ and $\prescript{2}{}{G}_{k}^{\ell p}$ are non-vanishing in general. Indeed, the coefficients in front of these terms are rational functions of $\gamma_k$ and the coefficients of $h$. Such functions are either identical to the zero function (which Eq.  (\ref{g1gen}) excludes) or non-vanishing on an open dense set.

New terms emerge that have an interpretation as higher-order interactions. The two double sums in Eq. (\ref{Eq1-transf-1}) have a combinatorial interpretation. The first double sum counts all pairs of nodes $(\ell, p)$ that both influenced node $k$ in the original network. The second double sum counts all pairs $(\ell, p)$ where $\ell$ influenced $k$ and $p$ influenced $\ell$ and $p$ need not influence $k$ directly in the old network, so that new node-dependency is formed. 

\subsection*{An explicit algorithm for predicting the emergent hypernetwork}

We present an algorithm for obtaining an emergent hypernetwork from a given network system. Its input consists of the adjacency matrix $A$, the function $h$ and the phases $\omega_1$ through $\omega_n$, and we assume the nonresonance conditions of the theorem to hold.  The algorithm is as follows:
    \begin{algorithm}
    \caption{Emergent Hypernetworks}\label{euclid}
    \hspace*{\algorithmicindent} \textbf{Input}: Adjacency matrix $A$, coupling function $h$, frequencies and amplitudes $\gamma_i$'s \\
    \hspace*{\algorithmicindent} \textbf{Output}: Hypernetwork and Coupling functions
\begin{algorithmic}[1]
\ForEach {$k \in \mathcal S $}
\ForEach {$\ell \in \mathcal S $}
    \If {$A_{k\ell}\not= 0$} 
    \State form the polynomials $\tilde h_{k\ell}(u_k,u_\ell) $ by the replacement rule $$z^{d_1}\bar{z}^{d_2}w^{d_3}\bar{w}^{d_4} \mapsto\frac{z^{d_1}\bar{z}^{d_2}w^{d_3}\bar{w}^{d_4}}{(d_1-1)\gamma_k + d_2\bar{\gamma}_k + d_3\gamma_{\ell} + d_4\bar{\gamma}_{\ell} }\,$$
       
    \ForEach {$p \in \mathcal S $}
    \If {$A_{k\ell} A_{kp} \not= 0$} 
    \State Compute  $^1G_{k}^{\ell p}$
    \EndIf
    \If {$A_{k \ell} A_{\ell p} \not= 0$} 
    \State Compute  $^2 G_{k}^{\ell p}$
    \EndIf
\EndFor
 \EndIf
\EndFor
\Procedure{Resonant terms in the coupling functions $G$}{}
   \ForEach {$u_k^{d_1}\bar{u}_k^{d_2}u_{\ell}^{d_3}\bar{u}_{\ell}^{d_4}u_p^{d_5}\bar{u}_p^{d_6}$ monomial of $^1 G_{k}^{\ell p}$ and  $^2 G_{k}^{\ell p}$} 
     \If {$(d_1 - d_2 - 1)\omega_k + (d_3 - d_4)\omega_{\ell} + (d_5 - d_6)\omega_{p} \not=0$} 
     \State discard term
\EndIf
\EndFor
 \EndProcedure 
 
\Procedure{Remaining monomials are the couplings of node $k$}{}
 \EndProcedure 
\EndFor

\end{algorithmic}
\end{algorithm}

\section*{Data Availability}

We provide the experimental time-series and the extracted phases of the oscillations (Fig.~\ref{Fig1}) at \cite{nijholt_data}.

\section*{Code availability}

The source code for reconstructing the functions representing hypernetwork dynamics from oscillatory networks dynamics is available  \cite{nijholt_eddie_2021_5749164}.

\bibliographystyle{naturemag}

\bibliography{references}

\section*{Acknowledgments}
We thank  Sajjad Bakrani, Zachary G. Nicolaou, Marcel Novaes, Edmilson Roque, Robert Ronge and Jeroen Lamb for enlightening discussions. TP was  supported in part by FAPESP Cemeai Grant No. 2013/07375-0 and is a Newton Advanced Fellow of the Royal Society NAF$\backslash$R1$\backslash$180236. TP and EN 
were partially supported by Serrapilheira Institute (Grant No. Serra-1709-16124). DE was supported by TUBITAK Grant No. 118C236 and the BAGEP Award of the Science Academy. JLO-E acknowledges financial support from CONACYT. IZK acknowledges support from National Science Foundation (grant CHE-1900011).

\section*{Author Contributions Statement}
EN and TP designed the overall study and formulated the theory. JLO-E and IZK designed and performed the experiments. DE implemented the numerical simulations and analyses. All authors contributed to the writing of the manuscript. All authors reviewed and approved the final manuscript.

\section*{Competing Interests Statement}
The authors declare no competing interests.

\bigskip\noindent{\bf\large List of supplementary materials}

\noindent
Supplementary Text\\

\clearpage

\clearpage
\baselineskip18pt

\setcounter{page}{1}

\renewcommand{\theequation}{S\arabic{equation}}

\begin{center}
{\large \bf Supplementary Information to \\
Emergent hypernetworks in weakly coupled oscillators} \\
\vspace{1cm}
{ Eddie Nijholt$^{1}$, Jorge Luis  Ocampo-Espindola$^{2}$,  Deniz Eroglu$^{3}$, \\ Istv\'{a}n Z. Kiss$^{2}$,  Tiago Pereira$^{1,4\ast}$\\~\\

{

$^1$Instituto de Ci\^encias Matem\'aticas e Computa\c{c}\~ao, Universidade de S\~ao Paulo, S\~ao Carlos, Brazil \\
\vspace{0.2cm}
$^{2}$Department of Chemistry, Saint Louis University, St. Louis, USA} \\
\vspace{0.2cm}
$^3$ Faculty of Engineering and Natural Sciences, Kadir Has University, Istanbul, Turkey \\
\vspace{0.2cm}
$^4$ Department of Mathematics, Imperial College London, SW7 2AZ, London, United Kingdom}
\end{center}

\renewcommand\contentsname{Supplementary Notes}
\tableofcontents

\newpage

\noindent{\huge \bf Supplementary Note} 
\section{Experimental setup and methods}\label{Sec_exp}
In this section, we describe the details about the experimental setup, the dynamical behavior of the oscillators without coupling, the 
of the phase model using LASSO.

\subsection{Experimental setup} 
The experiments were carried out in a standard three-electrode electrochemical cell. 

The cell consists of a nickel-array working electrode (W), a Pt-coated Ti rod as a counter electrode (C), and a Hg/Hg$_2$SO$_4$ sat.~K$_2$SO$_4$ as a reference electrode (R). The electrolyte was a 3.0 M sulfuric acid solution at a constant temperature of 10 $^{\circ}$C. The electrode array consisted of four 1-mm diameter nickel wires embedded in epoxy with a spacing of 3 mm. A multichannel potentiostat (Gill-IK64, ACM Instruments) interfaced with a real-time LabVIEW controller measured the potential drop $[E_k(t)$, with 
respect to the reference electrode] and set the circuit potential ($V_{0, k}$) of the working electrodes individually at a rate of 200 Hz. The electrode potentials are corrected for offset $o_1$=0.92 V, $o_2$=0.98 V, $o_3$= 0.91 V, and $o_4$=0.97 V. 
 
\subsection{Behavior without coupling} The offset circuit potential to each oscillator was established 20 mV above the Hopf bifurcation ($V_{0, 1}$=1850 mV, $V_{0, 2}$=1100 mV, $V_{0, 3}$=1660 mV, $V_{0, 4}$=1103 mV). The natural frequencies [Supplementary Fig.~\ref{fig:exp_set}~(a)] were adjusted to have values of $\omega_1$=0.152 Hz, $\omega_2$=0.385 Hz, $\omega_3$=0.237 Hz and $\omega_4$=0.384 Hz with a set of resistors and capacitors $R_{ind, 1}$=12.0 kohm, $C_{ind}$=440 $\mu$F, $R_{ind, 2}$=1.00 kohm, $R_{ind, 3}$=12.0 kohm, $R_{ind, 4}$=1.00 kohm. Without coupling, we observed that the slow oscillators (1 and 3) have about twice the amplitude than the fast oscillators (2 and 4). Supplementary Fig.~\ref{fig:exp_set}~(c) shows the electrode potential time series of each oscillator.

\subsection{Phase dynamics}

\textbf{Phase definition.}We used the peak-finding approach \cite{pikovsky2003synchronization} to extract the phase of each oscillator and then linear interpolation between peaks from the experimental electrode potential time series. When there is no coupling, the pairwise phase difference shows a linear growth [Supplementary Fig..~\ref{fig:exp_set}~(b)] and the triplet phase differences, $\phi_j$, $j$=1, 2, do not show phase slip behavior [Supplementary Fig.~\ref{fig:exp_set}~(d)].  ~\\

\textbf{Fitting of phase dynamics} As described in the main text, the impact of triplet interactions on the dynamics can be extracted with 
a LASSO fit to the $\dot{\theta}_k$ values according to
\begin{equation}\label{eq:exp}
\dot{\theta}_k = \hat{\omega}_k(t) +   \sum_{j=1}^{2} C_j^k \sin (\phi_j) 
+ D_j^k \cos (\phi_j)
\end{equation}
where $\hat{\omega}_k(t) = \hat{\omega}^0_k + \hat{\omega}^1_k t + \hat{\omega}^2_k t^2 $  is the fitted, slowly drifting (up to quadratic variation in time) natural frequency, and $C_j^k$ and $D_j^k$ are the amplitudes  of the sin and cos phase  coupling functions corresponding to the appropriate triplet phase differences. 

For the fit, the instantaneous frequency, $\dot{\theta}_k$ was obtained with the numerical derivative of the phase of each oscillator from the experimental times series. The $\dot{\theta}_k$ was filtered by a first order Savitzky-Golay filter for 45~s. Using $\phi_j$, we fitted the $\dot{\theta}_k$ with LASSO method. In LASSO, the regularization parameter determines how many parameters in the fitted model should be set to zero. We used a regularization parameter so that the mean square error is 20\% higher than the best fit (no regularization). The fitted parameters are: 

~\\
\begin{center}
\begin{table}
\caption{Recovered coefficients from Eq.~\ref{eq:exp}.}
\begin{tabular}{ c|cccc}
\hline
 Coefficients & \multicolumn{4}{|c}{Oscillator number} \\
 & 1 & 2 & 3 & 4 \\
\hline
\hline
$\hat{\omega}_k^0$ & \,\,\,\,0.953 & \,\,\,\,2.368 & \,\,\,\,1.467 & \,\,\,\,2.383 \\
$\hat{\omega}_k^1$ & \,\,\,\,$2.76\times 10^{-5} $ & 0 & $-4.61\times 10^{-5}$ & 0 \\
$\hat{\omega}_k^2$ & $-6.40\times 10^{-8}$ & $-6.63\times 10^{-8}$ & $-1.32\times 10^{-7}$ & 0 \\
$C_1^k$ & \,\,\,\,$4.89\times 10^{-3}$ & \,\,\,\,$5.04\times 10^{-3}$ & $\,\,\,\,3.15\times 10^{-3}$ & 0 \\
$D_1^k$ & \,\,\,\,$9.27\times 10^{-5}$ & \,\,\,\,$1.23\times 10^{-2}$ & $-3.10\times 10^{-3}$ & 0 \\
$C_2^k$ & $-1.49\times 10^{-3}$ & 0 & \,\,\,\,$3.16\times 10^{-3}$ & \,\,\,\,$4.63\times 10^{-3}$  \\
$D_2^k$ & $-1.73\times 10^{-3}$ & 0 & $-7.37\times 10^{-4}$ & \,\,\,\,$1.64\times 10^{-2}$ \\
\hline
\end{tabular}
\end{table}
\end{center}

~\\

\begin{figure}[!ht]
    \centering
    \includegraphics[width=1\columnwidth]{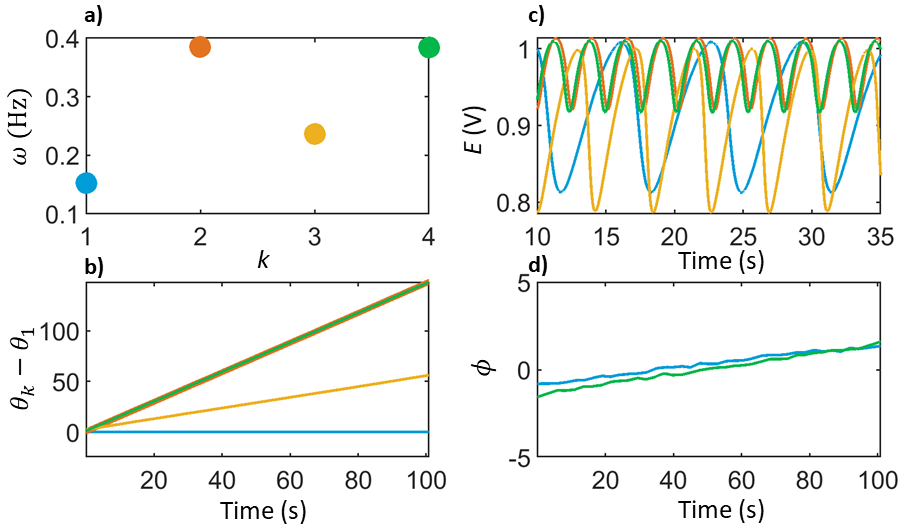}
    \caption{{\bf Dynamical behavior without coupling in the experiments.}  a) Natural frequency of each oscillators without coupling $\omega_1$=0.152 Hz, $\omega_2$=0.385 Hz, $\omega_3$=0.237 Hz and $\omega_4$=0.384 Hz.  The blue, orange, yellow and green dots represent the elements 1 to 4 respectively. b) Time series of the phase difference with respect oscillator one. The blue line: $\theta_1$-$\theta_1 = 0$, orange line:  $\theta_2$-$\theta_1$, yellow line:  $\theta_3$-$\theta_1$ and green line:  $\theta_4$-$\theta_1$. c) Electrode potential time series. Blue, orange, yellow, and green line corresponds to oscillator one to four respectively. d)  Time series of the slow phases, $\phi_1$ (blue) and $\phi_2$ (green) without coupling.}
    \label{fig:exp_set}
\end{figure}

\newpage

\section{Proof of emergent higher-order networks}
We consider ODEs of the general form
\begin{align}\label{theODE00}
\dot{z}_k &= \gamma_kz_k - \beta_k z_k|z_k|^2 + \alpha H_k(z_1, \dots, z_n) \, ,
\end{align}
for $k \in \{1, \dots, n\}$. Here, each $z_k$ takes values in $\mathbb{C}$ and $\alpha \in \mathbb{R}$ denotes the coupling parameter of the interaction.  We moreover have $\beta_k, \gamma_k \in \mathbb{C}$ non-zero, and write $\gamma_k = \lambda + i\omega_k$ for all $k \in \{1, \dots, n\}$. Note that $\lambda \in \mathbb{R}$ may be seen as the bifurcation parameter for a Hopf bifurcation, which might in particular vanish. Each interaction function $H_k: \mathbb{C}^n \rightarrow \mathbb{C}$ is assumed smooth (i.e. $C^{\infty}$) for convenience, and satisfies $H_k(0) = 0$ and $DH_k(0) = 0$.\\

We will show that the ODE \eqref{theODE00} can be put in a particular normal form that allows us to predict the dynamics of the phases of the oscillators. Our technique for doing so involves two successive coordinate transformations:
\begin{align}
  w_k &= z_k - \alpha P_k(z) \\ \nonumber
  u_k &= w_k - \alpha Q_k(w)\,
\end{align}

\noindent
for some appropriately chosen polynomials $P_k$ and $Q_k$. The first of these transformations is used to remove the term $\alpha H_k(z)$ from the ODE \eqref{theODE00}. This will generate additional terms in $\alpha^2$ that may be expressed in the coefficients of $H_k$ and $P_k$ following certain combinatorial rules. In order to describe this combinatorial behavior, we first introduce a useful bracket $[\bullet|| \bullet]$ on the space of polynomials, see Definition \ref{defibracket}. The first coordinate transformation will also produce terms in $\alpha$ involving $P_k$ and $\beta_k z_k|z_k|^2$. Again our bracket allows for a precise description of these new terms, which we then remove using the second coordinate transformation. The precise bookkeeping enabled by the bracket will be crucial in determining what non-resonance conditions are needed for the second transformation. In fact, it will turn out that the non-resonance conditions needed for the first transformation are sufficient to ensure the second.\\

We first present the main result, Theorem \ref{main001}, in Subsection \ref{Statementofresultsandpreliminaries}. There we also develop the necessary definitions, notation and machinery needed for the proof, which is then presented in the remaining subsections.

\subsection{Preliminaries and results}\label{Statementofresultsandpreliminaries}
In order to analyse the ODE \eqref{theODE00}, it will be useful to write 

\[
H_k(z) =  {H}^d_k(z) +  \mathcal{O}(|z|^{d+1})\, ,
\]
where ${H}^d_k(z)$ is a polynomial denoting the terms up to degree $d$ in the Taylor expansion of $H_k(z)$ around the origin. We will mostly work with the value $d=5$. Note that ${H}^d_k(z)$ is therefore a polynomial in both the variables $z_1, \dots, z_n$ and their complex conjugates $\overline{z}_1, \dots, \overline{z}_n$, with complex coefficients. In general, whenever we talk about a polynomial we will always mean a complex polynomial in its given complex variables and their complex conjugates. It will also be useful to write $H(z) = (H_k(z))$ for the vector valued function that captures all interaction functions $H_k(z)$ as its components, and similarly set $H^d(z) = (H^d_k(z))$.

 As is often the case with normal form calculations, we will need to assume some conditions on the $\omega_k$ (or more precisely the $\gamma_k$). These will depend on the coefficients of ${H}^5_k(z)$. To this end, we define:

\begin{defi}\label{nonreso}
Let
\begin{align}
    R(z) = c z_1^{s_1} \dots  z_n^{s_n} \overline{z}_1^{t_1} \dots \overline{z}_n^{t_n}\, 
\end{align}
be a monomial term in ${H}^d_k(z)$, where $c$ is a complex number and $s_1, \dots, s_n$, $t_1, \dots, t_n$ are non-negative integers. The $k$th \emph{non-resonance condition} of $R(z)$ is the condition

\begin{align}\label{non-resonance}
    s_1\omega_1 + \dots + s_n\omega_n - t_1\omega_1 - \dots - t_n\omega_n - \omega_k \not= 0\, .
\end{align}

Note that, as $\omega_{\ell}$ denotes the imaginary part of $\gamma_{\ell}$ for all $\ell \in \{1, \dots, n\}$, the $k$th non-resonance condition guarantees in particular that:

\begin{align}\label{non-resonance2}
    s_1\gamma_1 + \dots + s_n\gamma_n + t_1\overline{\gamma}_1 + \dots + t_n\overline{\gamma}_n - \gamma_k \not= 0\, ,
\end{align}

\noindent
which will play a role in much of the proofs and constructions below. In fact, varying $\lambda$ and allowing the particular case $\lambda = 0$, we see that equations \eqref{non-resonance} and \eqref{non-resonance2} are equivalent in general. Next, the $k$th \emph{non-resonance condition} of a polynomial is defined as the union of the $k$th non-resonance conditions of all of its monomial terms. 
Finally, the \emph{non-resonance condition} of a polynomial map ${H}^d(z) = ({H}^d_k(z))$ is the union over all $k \in \{1, \dots, n\}$ of the $k$th non-resonance conditions of ${H}^d_k(z)$.  \hfill $\triangle$
\end{defi}

\begin{ex}\label{example0easyinters1}
Suppose the interaction functions are given by the polynomials
\begin{equation}
    H_k(z) = \sum_{\ell=1}^n c_{k, \ell}z_k\overline{z}_\ell\, ,
\end{equation}
for some (possibly weighted) connection matrix $c = (c_{k,\ell}) \in\mathbb{C}^{n \times n}$. It follows that the $k$th non-resonance condition of $H_k(z)$ is given by
\begin{equation}
    \omega_k - {\omega}_{\ell} - \omega_k = -{\omega}_{\ell} \not= 0 \quad \text{ for all } {\ell} \text{ such that } c_{k, \ell} \not= 0 \, .
\end{equation}
Hence, we see that the non-resonance condition of $H(z) = (H_k(z))$ is satisfied if we simply have $\omega_{\ell} \not= 0$ for all nodes ${\ell} \in \{1, \dots, n\}$. \hfill $\triangle$
\end{ex}

\begin{ex}\label{example0easyinters12}
Suppose the interaction functions are given by the polynomials
\begin{equation}
    H_k(z) = \sum_{\ell=1}^n c_{k, \ell}(z_k\overline{z}_\ell + z_k^2\overline{z}_{\ell})\, ,
\end{equation}
for some connection matrix $c = (c_{k, \ell})$. The $k$th non-resonance condition of $H_k(z)$ is now given by
\begin{align}
    \omega_k - \omega_{\ell} - \omega_k &= -\omega_{\ell} \not= 0 \quad \text { and } \\ \nonumber
   2\omega_k -\omega_{\ell} - \omega_k &= \omega_k - \omega_\ell \not= 0 
\end{align}
for all $\ell$ such that $c_{k,l} \not= 0$. If we assume for convenience that $c$ encodes a symmetric, connected graph, then the non-resonance condition of $H(z)$ is satisfied if
\begin{align}
    \omega_{\ell} &\not= 0 \text{ for all nodes } \ell \in \{1, \dots, n\} \text{ and } \\ \nonumber
    \omega_{p} - \omega_q &\not= 0  \text{ for all edges } e = [p,q] \text{ between nodes } p \text{ and } q\, .
\end{align} \hfill $\triangle$
\end{ex}

\noindent We are now ready to formulate the main theorem. It tells us that, under the  relevant non-resonance conditions, we may transform the ODE \eqref{theODE00} into a system with leading interaction terms involving only three-way ``hyper-interactions'' and with coupling constant $\alpha^2$.  See Proposition \ref{newintterms} for an exact description of the new leading interaction terms in Theorem \ref{main001}.

\begin{thr}\label{main001}
Let ${H}^5_k$ denote the fifth order Taylor expansion of the $k$th interaction function $H_k: \mathbb{C}^n \rightarrow \mathbb{C}$. Assume the non-resonance conditions for ${H}^5 = ({H}^5_k)$ to hold. Then the ODE \eqref{theODE00} is locally conjugate to 
\begin{align}\label{transformedODE001}
\dot{u}_k &= \gamma_ku_k - \beta_k u_k|u_k|^2 - \alpha^2G_k(u) \\ \nonumber
&+ \mathcal{O}(|\alpha||u|^6 +|\alpha|^2|u|^5 + |\alpha|^3|u|^4 )\, ,
\end{align}
with $u_k \in \mathbb{C}$ and for some complex polynomials $G_k$ with only terms of degree $3$ and higher. See Proposition \ref{newintterms} for an exact description of the $G_k$.
\end{thr}

\noindent Note that Equation \eqref{transformedODE001} gives a precise description of $\dot{u}_k$ up to sixth order in $u$ and $\alpha$.

\begin{remk}
Theorem \ref{main001} tells us that, under the relevant non-resonance conditions, the $\alpha$-terms can be `pushed back' to order $|\alpha| |u|^6$. Technically speaking, this means we can ignore these remainder terms only when $\alpha u^6 << \alpha^2 u^4$, i.e. when $ u^2 << \alpha$. However, it will be clear from the proof in subsections \ref{thefirstreduction} and  \ref{thesecondreduction} that if the non-resonance conditions hold for terms in $H_k$ beyond ${H}_k^5$, then we may  arrange for a remainder in $\alpha$ with higher order terms in $u$. That is, we then get the remainder $\mathcal{O}(|\alpha||u|^N +|\alpha|^2|u|^5 + |\alpha|^3|u|^4 )$ for a corresponding value of $N > 6$. As the interaction functions $H_k$ are taken to be polynomials in our examples, with the non-resonance conditions holding for all terms, we in fact expect a remainder of the form $\mathcal{O}(|\alpha||u|^N +|\alpha|^2|u|^5 + |\alpha|^3|u|^4 )$ for arbitrarily high value of $N$. Hence, we may neglect all terms that are first order in $\alpha$, and obtain a new coupled system with coupling constant $\alpha^2$. \hfill $\triangle$
\end{remk}

 The new interaction functions $G_k$ can be obtained from $H$ by a combinatorial construction on the Taylor coefficients. This is best described using a bracket on polynomials that we define below. We will furthermore make extensive use of this bracket throughout the proof of Theorem \ref{main001}. 

\begin{defi}\label{defibracket}
Let $R(z)$ be a complex polynomial and let $S(z) = (S_1(z), \dots, S_n(z))$ be an $n$-tuple (i.e., a vector) of complex polynomials $S_1(z), \dots, S_n(z)$. We let $[R||S](z)$ be the complex polynomial obtained by (formally) assuming each variable $z_j$ is time-dependent (i.e., $z_j = z_j(t)$) and satisfies $\dot{z}_j =   S_j(z)$, after which we differentiate $R(z)$ with respect to $t$. That is, we set
\begin{align}
    [R||S](z) :=  \frac{d}{dt}R(z)\left|{\begin{array}{l}
  \dot{z}_j =   S_j(w) \\
   \forall \, j \in \\
   \{1, \dots, n\} 
\end{array}}\right.\, .
\end{align}
The reason we choose this notation, instead of one involving the Jacobian of $R$, is to avoid confusion with the complex conjugate variables $\overline{z}_j$. Because each term $\dot{\overline{z}}_j$ is substituted by    $\overline{S_j(w)}$, the expression $[R||S]$ is in general not complex linear in $S$. \\

 In the special case where $S(z) = (\gamma_1z_1, \dots, \gamma_nz_n)$ (with $\gamma_k$ as in Equation \eqref{theODE00}), we set
\begin{align}
    \Gamma R(z) := [R||S](z) = \frac{d}{dt}R(z)\left|{\begin{array}{l}
  \dot{z}_j =   \gamma_jz_j \\
   \forall \, j \in \\
   \{1, \dots, n\} 
\end{array}}\right.\, .
\end{align}
\hfill $\triangle$
\end{defi}

\begin{ex}\label{exgamma}
It is not hard to see that the polynomial $\Gamma R(z)$ is obtained by replacing every monomial 
\[c z_1^{s_1}\dots  z_n^{s_n} \overline{z}_1^{t_1}  \dots  \overline{z}_n^{t_n}\, ,\quad c \in \mathbb{C}\] 
in $R(z)$ by
\[c(s_1\gamma_1 +\dots + s_n\gamma_n + t_1\overline{\gamma}_1 + \dots + t_n\overline{\gamma}_n)z_1^{s_1}\dots  z_n^{s_n} \overline{z}_1^{t_1}  \dots  \overline{z}_n^{t_n}\, .\]
Consider for instance the monomial  $R(z) = R(z_1, z_2,\overline{z}_1, \overline{z}_2) = z_1^2\overline{z}_2$. We have
\begin{align}
\frac{d}{dt}R(z) = 2z_1\dot{z}_1\overline{z}_2 + z_1^2\dot{\overline{z}}_2  \, .
\end{align}
Hence, we indeed find
 \begin{align}
\Gamma R(z) = \frac{d}{dt}R(z)\left|{\begin{array}{l}
  \dot{z}_1 =  \gamma_1z_1 \\
   \dot{z}_2 =  \gamma_2z_2
\end{array}}\right. &= 2z_1(\gamma_1z_1)\overline{z}_2 + z_1^2(\overline{\gamma_2z_2}) \\ \nonumber
&= (2\gamma_1 + \overline{\gamma}_2)z_1^2\overline{z}_2 =(2\gamma_1 + \overline{\gamma}_2)R(z)\, .
\end{align}\hfill $\triangle$
\end{ex}

The term $G_k$ in Theorem \ref{main001} will be given as the bracket $[\bullet || \bullet]$ between ${H}^5$ and a polynomial obtained by slightly modifying ${H}^5_k$. Hence, intuitively,  $G_k$ should be thought of as $[H_k|| H]$. More precisely, we define:

\begin{defi}\label{augmentedd}
Let $P$ be a polynomial for which its $k$th non-resonance conditions are met. The ($k$th) \emph{modified polynomial} $\widehat{P}_{k}$ is obtained from $P$ by replacing each monomial 
\[c z_1^{s_1}\dots  z_n^{s_n} \overline{z}_1^{t_1}  \dots  \overline{z}_n^{t_n}\, ,\quad c \in \mathbb{C}\] 
in $P$ by
\[\frac{cz_1^{s_1}\dots  z_n^{s_n} \overline{z}_1^{t_1}  \dots  \overline{z}_n^{t_n}}{s_1\gamma_1 +\dots + s_n\gamma_n + t_1\overline{\gamma}_1 + \dots + t_n\overline{\gamma}_n - \gamma_k}\, .\]
In the special case where $P = {H}^d_k$ for some $d \leq 5$, we will simply write $\widehat{H}^d_{k} := \widehat{(H^d_k)}_{k}$ to denote the corresponding modified polynomial. \hfill $\triangle$
\end{defi}

\begin{prop}\label{newintterms}
In Theorem \ref{main001} the terms $G_k$ are given by
\begin{equation}
G_k = [\widehat{H}^3_{k}||{H}^3]\, .
\end{equation}
\end{prop}

It can be shown that $[\widehat{H}^3_{k}||{H}^3]$ indeed only has terms of degree $3$ and higher, using the assumption that each $H_{k}$ (and therefore each $\widehat{H}^d_{k}$) only has terms of degree 2 and higher. See Remark \ref{remkaboutdegrees} below. Note that we only care about the third and fourth order terms of $G_k$, as the rest are absorbed in the remainder terms of Equation \eqref{transformedODE001}. It will be clear from Remark \ref{remkaboutdegrees} that these lowest order terms do not change if we instead define
\begin{equation}
G_k = [\widehat{H}^d_{k}||{H}^d]\, ,
\end{equation}
for $d=4$ or $d=5$. For this reason we will often simply write
\begin{equation}
G_k = [\widehat{H}_{k}||{H}]\, . 
\end{equation}

\begin{remk}\label{I2}
Let $I_k \subset \{1, \dots, n\}$ denote the \emph{input set} of a node $k \in \{1, \dots, n\}$. That is, $I_k$ denotes those nodes that influence $k$, or more precisely those nodes $\ell$ for which $$\frac{\partial H_k(z)}{\partial z_{\ell}} \not= 0 \text{ or } \frac{\partial H_k(z)}{\partial \overline{z}_{\ell}} \not= 0 \, .$$ Note that $I_k$ might not contain $k$ itself. It follows that in general $G_k(u)$ depends on variables for nodes in the set
\begin{equation}
I^2_k := \left( \bigcup_{{\ell} \in I_k} I_{\ell}\right) \cup I_k\, . 
\end{equation}
This is because $G_k = [\widehat{H}^3_{k}||{H}^3]$ is formed by replacing a term $z_{\ell}$ (or $\overline{z}_{\ell}$) in $\widehat{H}^3_{k}$ by ${H}^3_{\ell}$ (or $\overline{{H}^3_{\ell}}$), and this is done for each $\ell \in I_k$. We have also used here that $\widehat{H}^3_{k}$ likewise only depends on the variables corresponding to nodes in $I_k$, or possibly a strict subset thereof. 

In a similar way one sees that the third order terms of $G_k$ (that is, its leading order terms) are given by `triplet terms' $\tilde{u}_r\tilde{u}_s\tilde{u}_t$, 
where we have $r \in I_k$ and $s,t \in I_{\ell}$ for some $\ell \in I_k$. (Here each $\tilde{u}_p$ may independently denote $u_p$ or its complex conjugate $\overline{u}_p$). See Figure \ref{fig1} for a schematic depiction of these emergent triplet terms.
Of course the specifics of $H_k(z)$ might put additional constraints on the terms that can show up in $G_k(u)$. \hfill $\triangle$
\end{remk}

\begin{figure}[h]
\centering
\begin{tikzpicture}

	\node[rectangle,draw=black, fill=blue, fill opacity = 0.1, rounded corners, minimum height=4.7cm, minimum width=7.4cm] (1) at (0,0.3) {};
	\node[rectangle,draw=black, fill=red, fill opacity = 0.1, rounded corners, minimum height=2cm, minimum width=3.2cm] (1) at (-1.5,0) {};
	\node[rectangle,draw=black, fill=yellow, fill opacity = 0.1, rounded corners, minimum height=2cm, minimum width=4cm] (1) at (1.1,-0.7) {};
	\fill[fill=blue!45](-2.2,0.5)--(2.4,-0.2)--(0.4,-1.2)--(-2.2,0.5);
	 \path[draw] (-2.2,0.5)--(2.4,-0.2)--(0.4,-1.2)--(-2.2,0.5);
	\node[circle,draw=black, fill=white, fill opacity = 1, inner sep=1.5pt, minimum size=14pt] (i) at (0.4,2.1) {k};
	\node[circle,draw=black, fill=white, fill opacity = 1, inner sep=1.5pt, minimum size=14pt] (k) at (-2.2,0.5) {r};
	\node[circle,draw=black, fill=white, fill opacity = 1, inner sep=1.5pt, minimum size=14pt] (j) at (-1.6,-0.5) {$\ell$};
	\node[circle,draw=black, fill=white, fill opacity = 1, inner sep=1.5pt, minimum size=14pt] (l) at (2.4,-0.2) {t};
	\node[circle,draw=black, fill=white, fill opacity = 1, inner sep=1.5pt, minimum size=14pt] (m) at (0.4,-1.2) {s};
	\node[] (Ii) at (-2.85,-0.73) {$I_k$};
	\node[] (Ij) at (-0.63,-1.49) {$I_{\ell}$};
	\draw [->,  >=stealth, thick, black, shorten <=2pt, shorten >=2pt] (k) to [bend right = -25] (i);
	\draw [->,  >=stealth, thick, black, shorten <=2pt, shorten >=2pt] (j) to [bend right = -25] (i);
	\draw [->,  >=stealth, thick, black, shorten <=2pt, shorten >=2pt] (l) to [bend right = -10] (j);
	\draw [->,  >=stealth, thick, black, shorten <=2pt, shorten >=2pt] (m) to [bend right = -10] (j);

\draw[fill=blue!45, rotate=-8.7, shift={(-3.6,-3.2)}] (2.8,3) arc (-90:0:0.7cm) -- (3.5, 4.5)--(3.2, 4.5)--(3.7, 5)--(4.2, 4.5)-- (3.9, 4.5) -- (3.9, 3.7) arc (180:270:0.7cm) ;
\end{tikzpicture}
\caption{Schematic depiction of the hidden `triplet terms' that show up in $G_k$.}
\label{fig1}
\end{figure}
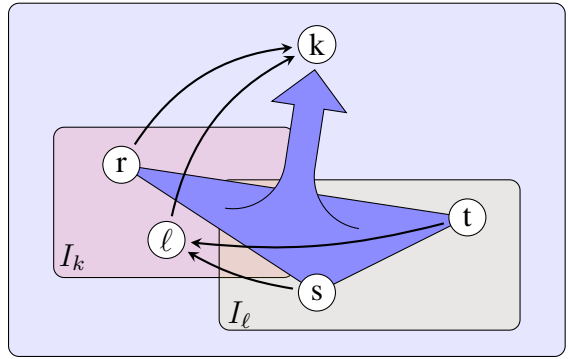

\begin{ex}\label{examzi11}
As in Example \ref{example0easyinters1}, let us make the particular choice for the interaction functions
\begin{equation}
H_k(z) = {H}_k^3(z) = \sum_{\ell = 1}^nc_{k,\ell}z_k\overline{z}_{\ell}\, .
\end{equation}
Here $c = (c_{k,\ell})$ is the connectivity matrix of the network. We will assume that $c_{k,\ell} \in \{0,1\}$, or more generally $c_{k,\ell} \in \mathbb{R}$, if the network is weighted.  We find
\begin{equation}
\widehat{H}_{k}(z) = \widehat{H}^3_{k}(z) =  \sum_{\ell = 1}^n\frac{c_{k,\ell}}{\overline{\gamma}_{\ell}}z_k\overline{z}_{\ell} \, .
\end{equation}
A direct calculation now shows that
\begin{align}
G_k(u) = [\widehat{H}_{k}|| H](u) &= \sum_{\ell= 1}^n\frac{c_{k,\ell}}{\overline{\gamma}_{\ell}}(H_k(u)\overline{u}_{\ell} + u_k\overline{H_\ell(u)}) \\ \nonumber
&= \sum_{\ell = 1}^n\frac{c_{k, \ell}}{\overline{\gamma}_{\ell}}\left(\sum_{p = 1}^n c_{k,p} u_k \overline{u}_p\overline{u}_{\ell} + u_k\sum_{p=1}^nc_{\ell,p}\overline{u}_{\ell}u_p\right) \\ \nonumber
&= \sum_{\ell = 1}^n \sum_{p = 1}^n\frac{c_{k,\ell}c_{k,p}}{\overline{\gamma}_{\ell}} u_k \overline{u}_{\ell}\overline{u}_p + \sum_{\ell = 1}^n \sum_{p = 1}^n\frac{c_{k,\ell}c_{\ell,p}}{\overline{\gamma}_\ell} u_k\overline{u}_{\ell}u_p \, .
\end{align}
The first of these two terms counts all trees in the network that are formed by the node $k$ and two nodes that influence node $k$. The second term counts all trees formed by the nodes $k$, $\ell$ and $p$, where $k$ depends on $\ell$ and $\ell$ depends on $p$.  \hfill $\triangle$
\end{ex}

We will gather some more facts about the bracket $[\bullet || \bullet]$. One important observation is given by:

\begin{lem}\label{useful0}
The map $(R,S) \mapsto [R||S]$ is complex linear in $R$ and real linear in $S$. In other words, given polynomials $R, R'$ and complex numbers $\lambda, \mu$, we have
\begin{equation}
    [\lambda R + \mu R'||S] = \lambda [R||S] + \mu [R'||S] \, .
\end{equation}
Given tuples $S, S'$ and real numbers $a,b$ we have
\begin{equation}
    [R||aS + bS'] = a [R||S] + b [R||S'] \, .
\end{equation}
\end{lem}

\begin{proof}
Complex linearity in $R$ is clear from the definition:
\begin{align}
    [R||S](z) :=  \frac{d}{dt}R(z)\left|{\begin{array}{l}
  \dot{z}_j =   S_j(w) \\
   \forall \, j \in \\
   \{1, \dots, n\} 
\end{array}}\right.\, .
\end{align}
Because of this, it suffices to show real linearity in $S$ when $R$ is given by a monomial of the form
\[R(z) = z_{i_1}z_{i_2}\dots z_{i_m}\overline{z}_{j_1}\overline{z}_{j_2}\dots\overline{z}_{j_l} ,  \]
for some (not necessarily distinct) $i_1, \dots, i_m, j_1, \dots, j_l \in \{1, \dots, n\}$. We get
\begin{align}\label{needlaterrron1}
    [R||S](z) = \sum_{s=1}^m\frac{R(z)}{z_{i_s}}S_{i_s}(z) + \sum_{r=1}^l\frac{R(z)}{\overline{z}_{i_r}}\overline{S_{i_r}(z)}\, ,
\end{align}
from which real linearity in $S$ follows readily. 
\end{proof}
We will also make extensive use of the following definition:

\begin{defi}\label{lowerdegree}
Given a complex monomial 
\[Q(z) = c z_1^{s_1}\dots  z_n^{s_n} \overline{z}_1^{t_1}  \dots  \overline{z}_n^{t_n}\, ,\] 
where $c \in \mathbb{C}$ and with $s_1, \dots, s_n, t_1, \dots, t_n$ non-negative integers, we define the degree of $Q(z)$ as the number $s_1 + \dots + s_n + t_1 + \dots + t_n$. The degree of a polynomial $P(z)$ is then defined as the maximum of the degrees of all the monomial terms of $P(z)$, as is common. Similarly, we define the \emph{lower degree} of a polynomial $P(z)$ as the minimum of the degrees of all of its monomial terms. \hfill $\triangle$
\end{defi}

It turns out our bracket has a predictable effect on degrees:
\begin{lem}\label{useful0b}
Let $R$ be a polynomial and $S= (S_1, \dots, S_n)$ a vector of polynomials. 
If $R$ has degree $p>0$ and each of the polynomial components of $S$ has degree at most $p'$, then $[R||S]$ has degree $p+p'-1$ or lower. If $R$ has lower degree $q>0$ and each of the polynomial components of $S$ has lower degree at least $q'$, then $[R||S]$ has lower degree $q+q'-1$ or higher.
\end{lem}

\begin{proof}
By linearity of the bracket (see Lemma \ref{useful0}), it suffices to show that the degree of $[R||S]$ is $d+d'-1$ if $R$ is a monomial of degree $d$ and the components of $S$ are all monomials of degree $d'$. (By convention, we treat the zero-polynomial as a polynomial of any degree.) As in the proof of Lemma \ref{useful0}, we write
\[R(z) = z_{i_1}z_{i_2}\dots z_{i_m}\overline{z}_{j_1}\overline{z}_{j_2}\dots\overline{z}_{j_l} ,  \]
for some (not necessarily distinct) $i_1, \dots, i_m, j_1, \dots, j_l \in \{1, \dots, n\}$. It follows that $m+l = d$. As in the previous proof we find 
\begin{align}\label{needlaterrron12}
    [R||S](z) = \sum_{s=1}^m\frac{R(z)}{z_{i_s}}S_{i_s}(z) + \sum_{r=1}^l\frac{R(z)}{\overline{z}_{i_r}}\overline{S_{i_r}(z)}\, ,
\end{align}
from which it follows readily that the degree of $[R||S]$ is indeed $d+d'-1$. This completes the proof.
\end{proof}

\begin{remk}\label{remkaboutdegrees}
Lemmas \ref{useful0} and \ref{useful0b} give us an easy way of finding the third and fourth order terms of $G_k = [\widehat{H}_{k}||{H}]$ (and higher terms if necessary). It follows that the third order terms of $G_k$ are given by $[\widehat{H}^2_{k}||{H}^2]$, where we use that $H$ (and therefore each $\widehat{H}^2_{k}$) has no constant and linear terms. Likewise, the fourth order terms of $G_k$ are given by the bracket between the second order terms of $\widehat{H}_{k}$ and the third order terms of $H$, plus the bracket between the third order terms of $\widehat{H}_{k}$ and the second order terms of $H$. Note also that $G_k$ need only be defined up to third and fourth order in Theorem \ref{main001}, as higher order terms of $\alpha^2G_k(u)$ can be absorbed in the remainder $\mathcal{O}(|\alpha|^2|u|^5)$ of Equation \eqref{transformedODE001}. For this reason the terms of degree 4 and higher in $\widehat{H}_{k}$ and ${H}$ play no role in the construction of (the relevant terms of) $G_k = [\widehat{H}_{k}||{H}]$. \hfill $\triangle$
\end{remk}

Next, we consider $G_k$ in the situation where $H_k$ describes a coupled cell system as in the examples of the main text.

\begin{remk}\label{sumsoftermsinH}
Suppose each $H_k$ is of the special form 
\begin{equation}\label{simplelinkform}
H_k(z) = \sum_{\ell=1}^nc_{k,\ell} h(z_{k}, z_{\ell})\, ,
\end{equation}
where $z = (z_1, \dots, z_n)$. Here $(c_{k,\ell})$ is a real adjacency matrix and $h: \mathbb{C}^2 \rightarrow \mathbb{C}$ has vanishing constant and linear terms. It follows that we may write
\begin{equation}
\widehat{H}_k(z) = \sum_{\ell=1}^nc_{k,\ell} \widehat{h}_{k,\ell}(z_{k}, z_{\ell})\, ,
\end{equation}
where $\widehat{h}_{k,\ell}(z_k, z_{\ell})$ is obtained from $h(z_k, z_{\ell})$ by applying a monomial substitution to its terms in precisely the same way $\widehat{H}_k$ is obtained from $H_k$. By linearity of the bracket $[\bullet||\bullet]$ in the first slot, we get
\begin{align}\label{0of1G2G}
G_k = [\widehat{H}_k||{H}] = [\sum_{\ell=1}^nc_{k,\ell} \widehat{h}_{k,\ell}(u_{k}, u_{\ell})||H] = \sum_{\ell=1}^nc_{k,\ell} [\widehat{h}_{k,\ell}(u_{k}, u_{\ell})||H] \, .
\end{align}
Moreover, we find 
\begin{align}\label{1of1G2G}
 [\widehat{h}_{k,\ell}(u_{k}, u_{\ell})||H](u) &= \frac{\partial \widehat{h}_{k,\ell}(u_{k}, u_{\ell})}{\partial u_{k}}H_k + \frac{\partial \widehat{h}_{k,\ell}(u_{k}, u_{\ell})}{\partial \overline{u}_{k}}\overline{H}_k + \frac{\partial \widehat{h}_{k,\ell}(u_{k}, u_{\ell})}{\partial u_{\ell}}H_{\ell} + \frac{\partial \widehat{h}_{k,\ell}(u_{k}, u_{\ell})}{\partial \overline{u}_{\ell}}\overline{H}_{\ell} \\ \nonumber
 &= \frac{\partial \widehat{h}_{k,\ell}(u_{k}, u_{\ell})}{\partial u_{k}}\left(\sum_{p=1}^nc_{k,p} h(u_{k}, u_p)\right) + \frac{\partial \widehat{h}_{k,\ell}(u_{k}, u_{\ell})}{\partial \overline{u}_{k}}\left(\sum_{p=1}^nc_{k,p} \overline{h(u_{k}, u_p)}\right) \\ \nonumber 
 &+ \frac{\partial \widehat{h}_{k,\ell}(u_{k}, u_{\ell})}{\partial u_{\ell}}\left(\sum_{p=1}^nc_{\ell,p} h(u_{\ell}, u_p)\right) + \frac{\partial \widehat{h}_{k,\ell}(u_{k}, u_{\ell})}{\partial \overline{u}_{\ell}}\left(\sum_{p=1}^nc_{\ell,p} \overline{h(u_{\ell}, u_p)}\right) \\ \nonumber 
 &=\sum_{p=1}^nc_{k,p} \left( \frac{\partial \widehat{h}_{k,\ell}(u_{k}, u_{\ell})}{\partial u_{k}} h(u_{k}, u_p) + \frac{\partial \widehat{h}_{k,\ell}(u_{k}, u_{\ell})}{\partial \overline{u}_{k}} \overline{h(u_{k}, u_p)}\right) \\ \nonumber 
 &+ \sum_{p=1}^nc_{\ell,p}\left(\frac{\partial \widehat{h}_{k,\ell}(u_{k}, u_{\ell})}{\partial u_{\ell}} h(u_{\ell}, u_p) + \frac{\partial \widehat{h}_{k,\ell}(u_{k}, u_{\ell})}{\partial \overline{u}_{\ell}} \overline{h(u_{\ell}, u_p)}\right) \\ \nonumber 
 &= \sum_{p=1}^nc_{k,p} \prescript{1}{}{G}_{k}^{\ell p}(u_k, u_{\ell}, u_p)+  \sum_{p=1}^nc_{\ell,p}\prescript{2}{}{G}_{k}^{\ell p}(u_k, u_{\ell}, u_p)\, ,
\end{align}
where we have set 
\begin{align}\label{2of1G2G}
\prescript{1}{}{G}_{k}^{\ell p}(u_k, u_{\ell}, u_p) &:= \frac{\partial \widehat{h}_{k,\ell}(u_{k}, u_{\ell})}{\partial u_{k}} h(u_{k}, u_p) + \frac{\partial \widehat{h}_{k,\ell}(u_{k}, u_{\ell})}{\partial \overline{u}_{k}} \overline{h(u_{k}, u_p)} \quad \text{ and } \\ \label{3of1G2G}
\prescript{2}{}{G}_{k}^{\ell p}(u_k, u_{\ell}, u_p) &:= \frac{\partial \widehat{h}_{k,\ell}(u_{k}, u_{\ell})}{\partial u_{\ell}} h(u_{\ell}, u_p) + \frac{\partial \widehat{h}_{k,\ell}(u_{k}, u_{\ell})}{\partial \overline{u}_{\ell}} \overline{h(u_{\ell}, u_p)}\, .
\end{align}
Combining equations \eqref{0of1G2G} through \eqref{3of1G2G}, we obtain
\begin{align}\label{triangleform01}
G_k(u) = \sum_{\ell=1}^n \sum_{p=1}^n c_{k,\ell}  c_{k,p} \prescript{1}{}{G}_{k}^{\ell p}(u_k, u_{\ell}, u_p) + \sum_{\ell=1}^n   \sum_{p=1}^n c_{k,\ell} c_{\ell,p}\prescript{2}{}{G}_{k}^{\ell p}(u_k, u_{\ell}, u_p) \, .
\end{align}
We may interpret Equation \eqref{triangleform01} as representing a new interaction structure, one where the interaction is now encoded through certain trees in the graph instead of links. In this regard, the emergent interaction function \eqref{triangleform01} looks a lot like our original response function \eqref{simplelinkform}, but counting such trees instead of links. The only way in which Equation \eqref{triangleform01}  does not generalize  Equation \eqref{simplelinkform} perfectly is by the fact that $\prescript{1}{}{G}_{k}^{\ell p}$ and $\prescript{1}{}{G}_{k}^{\ell p}$ have indices $k,\ell$ and $p$ (whereas ${h}$ does not). However, we see from equations \eqref{2of1G2G}  and \eqref{3of1G2G} that there is no dependence on $p$; this index is only there for notational purposes. Moreover, the dependence on $k$ and $\ell$ is only through a rescaling of the monomials. Hence, we find an emergent interaction that is in very good agreement with a generalization of our original interaction to tree interaction. What is more, the trees that Equation \eqref{triangleform01} counts are easily identified in the original graph. See Example \ref{examzi11}, which describes a special case of interaction through \eqref{simplelinkform}, and the corresponding Figure 3 of the main manuscript. \hfill $\triangle$
\end{remk}

\begin{ex}\label{secondspecialcase}
We return to Example \ref{example0easyinters12}, where the interaction functions are given by
\begin{equation}
 H_k(z) = \sum_{\ell=1}^n c_{k, \ell}(z_k\overline{z}_\ell + z_k^2\overline{z}_{\ell})\, .
\end{equation}
This is of the form \eqref{simplelinkform} as discussed in Remark \ref{sumsoftermsinH}, with $h$ given by
\begin{equation}
h(z_k, z_{\ell}) =  (z_k + z_k^2)\overline{z}_\ell\, .
\end{equation}
Following the notation of Remark \ref{sumsoftermsinH}, we see that
\begin{equation}
\widehat{h}_{k, \ell}(z_k, z_{\ell}) = \frac{z_k\overline{z}_\ell}{\overline{\gamma}_{\ell}} + \frac{z_k^2\overline{z}_{\ell}}{\gamma_k + \overline{\gamma}_{\ell}}.
\end{equation} 
We therefore find 
\begin{align}\label{2of1G2Gal}
\prescript{1}{}{G}_{k}^{\ell p}(u_k, u_{\ell}, u_p) &= \left(\frac{\overline{z}_\ell}{\overline{\gamma}_{\ell}} + \frac{2z_k\overline{z}_{\ell}}{\gamma_k + \overline{\gamma}_{\ell}}\right)(z_k + z_k^2)\overline{z}_p  = \frac{(z_k + z_k^2)\overline{z}_\ell \overline{z}_p}{\overline{\gamma}_{\ell}} + \frac{2(z^2_k + z_k^3)\overline{z}_{\ell}\overline{z}_p}{\gamma_k + \overline{\gamma}_{\ell}}\, , \\ \nonumber
\prescript{2}{}{G}_{k}^{\ell p}(u_k, u_{\ell}, u_p) &= \left(\frac{z_k}{\overline{\gamma}_{\ell}} + \frac{z_k^2}{\gamma_k + \overline{\gamma}_{\ell}} \right)(\overline{z}_{\ell} + \overline{z}_{\ell}^2){z}_p =  \frac{z_k(\overline{z}_{\ell} + \overline{z}_{\ell}^2){z}_p}{\overline{\gamma}_{\ell}} + \frac{z_k^2(\overline{z}_{\ell} + \overline{z}_{\ell}^2){z}_p}{\gamma_k + \overline{\gamma}_{\ell}} \, .
\end{align}
As we may ignore terms of degree 5 and higher, we may also set
\begin{align}\label{2of1G2Gall}
\prescript{1}{}{G}_{k}^{\ell p}(u_k, u_{\ell}, u_p) &= \frac{(z_k + z_k^2)\overline{z}_\ell \overline{z}_p}{\overline{\gamma}_{\ell}} + \frac{2z^2_k\overline{z}_{\ell}\overline{z}_p}{\gamma_k + \overline{\gamma}_{\ell}}\, , \\ \nonumber
\prescript{2}{}{G}_{k}^{\ell p}(u_k, u_{\ell}, u_p) &=  \frac{z_k(\overline{z}_{\ell} + \overline{z}_{\ell}^2){z}_p}{\overline{\gamma}_{\ell}} + \frac{z_k^2\overline{z}_{\ell}{z}_p}{\gamma_k + \overline{\gamma}_{\ell}} \, ,
\end{align}
which describe the new interaction through Equation \eqref{triangleform01}. \hfill $\triangle$
\end{ex}

\subsection{A coordinate transformation}\label{thefirstreduction} 
In this subsection and the next ones we prove Theorem \ref{main001} and the accompanying Proposition \ref{newintterms}. Recall that we want to transform 
\begin{align}\label{theODE00-1}
\dot{z}_k &= \gamma_kz_k - \beta_k z_k|z_k|^2 + \alpha H_k(z) \, ,
\end{align}
into an ODE where the leading interaction terms are of order $\alpha^2$. Recall as well that $H_k^5(z)$ denotes the Taylor expansion of $H_k(z)$ up to fifth order. In particular, we may write 
\begin{align}\label{formofH1i}
    H_k(z) =  H^5_k(z) + \mathcal{O}(|z|^6)\, .
\end{align}

It follows that $H^5_k(z)$ is a complex polynomial of order $5$ in the variables $z_1, \dots, z_n$ and $\overline{z}_1, \dots, \overline{z}_n$. We write $z = (z_1, \dots, z_n)$, and similarly for other variables, and assume implicitly that any function of $z$ may also depend on its complex conjugate $\overline{z} = (\overline{z}_1, \dots, \overline{z}_n)$.

\noindent We start by rewriting the ODE \eqref{theODE00-1} using the transformation 
\begin{equation}\label{expresssss1}
w_k = z_k - \alpha \widehat{H}_k^5(z) =  z_k - \alpha P_k(z)   \, ,
\end{equation}
where we have set $P_k := \widehat{H}_k^5$ for convenience. Note that each $P_k: \mathbb{C}^n \rightarrow \mathbb{C}$ is a complex polynomial  of lower degree $2$ (see Definition \ref{lowerdegree}). It follows that Expression \eqref{expresssss1} describes an invertible transformation around $z = 0$. The following lemma deals with its inverse.

\begin{lem}\label{lemzinw0}
Suppose the variables $w = (w_1, \dots, w_n)$ may be expressed in $z = (z_1, \dots, z_n)$ and $\alpha$ by
\begin{equation}\label{expressss10}
w_k = z_k - \alpha P_k(z) \, ,
\end{equation}
for some polynomials $P_k$ of lower degree $d \geq 2$. Then $z$ can be expressed in $w$ and $\alpha$ by the formal expression
\begin{equation}\label{expressss20}
z_k = w_k+ \alpha P_k(w) + \alpha^2 R_{k,2}(w) + \alpha^3 R_{k,3}(w) + \dots \, .
\end{equation}
Here the $R_{k,t}(w)$ are polynomials with lower degree $(d-1)t+1$ or higher. 
\end{lem}

\begin{proof}
We write
\begin{equation}\label{expressssss200}
z_k = R_{k,0}(w) + \alpha R_{k,1}(w) + \alpha^2R_{k,2}(w)+\dots\, ,
\end{equation}
for some functions $R_{k,t}(w):\mathbb{C}^n \rightarrow \mathbb{C}$.  To determine these functions, we substitute the $z$ variables in Equation \eqref{expressss10} by Expression \eqref{expressssss200}. We obtain
\begin{align}\label{winzab0}
w_k &= z_k - \alpha P_k(z_1, \dots, z_n) \\ \nonumber
&=  [R_{k,0}(w) + \alpha R_{k,1}(w) + \dots]  - \alpha P_k([R_{1,0}(w) + \alpha R_{1,1}(w) + \dots],  \dots, [R_{n,0}(w) + \alpha R_{n,1}(w) + \dots])\, .
\end{align}
Comparing constant terms in $\alpha$ (i.e. $\alpha^0$), Expression \eqref{winzab0} gives us
\begin{equation}
w_k = R_{k,0}(w)\, .
\end{equation}
This simplifies Equation \eqref{winzab0} to 
\begin{align}\label{winzab20}
w_k &= [w_k + \alpha R_{k,1}(w) + \dots]  
- \alpha P_k([w_1 + \alpha R_{1,1}(w) + \dots], \dots, [w_n + \alpha R_{n,1}(w) + \dots])\, .
\end{align}
Comparing $\alpha$-terms now yields
\begin{equation}
0 = R_{k,1}(w) - P_k(w)\, ,
\end{equation}
so that
\begin{equation}
R_{k,1}(w) = P_k(w)\, .
\end{equation}
It remains to show that the higher order terms are indeed polynomials of the required lower degree. We will show this by induction on $t$. Note that $ R_{k,0}(w) = w_k$ has (lower) degree $(d-1)0 +1 = 1$. Likewise, $R_{k,1}(w) = P_k(w)$ is of  lower degree $(d-1)1 +1 = d$. We therefore fix an integer $T > 1$ and assume that the function $R_{k,t}(w)$ is a complex polynomial of lower degree $(d-1)t +1$ or higher for all $t<T$ and $k \in \{1, \dots, n\}$. The $\alpha^T$ terms in Equation \eqref{winzab20} are given by
\begin{align}\label{compaarq0}
0 &= R_{k,T}(w) - [\alpha^{T-1}]P_k([w_1 + \alpha R_{1,1}(w) + \dots], \dots, [w_n + \alpha R_{n,1}(w) + \dots])\,.
\end{align}
Here $[\alpha^{T-1}]F(\alpha)$ denotes the $\alpha^{T-1}$ term in the expansion of a function $F$ in $\alpha$. As $P_k$ is a polynomial of lower degree $d$, the $\alpha^{T-1}$ term in 
\[P_k([w_1 + \alpha R_{1,1}(w) + \dots], \dots, [w_n + \alpha R_{n,1}(w) + \dots]) \]
must be a finite sum of scalar multiples of expressions of the form
\[ \tilde{R}_{i_1,t_1}(w)\tilde{R}_{i_2,t_2}(w) \dots \tilde{R}_{i_s,t_s}(w) \, , \]
for $s \geq d$ and for some $i_1, \dots, i_s \in \{1, \dots, n\}$ and $t_1, \dots, t_s \in \mathbb{Z}_{\geq 0}$ satisfying $t_1 + \dots + t_s = T-1$. Each term $\tilde{R}_{i_j,t_j}(w)$ may furthermore independently denote ${R}_{i_j,t_j}(w)$ or its complex conjugate $\overline{R_{i_j,t_j}(w)}$. As we have $t_1 + \dots + t_s = T-1$, it in particular holds that $t_1, \dots, t_s \leq T-1$. By the induction hypothesis, we therefore know that each of the terms $\tilde{R}_{i_j,t_j}$ is a polynomial of lower degree $(d-1)t_j + 1$ or higher. This means the expression
\[ \tilde{R}_{i_1,t_1}(w)\tilde{R}_{i_2,t_2}(w) \dots \tilde{R}_{i_s,t_s}(w) \,  \]
is a polynomial of lower degree $D$ satisfying
\begin{align}
  D &\geq  [(d-1)t_1 + 1] + [(d-1)t_2 + 1] + \dots + [(d-1)t_s + 1] \nonumber \\ \nonumber
    &= (d-1)(t_1 + t_2 + \dots + t_s) + s \\ \nonumber
    &= (d-1)(T-1) + s \geq  (d-1)(T-1) + d   \\ \nonumber
    &=(d-1)T + 1 \, .
\end{align}
It follows from Equation \eqref{compaarq0} that $R_{k,T}(w)$ is indeed a polynomial of lower degree $(d-1)T + 1$ or higher for all $k \in \{1, \dots, n\}$. This proves the lemma by induction.
\end{proof}
\noindent Setting $d=2$, it follows that the inverse of Equation \eqref{expresssss1} is given by 
\begin{equation}\label{expresssss10}
z_k = w_k + \alpha P_k(w) + \mathcal{O}(|\alpha|^2|w|^3) \, .
\end{equation}
At some point later on, we will need to know Expression \eqref{expresssss10} up to higher order terms. To this end, we will show how our bracket $[\bullet|| \bullet]$ from Definition \ref{defibracket} shows up when performing coordinate transformations.

\begin{lem}\label{useful1}
Let $R(z)$ be a complex polynomial and suppose we may express the $z$-variables in some new $w$-variables by 
\begin{equation}\label{plugin234}
    z_k = w_k + \alpha S_k(w) + \mathcal{O}(|\alpha|^2) \, , \quad k \in \{1, \dots, n\} \, .
\end{equation}
Here each $S_k$ is a complex polynomial and we have $\alpha \in \mathbb{R}$. Then $R(z)$ is given in the $w$-variables by
\begin{equation}\label{transfff123}
R(z) = R(w) + \alpha[R||S](w) +  \mathcal{O}(|\alpha|^2) \, ,
\end{equation}
where we have set $S = (S_1, \dots, S_n)$.
\end{lem}

\begin{proof}
We write $R(z)$ as 
\begin{equation}
    R(z) = T_0(w) + \alpha T_1(w) +  \mathcal{O}(|\alpha|^2) \, ,
\end{equation}
where $T_0(w)$ and $T_1(w)$ are to be determined. Assume first that $R(z)$ is given by 
\[R(z) = z_{i_1}z_{i_2}\dots z_{i_m}\overline{z}_{j_1}\overline{z}_{j_2}\dots\overline{z}_{j_l} ,  \]
for some (not necessarily distinct) $i_1, \dots, i_m, j_1, \dots, j_l \in \{1, \dots, n\}$. We get
\begin{align}
    R(z) &= z_{i_1}\dots z_{i_m}\overline{z}_{j_1}\dots\overline{z}_{j_l} \\ \nonumber
    &= (w_{i_1}+\alpha S_{i_1}(w))\dots (w_{i_m}+\alpha S_{i_m}(w)) \overline{(w_{j_1}+\alpha S_{j_1}(w))}\dots\overline{(w_{j_l}+\alpha S_{j_l}(w))} + \mathcal{O}(|\alpha|^2) \\ \nonumber
    &= w_{i_1}\dots w_{i_m}\overline{w}_{j_1}\dots\overline{w}_{j_l} + \alpha \left( \sum_{s=1}^m\frac{R(w)}{w_{i_s}}S_{i_s}(w) + \sum_{r=1}^l\frac{R(w)}{\overline{w}_{i_r}}\overline{S_{i_r}(w)} \right) + \mathcal{O}(|\alpha|^2) \\ \nonumber
    &= R(w) + \alpha [R||S](w) + \mathcal{O}(|\alpha|^2)\, ,
\end{align}
where in the last line we have used Expression \eqref{needlaterrron1} from the proof of Lemma \ref{useful0}. As $T_0$ and $T_1$ are determined linearly by $R$, we may conclude from Lemma \ref{useful0} that $T_0(w) = R(w)$ and $T_1(w) = [R||S](w)$ for general polynomials $R$. This completes the proof.
\end{proof}

\begin{ex}
Suppose we are given the polynomial $Q(z) = Q(z_1, z_2, \overline{z}_1, \overline{z}_2) = z_1^2 + z_1\overline{z}_2$. Equation \eqref{plugin234} gives
\begin{align}
 Q(z) &=    (w_1 + \alpha S_1(w) + \mathcal{O}(|\alpha|^2))^2 + (w_1 + \alpha S_1(w) + \mathcal{O}(|\alpha|^2)\overline{(w_2 + \alpha S_2(w) + \mathcal{O}(|\alpha|^2))} \\ \nonumber
 &=    (w_1 + \alpha S_1(w))^2 + (w_1 + \alpha S_1(w))(\overline{w}_2 + \alpha \overline{S_2(w)}) + \mathcal{O}(|\alpha|^2) \\ \nonumber
 &= w_1^2 + w_1\overline{w}_2 + \alpha(2w_1S_1(w) + w_1\overline{S_2(w)} + S_1(w)\overline{w}_2) + \mathcal{O}(|\alpha|^2) \\ \nonumber
 &= Q(w) + \alpha \frac{d}{dt} Q(w)\left|{\begin{array}{l}
  \dot{w}_1 =   S_1(w) \\
  \dot{w}_2 =   S_2(w) 
\end{array}}\right. + \mathcal{O}(|\alpha|^2) \\ \nonumber
 &= Q(w) + \alpha[Q||S](w) + \mathcal{O}(|\alpha|^2)\,  ,
\end{align}
which is in accordance with Lemma \ref{useful1}. \hfill $\triangle$
\end{ex}
\noindent Returning to the transformation \eqref{expresssss1} with inverse Equation \eqref{expresssss10}, we may in fact conclude the following:
\begin{lem}\label{fullinverse0}
Suppose we have a coordinate transformation of the form
\begin{equation}\label{expresssss01}
w_k = z_k - \alpha P_k(z) \, ,
\end{equation}
where each $P_k: \mathbb{C}^n \rightarrow \mathbb{C}$ is a complex polynomial of lower degree $2$ or higher. The inverse transformation is given by
\begin{equation}\label{expresssss122}
z_k = w_k + \alpha P_k(w) + \alpha^2 [P_k||P](w) + \mathcal{O}(|\alpha|^3|w|^4)\, ,
\end{equation}
where $P = (P_1, \dots, P_n)$.
\end{lem}

\begin{proof}
It follows from Lemma \ref{lemzinw0} that we may write
\begin{equation}\label{expresssss1223}
z_k = w_k + \alpha P_k(w) + \alpha^2 R_{k,2}(w) + \mathcal{O}(|\alpha|^3|w|^4)\, ,
\end{equation}
for some function $R_{k,2}(w)$. Hence, we only have to show that $R_{k,2}(w) = [P_k||P](w)$. To this end, we rewrite Expression \eqref{expresssss01} as
\begin{equation}\label{expresssss01inv}
z_k = w_k + \alpha P_k(z) \, .
\end{equation}
Next, we use Equation \eqref{expresssss1223} to write
\begin{equation}\label{expresssss12234}
z_k = w_k + \alpha P_k(w) + \mathcal{O}(|\alpha|^2)\, .
\end{equation}
Applying Lemma \ref{useful1} to the term $P_k(z)$ and the transformation \eqref{expresssss12234} yields
\begin{align}
    P_k(z) = P_k(w) + \alpha [P_k||P](w) + \mathcal{O}(|\alpha|^2)\, .
\end{align}
Combined with Equation \eqref{expresssss01inv}, we obtain
\begin{align}\label{expresssss01inv1}
z_k &= w_k + \alpha (P_k(w) + \alpha [P_k||P](w) + \mathcal{O}(|\alpha|^2)) \\ \nonumber
&= w_k + \alpha P_k(w) + \alpha^2 [P_k||P](w) + \mathcal{O}(|\alpha|^3)\, .
\end{align}
Comparing the two expressions \eqref{expresssss1223} and \eqref{expresssss01inv1} for $z_k$, we see that indeed 
\begin{equation}\label{expresssss122w}
z_k = w_k + \alpha P_k(w) + \alpha^2 [P_k||P](w) + \mathcal{O}(|\alpha|^3|w|^4)\, .
\end{equation}
This proves the lemma.
\end{proof}

Our next step is to differentiate Equation \eqref{expresssss1} with respect to time. This gives us
\begin{align}\label{calcu0001}
\dot{w}_k = \dot{z}_k - \alpha \frac{d}{d t}P_k(z) \, .
\end{align}
We will first focus on the term 
\begin{equation}\label{ddtpkzxx} 
\frac{d}{d t}P_k(z)\, ,
\end{equation}
 and then deal with the term $\dot{z}_k$. \\

\paragraph{The term $\partial_tP_k$}\label{firstterm} 

\noindent We first focus on the term \eqref{ddtpkzxx}. We start by rewriting Equation \eqref{theODE00-1} as
\begin{align}\label{theODE00c}
\dot{z}_k   &= \gamma_kz_k - \beta_k z_k|z_k|^2 + \alpha H_k(z) \\ \nonumber
            &= \gamma_kz_k - \beta_k z_k|z_k|^2 + \alpha{H}^5_k(z)  + \mathcal{O}(|\alpha||z|^6)\, ,
\end{align}
where we recall that ${H}^5_k(z)$ denotes the Taylor expansion of $H_k(z)$ up to fifth order. 
From Equation \eqref{theODE00c} and Lemma \ref{useful0} we get
\begin{align}\label{part1of1}
\frac{d}{dt}P_k(z) &= \frac{d}{dt}P_k(z)\left|{\begin{array}{l}
  \dot{z}_j =  \gamma_jz_j - \beta_j z_j|z_j|^2 + \alpha {H}_j(z)\\
   \forall \, j \in 
   \{1, \dots, n\} \\
\end{array}}\right. \\ \nonumber
&= \frac{d}{dt}P_k(z)\left|{\begin{array}{l}
  \dot{z}_j =  \gamma_jz_j - \beta_j z_j|z_j|^2 + \alpha {H}^5_j(z)\\
   \forall \, j \in 
   \{1, \dots, n\} \\
\end{array}}\right. + \quad \mathcal{O}(|\alpha||z|^7) \\ \nonumber
&= [P_k||(\dots, \gamma_jz_j - \beta_j z_j|z_j|^2 + \alpha {H}^5_j(z), \dots)](z) + \mathcal{O}(|\alpha||z|^7) \\ \nonumber
&= [P_k||(\dots, \gamma_jz_j, \dots)](z) -[P_k|| (\dots, \beta_j z_j|z_j|^2, \dots)](z) \\ \nonumber
&+ \alpha [P_k|| (\dots, {H}^5_j(z), \dots)](z) + \mathcal{O}(|\alpha||z|^7) \\ \nonumber
&= \Gamma P_k(z) - L^1_k(z) + \alpha [P_k|| {H}^5](z) + \mathcal{O}(|\alpha||z|^7) \, ,
\end{align}
where we have set 
\[L^1_k(z) :=  [P_k|| (\dots, \beta_j z_j|z_j|^2, \dots)](z) \] and
\[{H}^5 := (\dots, {H}^5_j(z), \dots) \, .\] 
We have moreover used that $P_k(z)$ has lower degree at least $2$ to arrive at the remainder term  $\mathcal{O}(|\alpha||z|^7)$, and we refer to Definition \ref{defibracket} for the meaning of the term $\Gamma P_k(z)$. Note that $L^1_k(z)$ is a polynomial of lower degree at least $4$, whereas $[P_k|| {H}^5](z)$ has lower degree $3$ or higher.

 Next, we return to Equation \eqref{expresssss10}, which we recall states
\begin{equation}\label{replacccs}
z_k = w_k + \alpha P_k(w) + \mathcal{O}(|\alpha|^2|w|^3) =  w_k  + \mathcal{O}(|\alpha||w|^2)  \, .
\end{equation}
We obtain
\begin{align}\label{part2of1}
    L^1_k(z) &= L^1_k(w) + \mathcal{O}(|\alpha||w|^5) \\ \label{part2of1b}
    [P_k|| {H}^5](z) &= [P_k|| {H}^5](w) + \mathcal{O}(|\alpha||w|^4) \, .
\end{align}
From Lemma \ref{useful1} we furthermore get
\begin{align}\label{part3of1}
    \Gamma P_k(z) &= \Gamma P_k(w) + \alpha [\Gamma P_k||P](w) + \mathcal{O}(|\alpha|^2|w|^4) \, ,
\end{align}
where we have set $P = (P_1, \dots, P_n)$. The remainder term in Equation \eqref{part3of1} follows from the lower degrees of $\Gamma P_k(z)$ and $P_k(z)$, and the remainder in Equation \eqref{replacccs}.

Note  that $[\Gamma P_k||P](w)$ is a polynomial of lower degree $3$ or higher. Combining equations \eqref{part1of1}, \eqref{part2of1}, \eqref{part2of1b} and \eqref{part3of1}, we arrive at:

\begin{lem}\label{part1}
The term 
\[\frac{d}{dt}P_k(z) \]
may be expressed in the new $w$ coordinates by
\begin{align}
\frac{d}{dt}P_k(z) &=  \Gamma P_k(z) - L^1_k(z) + \alpha [P_k|| {H}^5](z) + \mathcal{O}(|\alpha||z|^7) \\ \nonumber
&=  \Gamma P_k(w) - L^1_k(w) + \alpha [P_k|| {H}^5](w) + \alpha [\Gamma P_k||P](w) + \mathcal{O}(|\alpha||w|^5 + |\alpha|^2|w|^4)\, .
\end{align}
\end{lem}

\paragraph{The term $\dot{z}_k$}\label{secondterm} 

\noindent Next, we focus on the term $\dot{z}_k$. Again we write.
\begin{align}
\dot{z}_k   &= \gamma_kz_k - \beta_k z_k|z_k|^2 + \alpha H_k(z) \\ \nonumber
            &= \gamma_kz_k - \beta_k z_k|z_k|^2 + \alpha H_k^5(z) + \mathcal{O}(|\alpha||z|^6)\, ,
\end{align}
where ${H}^5_k(z)$  denotes the Taylor expansion of $H_k(z)$ up to fifth order.

Recall the result of Lemma \ref{fullinverse0}, which tells us that
\begin{equation}\label{expresssss1225}
z_k = w_k + \alpha P_k(w) + \alpha^2 [P_k||P](w) + \mathcal{O}(|\alpha|^3|w|^4)\, .
\end{equation}
Combined, and using Lemma \ref{useful1}, we get
\begin{align}\label{part2aa}
    \dot{z}_k &= \gamma_kz_k - \beta_k z_k|z_k|^2 + \alpha {H}^5_k(z) + \mathcal{O}(|\alpha||z|^6) \\ \nonumber
    &= \gamma_k(w_k + \alpha P_k(w) + \alpha^2 [P_k||P](w)) + \mathcal{O}(|\alpha|^3|w|^4) \\ \nonumber
    &- \beta_k w_k|w_k|^2 -  \alpha [\beta_k w_k |w_k|^2|| P](w) + \mathcal{O}(|\alpha|^2|w|^5) \\ \nonumber
    &+ \alpha {H}^5_k(w) +\alpha^2[{H}^5_k|| P](w) + \mathcal{O}(|\alpha|^3|w|^4) \\ \nonumber
& + \mathcal{O}(|\alpha||w|^6) \\ \nonumber
&= \gamma_kw_k - \beta_k w_k|w_k|^2 + \alpha(\gamma_kP_k(w) + {H}^5_k(w)) - \alpha [\beta_k w_k |w_k|^2|| P](w) \\ \nonumber
&+ \alpha^2(\gamma_k[P_k||P](w) + [{H}^5_k|| P](w))  + \mathcal{O}(|\alpha||w|^6 + |\alpha|^2|w|^5 + |\alpha|^3|w|^4)\, .
\end{align}
We will write 
\begin{align}
   L^2_k(w) :=  [\beta_k w_k |w_k|^2|| P](w)\, ,
\end{align}
which has lower degree $4$ or higher, to arrive at:
\begin{lem}\label{part2}
The term $\dot{z}_k$ may be expressed in the new $w$ coordinates by
\begin{align}\label{part2tt}
    \dot{z}_k &= \gamma_kw_k - \beta_k w_k|w_k|^2 \\ \nonumber
    &+  \alpha(\gamma_kP_k(w) + {H}^5_k(w))  - \alpha L^2_k(w) \\ \nonumber
  &+ \alpha^2(\gamma_k[P_k||P](w) + [{H}^5_k|| P](w)) \\ \nonumber
& + \mathcal{O}(|\alpha||w|^6 + |\alpha|^2|w|^5 + |\alpha|^3|w|^4)\, .
\end{align}
\end{lem}

\subsection{The first reduction}\label{first2}

 We may now substitute the results of Lemma \ref{part1} and Lemma \ref{part2} into
\begin{align}
\dot{w}_i = \dot{z}_i - \alpha \frac{d}{d t}P_i(z) \, .
\end{align}
We obtain
\begin{align}\label{expandedff1}
\dot{w}_k &=  \gamma_kw_k - \beta_k w_k|w_k|^2 \\ \nonumber
    &+  \alpha(\gamma_kP_k(w) + {H}^5_k(w))  - \alpha L^2_k(w) \\ \nonumber
  &+ \alpha^2(\gamma_k[P_k||P](w) + [{H}^5_k|| P](w)) \\ \nonumber
& + \mathcal{O}(|\alpha||w|^6 + |\alpha|^2|w|^5 + |\alpha|^3|w|^4)\\ \nonumber
& - \alpha(\Gamma P_k(w) - L^1_k(w) + \alpha [P_k|| {H}^5](w) + \alpha [\Gamma P_k||P](w)) \\ \nonumber
&=  \gamma_kw_k - \beta_k w_k|w_k|^2 \\ \nonumber
&+  \alpha(\gamma_kP_k(w) + {H}^5_k(w) - \Gamma P_k(w)) + \alpha(L^1_k(w) - L^2_k(w)) \\ \nonumber
&+ \alpha^2(\gamma_k[P_k||P](w) + [{H}^5_k|| P](w) - [P_k|| {H}^5](w) - [\Gamma P_k||P](w)) \\ \nonumber
&+ \mathcal{O}(|\alpha||w|^6 + |\alpha|^2|w|^5 + |\alpha|^3|w|^4)\, . 
\end{align}
By Lemma \ref{useful0} we may further write this as
\begin{align}\label{expandedff2}
\dot{w}_k &=  \gamma_kw_k - \beta_k w_k|w_k|^2 \\ \nonumber
&+  \alpha(\gamma_kP_k(w) + {H}^5_k(w) - \Gamma P_k(w)) + \alpha(L^1_k(w) - L^2_k(w)) \\ \nonumber
&+ \alpha^2[\gamma_kP_k + {H}^5_k - \Gamma P_k||P](w)  - \alpha^2[P_k|| {H}^5](w)  \\ \nonumber
&+ \mathcal{O}(|\alpha||w|^6 + |\alpha|^2|w|^5 + |\alpha|^3|w|^4)\, . 
\end{align}
Next, we claim that our choice of polynomial $P_k = \widehat{H}^5_k$ guarantees that the term
$$\label{solveforalpha}\gamma_kP_k + {H}^5_k - \Gamma P_k$$
vanishes. More precisely, we prove:
\begin{lem}\label{solvinfhomological}
Let $Q$ be a polynomial for which the $k$th non-resonance condition is satisfied. In particular, it follows that the corresponding modified polynomial $\widehat{Q}_k$ is well-defined. We then have 
\begin{align}\label{solveforalpha}\gamma_k\widehat{Q}_k + Q - \Gamma \widehat{Q}_k = 0\,.\end{align}
\end{lem}

\begin{proof}
By definitions \ref{defibracket} and \ref{augmentedd}, we see that it suffices to show this when $Q$ is given by a single monomial 
\[ Q(z) =  z_1^{s_1}\dots  z_n^{s_n} \overline{z}_1^{t_1}  \dots  \overline{z}_n^{t_n}\, ,\] 
where $t_1, \dots, t_n, s_1, \dots, s_n$ are non-negative integers. More precisely, we use here that the maps $Q \mapsto \widehat{Q}_k$ and $Q \mapsto \Gamma Q = [Q||\dots ,\gamma_jz_j, \dots]$ are complex linear, when defined. By Definition  \ref{augmentedd} we find 
\[\widehat{Q}_k(z) = \frac{z_1^{s_1}\dots  z_n^{s_n} \overline{z}_1^{t_1}  \dots  \overline{z}_n^{t_n}}{s_1\gamma_1 +\dots + s_n\gamma_n + t_1\overline{\gamma}_1 + \dots + t_n\overline{\gamma}_n - \gamma_k} = \frac{Q(z)}{s_1\gamma_1 +\dots + s_n\gamma_n + t_1\overline{\gamma}_1 + \dots + t_n\overline{\gamma}_n - \gamma_k}\, .\]
Example \ref{exgamma} now tells us that 
\begin{align}
\Gamma\widehat{Q}_k(z) &= \frac{(s_1\gamma_1 +\dots + s_n\gamma_n + t_1\overline{\gamma}_1 + \dots + t_n\overline{\gamma}_n)z_1^{s_1}\dots  z_n^{s_n} \overline{z}_1^{t_1}  \dots  \overline{z}_n^{t_n}}{s_1\gamma_1 +\dots + s_n\gamma_n + t_1\overline{\gamma}_1 + \dots + t_n\overline{\gamma}_n - \gamma_k} \\ \nonumber
&= \frac{(s_1\gamma_1 +\dots + s_n\gamma_n + t_1\overline{\gamma}_1 + \dots + t_n\overline{\gamma}_n)Q(z)}{s_1\gamma_1 +\dots + s_n\gamma_n + t_1\overline{\gamma}_1 + \dots + t_n\overline{\gamma}_n - \gamma_k}\, .
\end{align}
We therefore conclude that
\begin{align}
\Gamma\widehat{Q}_k(z) - \gamma_k\widehat{Q}_k(z)&= \frac{(s_1\gamma_1 +\dots + s_n\gamma_n + t_1\overline{\gamma}_1 + \dots + t_n\overline{\gamma}_n)Q(z)}{s_1\gamma_1 +\dots + s_n\gamma_n + t_1\overline{\gamma}_1 + \dots + t_n\overline{\gamma}_n - \gamma_k} \\ \nonumber
&- \frac{\gamma_kQ(z)}{s_1\gamma_1 +\dots + s_n\gamma_n + t_1\overline{\gamma}_1 + \dots + t_n\overline{\gamma}_n - \gamma_k} \\ \nonumber
&=  \frac{(s_1\gamma_1 + \dots + s_n\gamma_n + t_1\overline{\gamma}_1 + \dots + t_n\overline{\gamma}_n - \gamma_k)Q(z)}{s_1\gamma_1 +\dots + s_n\gamma_n + t_1\overline{\gamma}_1 + \dots + t_n\overline{\gamma}_n - \gamma_k}  = Q(z) \, .
\end{align}
Thus, we precisely find 
\begin{align} 
\gamma_k\widehat{Q}_k + Q - \Gamma \widehat{Q}_k = - (\Gamma \widehat{Q}_k - \gamma_k\widehat{Q}_k) + Q = -Q + Q = 0\, ,
\end{align}
which completes the proof.
\end{proof}
As we have used the shorthand notation  $\widehat{H}_k^5 := \widehat{({H^5_k})}_k$, we see that indeed
\begin{align}
\gamma_kP_k + {H}^5_k - \Gamma P_k = \gamma_k\widehat{H}_k^5 + {H}^5_k - \Gamma \widehat{H}_k^5 = 0\, .
\end{align}
Returning to Equation \eqref{expandedff2}, we find that it simplifies to 
\begin{align}\label{expandedff3}
\dot{w}_k &=  \gamma_kw_k - \beta_k w_k|w_k|^2 + \alpha(L^1_k(w) - L^2_k(w))  - \alpha^2[P_k|| {H}^5](w)  \\ \nonumber
&+ \mathcal{O}(|\alpha||w|^6 + |\alpha|^2|w|^5 + |\alpha|^3|w|^4)\, , 
\end{align}

\noindent where we recall that $L^1_k$ and $L^2_k$ are defined as
\begin{align}
L^1_k(w) &:=  [P_k|| (\dots, \beta w_j|w_j|^2, \dots)](w) \text{ and }\\ \nonumber
L^2_k(w) &:=  [\beta_k w_k |w_k|^2|| P](w)\, ,
\end{align}
which are both polynomials of lower degree $4$ or higher. 
\subsection{The second reduction}\label{thesecondreduction}
Next, we wish to get rid of the term $\alpha(L^1_k(w) - L^2_k(w))$ in Equation \eqref{expandedff3}. This follows along the same lines as in the previous reduction. We start by defining new variables
\begin{equation}
    u_k = w_k - \alpha Q_k(w)\,
\end{equation}
where each $Q_k$ is a polynomial of lower degree $4$ or higher. Note that by Lemma \ref{lemzinw0} we may write
\begin{equation}\label{invssofw}
    w_k = u_k + \alpha Q_k(u) + \mathcal{O}(|\alpha|^2|u|^7)\, .
\end{equation}
Using Equation \eqref{expandedff3} we obtain
\begin{align}
    \dot{u}_k &= \dot{w}_k - \alpha \frac{d}{dt}Q_k(w)\left|{\begin{array}{l}
  \dot{w}_j =  \gamma_jw_j\\
   \forall \, j \in 
   \{1, \dots, n\} \\
\end{array}}\right. + \quad \mathcal{O}(|\alpha||w|^6) \\ \nonumber
&= \dot{w}_k - \alpha\Gamma Q_k(w) + \mathcal{O}(|\alpha||w|^6) \\ \nonumber
&= \gamma_kw_k - \beta_k w_k|w_k|^2 +\alpha(L^1_k(w) - L^2_k(w))  - \alpha^2[P_k|| {H}^5](w) - \alpha\Gamma Q_k(w) \\ \nonumber
&+ \mathcal{O}(|\alpha||w|^6 + |\alpha|^2|w|^5 + |\alpha|^3|w|^4) \, .
\end{align}
Next, substituting $w_k$ by the right hand side of Equation \eqref{invssofw} yields
\begin{align}\label{equationalmostthere345}
 \dot{u}_k &= \gamma_kw_k - \beta_k w_k|w_k|^2 +\alpha(L^1_k(w) - L^2_k(w))  - \alpha^2[P_k|| {H}^5](w) - \alpha\Gamma Q_k(w) \\ \nonumber
&+ \mathcal{O}(|\alpha||w|^6 + |\alpha|^2|w|^5 + |\alpha|^3|w|^4) \\ \nonumber
&= \gamma_ku_k + \alpha\gamma_kQ_k(u)- \beta_k u_k|u_k|^2 +\alpha(L^1_k(u) - L^2_k(u))  - \alpha^2[P_k|| {H}^5](u) - \alpha\Gamma Q_k(u) \\ \nonumber
&+ \mathcal{O}(|\alpha||u|^6 + |\alpha|^2|u|^5 + |\alpha|^3|u|^4) \\ \nonumber
&= \gamma_ku_k - \beta_k u_k|u_k|^2 +\alpha(L^1_k(u) - L^2_k(u)+ \gamma_kQ_k(u) - \Gamma Q_k(u) )  - \alpha^2[P_k|| {H}^5](u) \\ \nonumber
&+ \mathcal{O}(|\alpha||u|^6 + |\alpha|^2|u|^5 + |\alpha|^3|u|^4)  \, .
\end{align}
It remains to choose $Q_k$ such that
\begin{equation}\label{tosolve283}
L^1_k(u) - L^2_k(u)+ \gamma_kQ_k(u) - \Gamma Q_k(u) = 0\, .
\end{equation}
Setting $S_k(u) := L^1_k(u) - L^2_k(u)$, Equation \eqref{tosolve283} becomes 
\begin{align}\label{solveforalpha2}
\gamma_kQ_k + S_k - \Gamma Q_k = 0\, ,
\end{align}
which is of the same form as Equation \eqref{solveforalpha}. It therefore follows from Lemma \eqref{solvinfhomological} that a solution to Equation \eqref{tosolve283} is given by $Q_k = \widehat{S}_k := \widehat{(S_k)}_k$, if indeed this is well-defined.
The following lemmas show that the non-resonance conditions of $H_k^5$ are enough to ensure $\widehat{S}_k$ exists.  

\begin{lem}\label{lemnonreso1}
The polynomial $S_k(u) = L^1_k(u) - L^2_k(u)$ may be expressed as the sum of terms $u_k^2\overline{R(u)}$ and  $|u_j|^2R(u)$ for  $j \in \{1, \dots, n\}$ and with $R(u)$ a monomial term appearing in $P_k(u)$.
\end{lem}

\begin{proof}
We start with $L^2_k(u)$. By definition, we have
\begin{align}
    L^2_k(u) &=  [\beta_k u_k |u_k|^2|| P](u) = 2\beta_k u_k\overline{u}_kP_k(u) + \beta_k u_k^2\overline{P_k(u)} \\ \nonumber
    &=  2\beta_k |u_k|^2P_k(u) + \beta_k u_k^2\overline{P_k(u)}\, .
\end{align}
As $P_k(u)$ may be expressed as the sum of monomials that appear in $P_k(u)$ (tautologically), we see that $ L^2_k(u)$  can indeed be written as the sum of terms $u_k^2\overline{R(u)}$ and  $|u_j|^2R(u)$, with $R(u)$ a monomial appearing in $P_k(u)$. \\

 Next, recall that $L^1_k(u)$ is defined as
\begin{equation}
    L^1_k(u) :=  [P_k|| (\dots, \beta_j u_j|u_j|^2, \dots)](u)\, .
\end{equation}
By definition of the bracket $[\bullet|| \bullet]$, this means $L^1_k(u)$ is obtained from $P_k$ by substituting terms $u_j$ by $\beta_j u_j|u_j|^2$ and terms $\overline{u}_j$ by $\overline{\beta_j u_j|u_j|^2} = \overline{\beta_j u_j}|u_j|^2$. More precisely, if $R(u)$ is a monomial term of $P_k(u)$ given by
\begin{align}
    R(u) = u_1^{s_1}\dots  u_n^{s_n} \overline{u}_1^{t_1} \dots \overline{u}_n^{t_n}\, ,
\end{align}
then we find
\begin{align}
    &[R(u)|| (\dots, \beta_j u_j|u_j|^2, \dots)](u)\\ \nonumber
    = &\sum_{j=1}^n s_j u_1^{s_1} \dots   u_j^{s_j-1}(\beta_j u_j|u_j|^2)  \dots  u_n^{s_n} \overline{u}_1^{t_1} \dots \overline{u}_n^{t_n} 
    + \sum_{j=1}^n t_j u_1^{s_1} \dots  u_n^{s_n} \overline{u}_1^{t_1} \dots    \overline{u}_j^{t_j-1}\overline{(\beta_j u_j|u_j|^2)}  \dots  \overline{u}_n^{t_n} \\ \nonumber
     = &\sum_{j=1}^n s_j\beta_j |u_j|^2 u_1^{s_1} \dots   u_j^{s_j}  \dots  u_n^{s_n} \overline{u}_1^{t_1} \dots \overline{u}_n^{t_n}    + \sum_{j=1}^n t_j \overline{\beta}_j|u_j|^2 u_1^{s_1} \dots  u_n^{s_n} \overline{u}_1^{t_1} \dots    \overline{u}_j^{t_j}  \dots  \overline{u}_n^{t_n} \\ \nonumber
    = &\sum_{j=1}^n s_j\beta_j |u_j|^2 R(u)
    + \sum_{j=1}^n t_j \overline{\beta}_j|u_j|^2 R(u) \, .
\end{align}
Hence, by linearity of $[\bullet|| \bullet]$ in the first slot (see Lemma \ref{useful0}), we see that $L^1_k(u)$ is again of the right form.\\

 It follows that $S_k(u) = L^1_k(u) - L^2_k(u)$ can indeed be expressed as a sum of the given monomials terms. This completes the proof.
\end{proof}

\begin{lem}\label{lemnonreso2}
Let $R(u)$ be a monomial and let $k,j \in \{1, \dots, n\}$ be fixed indices. 
The $k$th non-resonance condition of $R(u)$ is satisfied if and only if the $k$th non-resonance condition of $|u_j|^2R(u)$ is satisfied, if and only if the $k$th non-resonance condition of $u_k^2\overline{R(u)}$ is satisfied.
\end{lem}

\begin{proof}
We write
\begin{align}
    R(u) = u_1^{s_1}\dots u_n^{s_n} \overline{u}_1^{t_1} \dots \overline{u}_n^{t_n}\, ,
\end{align}
so that the $k$th non-resonance condition of $R(u)$ is given by
\begin{align}\label{oricompareq}
   s_1\omega_1 + \dots + {s_n}\omega_n - t_1 \omega_1 - \dots - {t_n} \omega_n - \omega_k \not= 0\, .
\end{align}
It follows that the $k$th non-resonance condition of $|u_j|^2R(u)$ is given by
\begin{align}
  \omega_j - \omega_j + &s_1\omega_1 + \dots + {s_n}\omega_n - t_1 \omega_1 - \dots - {t_n} \omega_n - \omega_k \\ \nonumber
  = &s_1\omega_1 + \dots + {s_n}\omega_n - t_1 \omega_1 - \dots - {t_n} \omega_n - \omega_k\not= 0\, ,
\end{align}
which coincides with that of $R(u)$. \\
Likewise, the $k$th non-resonance condition of $u_k^2\overline{R(u)}$ is given by
\begin{align}
  &2\omega_k - s_1\omega_1 - \dots - {s_n}\omega_n + t_1 \omega_1 + \dots + {t_n} \omega_n - \omega_k \\ \nonumber
 =  - &s_1\omega_1 - \dots - {s_n}\omega_n + t_1 \omega_1 + \dots + {t_n} \omega_n + \omega_k \\ \nonumber
  =  - &(s_1\omega_1 + \dots + {s_n}\omega_n - t_1 \omega_1 - \dots - {t_n} \omega_n - \omega_k) \not= 0\, ,
\end{align}
which is again equivalent to Equation \eqref{oricompareq}. This completes the proof.
\end{proof}
Lemmas \ref{lemnonreso1} and \ref{lemnonreso2} guarantee that the $k$th non-resonance condition of $S_k(u) = L^1_k(u) - L^2_k(u)$ is satisfied if the $k$th non-resonance condition of $P_k = \widehat{H}^5_k$ is satisfied. As $\widehat{H}^5_k$ has the same monomial terms as $H^5_k$ (though rescaled), we see that the $k$th non-resonance conditions of $S_k(u)$ are indeed satisfied. Therefore, a solution to Equation \eqref{tosolve283} exists by  Lemma \eqref{solvinfhomological} and may be given by $Q_k = \widehat{S}_k$. Note that this choice of $Q_k$ has lower degree $4$ or higher, as we assumed throughout.

Returning to Equation \eqref{equationalmostthere345}, we finally arrive at
\begin{align}\label{thethetheequation}
 \dot{u}_k = \gamma_ku_k - \beta_k u_k|u_k|^2  - \alpha^2[P_k|| {H}^5](u) + \mathcal{O}(|\alpha||u|^6 + |\alpha|^2|u|^5 + |\alpha|^3|u|^4)  \, .
\end{align}
We have therefore shown:
\begin{proof}[Proof of Theorem \ref{main001} and Proposition \ref{newintterms}]
The calculations in this section show that the successive coordinate transformations 
\begin{align}
  w_k &= z_k - \alpha P_k(z) \\ \nonumber
  u_k &= w_k - \alpha Q_k(w)\,
\end{align}
bring the ODE 
$$\dot{z}_k = \gamma_kz_k - \beta_k z_k|z_k|^2 + \alpha H_k(z)$$  
into the form
$$ \dot{u}_k = \gamma_ku_k - \beta_k u_k|u_k|^2  - \alpha^2[P_k|| {H}^5](u) + \mathcal{O}(|\alpha||u|^6 + |\alpha|^2|u|^5 + |\alpha|^3|u|^4)  \, .$$
If we now set $$G_k(u) := [P_k|| {H}^5] = [\widehat{H}^5_k|| {H}^5]\, ,$$
then we indeed get 
$$ \dot{u}_k = \gamma_ku_k - \beta_k u_k|u_k|^2  - \alpha^2G_k(u) + \mathcal{O}(|\alpha||u|^6 + |\alpha|^2|u|^5 + |\alpha|^3|u|^4)  \, .$$
This completes the proof.
\end{proof}

\newpage

\section{Anomalous synchronization on a 4-node ring}\label{anomalous}

Consider the four node ring network with a coupling function
$
h(z,w) = z \bar w. 
$
\noindent
leading to 
\begin{align}\label{example1_ode}
\dot{z}_1 &= \gamma_1z_1 - \beta z_1|z_1|^2 + \alpha (z_1\overline{z}_2 + z_1\overline{z}_4 )\\ \nonumber
\dot{z}_2 &= \gamma_2z_2 - \beta z_2|z_2|^2 + \alpha (z_2\overline{z}_3 + z_2\overline{z}_1 ) \\ \nonumber 
\dot{z}_3 &= \gamma_3z_3 - \beta z_3|z_3|^2 + \alpha (z_3\overline{z}_4 + z_3\overline{z}_2 ) \\ \nonumber
\dot{z}_4 &= \gamma_4z_4 - \beta z_4|z_4|^2 + \alpha (z_4\overline{z}_1 + z_4\overline{z}_3 ) \, 
\end{align}

We set the parameters $\beta_k = -1$, $\gamma_k = \lambda + i \omega_k$, $\lambda = 1 $ and $\omega_1 = 1+ \delta$, $\omega_2 = 1$ and $\omega_3 = 5$ and $\omega_4 = 6$ for performing the simulation of Eq.~(\ref{example1_ode}). We then vary the mismatch $\delta$ and coupling $\alpha$. Notice that by a naive inspection of the original equations we obtain 
$$
\dot \theta_k = \omega_k + \alpha \sum_{\ell=1}^4 
A_{k \ell} \sin \theta_{\ell} 
$$
thus, instead of a diffusive interaction we would obtain a interaction akin to forcing \cite{Eroglu_2017_sync_chaos}. For each set of $\alpha$ and $\delta$ values, we simulate the network for  50000s with 0.01 time step. We remove first 10000s from  as transient and compute the unwrapped phases. As we are interested in the phase synchronization, we introduce a new variable for the phase differences, $\phi = \theta_1-\theta_2$, and a naive calculation leads to  
$$
\dot \phi = \delta + \alpha[ \sin \theta_2 + \sin \theta_4 - \sin \theta_1 - \sin \theta_3]  + O(\alpha^2). 
$$
Because $\omega_2,\omega_3 \ll \omega_1$, we can average over the fast oscillations and neglect contributions from these phases then  $\theta_2 = \theta_1 + \phi$ we would obtain a interaction term as 
$
\alpha [\sin \theta_1 \cos \phi + \cos \theta_1 \sin \phi] + O(\alpha^2),
$
however, since $\theta_1$ is a fast variable for $\phi$ we also average over $\theta_1$. Thus the whole interaction term linear in $\alpha$ vanishes.  

We calculate the mean synchronization error as
\begin{align}\label{reconstructed_md}
E = \frac{1}{T} \sum_{t=1}^T |\phi(t)|.
\end{align}
The synchronization error for varying $\delta = -0.2$ to $0.2$ with $0.01$ step size and $\alpha = 0.0$ to $0.5$ with step size $0.025$, we observed a synchronization tongue scales with $\alpha \propto \sqrt{\delta}$ (Supplementary Fig.~\ref{fig:sync_tongue}).

\begin{figure}[h]
    \centering
    \includegraphics[width=0.5\columnwidth]{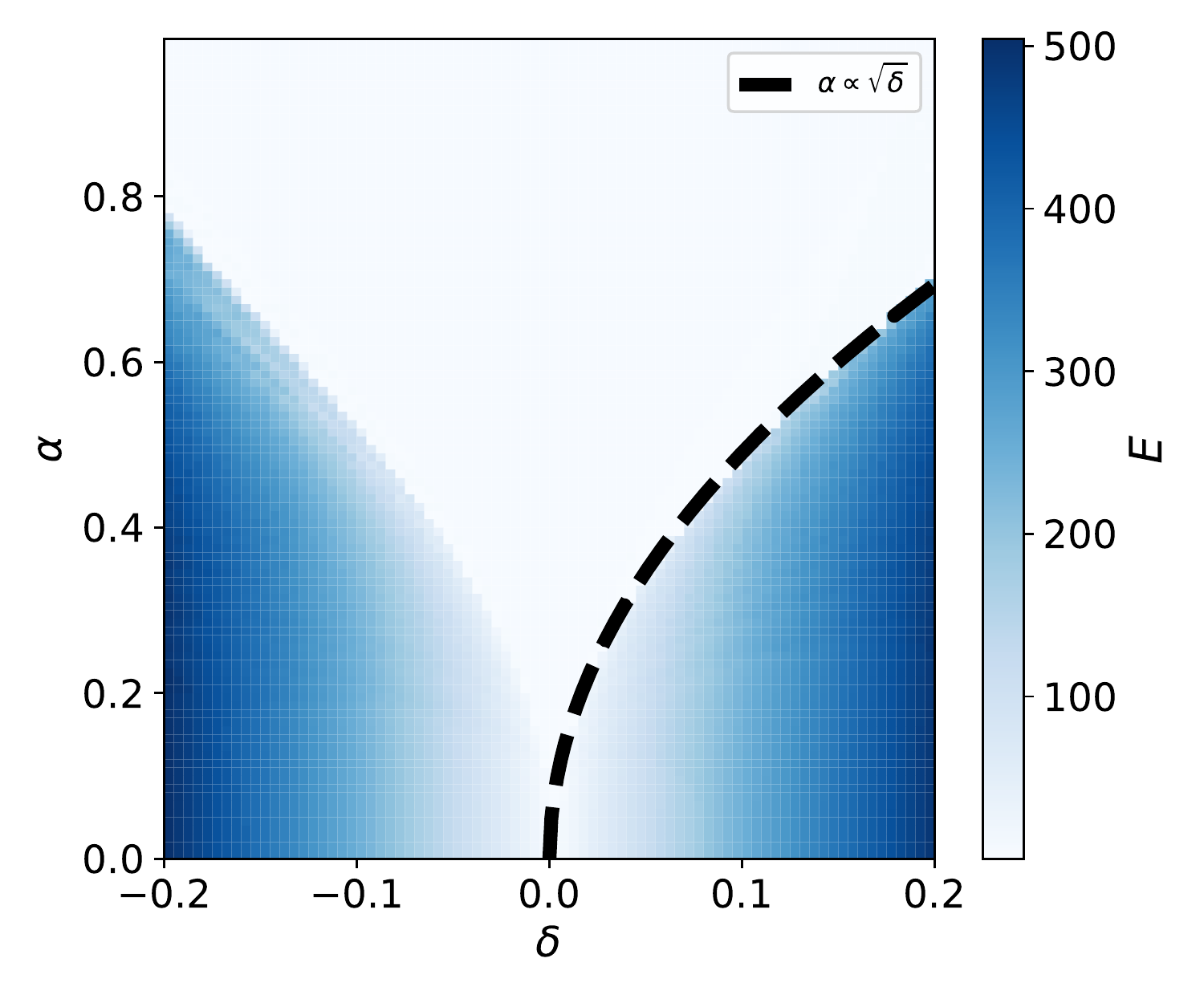}
    \caption{{\bf Arnold tongue represents the phase synchronization between nodes 1 and 2}. The tongue is illustrated for varying mismatch $\delta$ and coupling strength $\alpha$. As the absolute value of $|\delta| \ll 1$, synchronization occurs for small coupling strength. However, if the mismatch is high then also stronger coupling is needed to emerge synchronization. The borderline between synchrony and asynchrony scales with $\alpha \propto \sqrt{\delta}$.}
    \label{fig:sync_tongue}
\end{figure}

\subsection{Explanation of anomalous synchronization via normal form theory}

Recall that we may bring this ODE into the form
\begin{align}\label{repeattt1}
\dot{u}_1 &= \gamma_1u_1 - \beta u_1|u_1|^2 \\ \nonumber
&- \alpha^2 u_1 \left(\frac{\overline{u}_2\overline{u}_4}{\overline{\gamma}_2} + \frac{\overline{u}_4\overline{u}_2}{\overline{\gamma}_4} + \frac{\overline{u}_2\overline{u}_2}{\overline{\gamma}_2} + \frac{\overline{u}_4\overline{u}_4}{\overline{\gamma}_4} + \frac{\overline{u}_2u_3}{\overline{\gamma}_2} + \frac{\overline{u}_2u_1}{\overline{\gamma}_2} + 
\frac{\overline{u}_4u_3}{\overline{\gamma}_4} +
\frac{\overline{u}_4u_1}{\overline{\gamma}_4} \right) + \text{h.o.t.}\\ \nonumber
\dot{u}_2 &= \gamma_2u_2 - \beta u_2|u_2|^2 \\ \nonumber
&- \alpha^2 u_2 \left(\frac{\overline{u}_1\overline{u}_3}{\overline{\gamma}_1} + \frac{\overline{u}_3\overline{u}_1}{\overline{\gamma}_3} + \frac{\overline{u}_1\overline{u}_1}{\overline{\gamma}_1} + \frac{\overline{u}_3\overline{u}_3}{\overline{\gamma}_3} + \frac{\overline{u}_3u_4}{\overline{\gamma}_3} + \frac{\overline{u}_3u_2}{\overline{\gamma}_3} + 
\frac{\overline{u}_1u_4}{\overline{\gamma}_1} +
\frac{\overline{u}_1u_2}{\overline{\gamma}_1} \right) + \text{h.o.t.}\\ \nonumber
\dot{u}_3 &= \gamma_3u_3 - \beta u_3|u_3|^2 \\ \nonumber
&- \alpha^2 u_3 \left(\frac{\overline{u}_2\overline{u}_4}{\overline{\gamma}_2} + \frac{\overline{u}_4\overline{u}_2}{\overline{\gamma}_4} + \frac{\overline{u}_2\overline{u}_2}{\overline{\gamma}_2} + \frac{\overline{u}_4\overline{u}_4}{\overline{\gamma}_4} + \frac{\overline{u}_2u_3}{\overline{\gamma}_2} + \frac{\overline{u}_2u_1}{\overline{\gamma}_2} + 
\frac{\overline{u}_4u_3}{\overline{\gamma}_4} +
\frac{\overline{u}_4u_1}{\overline{\gamma}_4} \right) + \text{h.o.t.}\\ \nonumber
\dot{u}_4 &= \gamma_4u_4 - \beta u_4|u_4|^2 \\ \nonumber
&- \alpha^2 u_4 \left(\frac{\overline{u}_1\overline{u}_3}{\overline{\gamma}_1} + \frac{\overline{u}_3\overline{u}_1}{\overline{\gamma}_3} + \frac{\overline{u}_1\overline{u}_1}{\overline{\gamma}_1} + \frac{\overline{u}_3\overline{u}_3}{\overline{\gamma}_3} + \frac{\overline{u}_3u_4}{\overline{\gamma}_3} + \frac{\overline{u}_3u_2}{\overline{\gamma}_3} + 
\frac{\overline{u}_1u_4}{\overline{\gamma}_1} +
\frac{\overline{u}_1u_2}{\overline{\gamma}_1} \right) + \text{h.o.t.} 
\end{align}
We will assume that $\gamma_1 \approx \gamma_2$. In particular, we consider the possibility that $\gamma_1 = \gamma_2$. In that case we have $\gamma_1 + \overline{\gamma_2} = 2\re(\gamma_1)$, which may be arbitrarily small. Therefore, we may not assume that $\gamma_1 + \overline{\gamma_2} \not= 0$. Other than this, there are no relevant restrictions. I.e., $\gamma_1, \dots, \gamma_4, \gamma_2 - \overline{\gamma_3} \dots \gamma_3 - \overline{\gamma_4}$ are all sufficiently large.
It follows that we may bring equation \eqref{repeattt1} into the form
\begin{align}\label{repeattt2}
\dot{v}_1 &= \gamma_1v_1 - \beta v_1|v_1|^2 - \epsilon\frac{v_1\overline{v}_2v_1}{\overline{\gamma}_2}  + \mathcal{O}(|\epsilon, v|^5)\\ \nonumber
\dot{v}_2 &= \gamma_2v_2 - \beta v_2|v_2|^2 - \epsilon\frac{v_2\overline{v}_1v_2}{\overline{\gamma}_1}  + \mathcal{O}(|\epsilon, v|^5)\\ \nonumber
\dot{v}_3 &= \gamma_3v_3 - \beta v_3|v_3|^2 - \epsilon\frac{v_3\overline{v}_2v_1}{\overline{\gamma}_2}  + \mathcal{O}(|\epsilon, v|^5)\\ \nonumber
\dot{v}_4 &= \gamma_4v_4 - \beta v_4|v_4|^2 - \epsilon\frac{v_4\overline{v}_1v_2}{\overline{\gamma}_1}  + \mathcal{O}(|\epsilon, v|^5)\, ,
\end{align}
where $\epsilon = \alpha^2$. Note that the network topology has changed drastically. Moreover, performing a phase reduction and introducing the phase difference $\psi = \theta_1 - \theta_2$
we obtain

$$
\dot \psi = \delta - c \alpha^2 \sin \psi
$$

where $c$ is a constant depending on $\gamma_1$. By analyzing the fixed points of this equation we obtain the synchronization tongue behavior where the critical coupling $\alpha_c$ for synchronization scales as $\sqrt{\delta}$.

\newpage

\section{Phase reduction for $h = (z^2 +  z) \bar w$ and resonance $\omega_1 - \omega_{2,4} + \omega_3 = 0$}\label{Sec:NumRecovery}

For simplicity we fix $\beta_k = -1$ and obtain
$
\gamma_k = r_0^2 + i \omega_k. 
$ We also introduce $\Delta_{pq} = \omega_p - \omega_q$.  
Note that $r(t) = r_0 + h.o.t$ along with
$
\bar z \dot z = i r_0^2 \dot \theta +h.o.t.
$
and
\begin{eqnarray}
\frac{1}{\gamma_p + \bar \gamma_q} &=& \frac{2r_0^2 - i \Delta_{pq}}{4 r_0^4 + \Delta_{pq}^2}
\end{eqnarray}
Replacing these observations into Eq. (6) of the main manuscript and performing the reduction we obtain the functions in Eq. (7-8) of the main manuscript as

\begin{eqnarray}
\rho_{pq}(\phi)  &=& - \frac{\Delta_{pq}}{4 r_0^4 + \Delta_{pq}^2} \cos \phi +  \frac{2r_0^2}{4 r_0^4 + \Delta_{pq}^2} \sin \phi \nonumber \\
\sigma_{pqr}(\phi) &=& - \chi_{pqr} \sin \phi + \upsilon_{pqr} \cos \phi \nonumber \\
\end{eqnarray}

where 
$$
\chi_{pqr} = r_0^2 \left( \frac{4}{4 r_0^4 +  \Delta_{pq}^2} 
+ \frac{4}{4 r_0^4 +  \Delta_{pr}^2} 
+ \frac{1}{r_0^4 + \omega_q^2} 
+ \frac{1}{r_0^4 + \omega_r^2} \right)
$$
and
$$
\upsilon_{pqr} =  - \frac{ 2\Delta_{pq}}{4 r_0^4 +  \Delta_{pq}^2} 
- \frac{2 \Delta_{pr}}{4 r_0^4 +  \Delta_{pr}^2} 
+ \frac{\omega_q}{r_0^4 + \omega_q^2} 
+ \frac{\omega_r}{r_0^4 + \omega_r^2}
$$

\section{Slow phase dynamics for $h = (z^2 +  z) \bar w$ and resonance $\omega_1 - \omega_{2,4} + \omega_3 = 0$}
We consider networks of $n$ coupled oscillators
\begin{equation}\label{Eq1}
\dot{z}_k = f_k(z_k) + \alpha \sum_{\ell=1}^n A_{k\ell} h_k(z_k, z_{\ell})
\end{equation}
where $z_k \in \mathbb{C}$ is the state of the $k$th oscillator, $f_k :  \mathbb{C} \rightarrow  \mathbb{C}$  is its isolated vector field, $h_k:  \mathbb{C} \times  \mathbb{C} \rightarrow  \mathbb{C}$ is the pairwise coupling function, $ \bm A = ( A_{ij})_{i,j=1}^n$ is the adjacency matrix describing the network structure, and $\alpha > 0 $ is the coupling strength.  We then generate a multivariate time series for a four-node ring network, as illustrated in Figure \ref{fig:networks}~(a), with nonlinear pairwise coupling function
\begin{equation}\label{degree3}
h(z,w) = (z + z^2) \bar w. 
\end{equation} 
\noindent 

\begin{figure}
\centering
\includegraphics[width=.60\columnwidth]{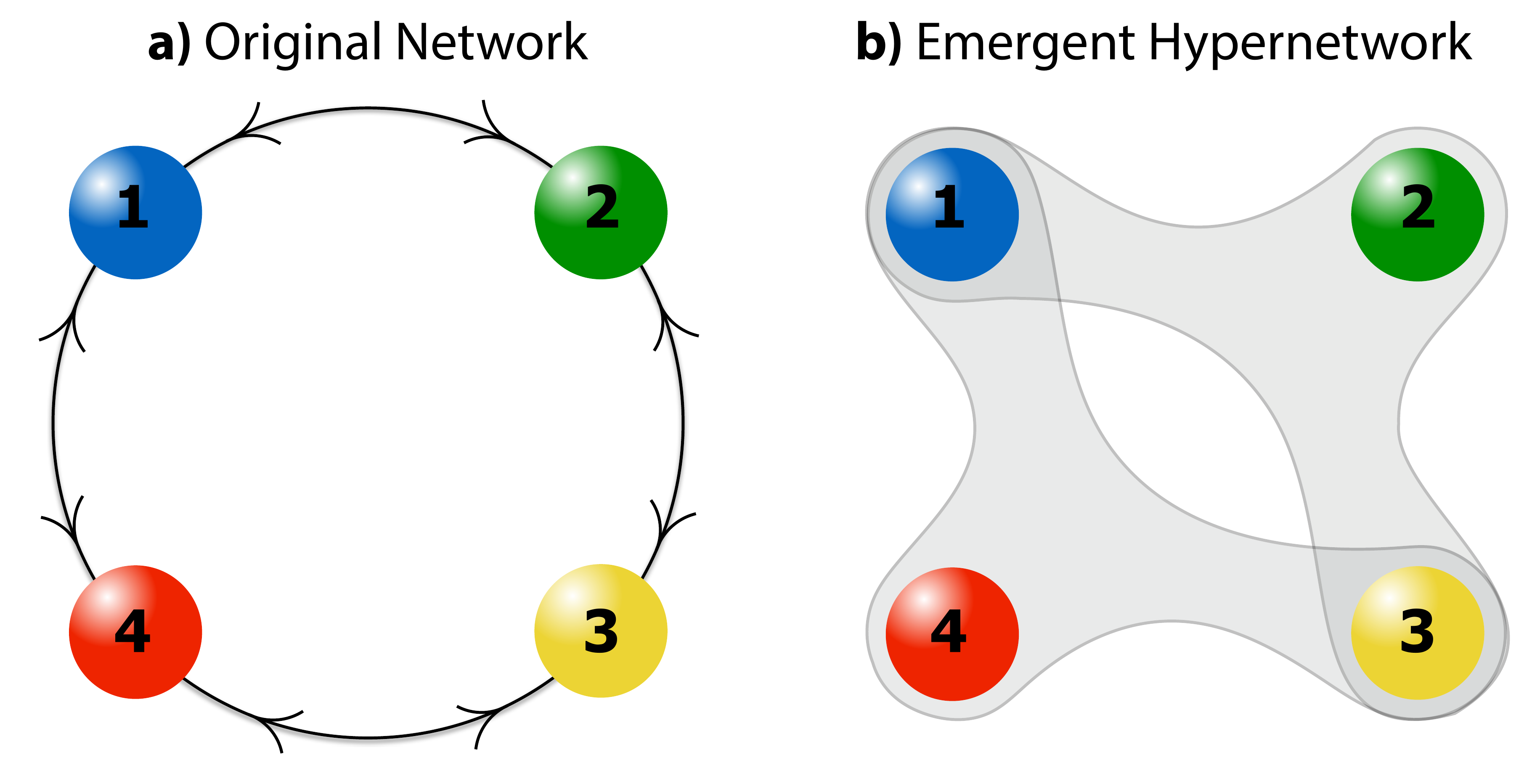}
\caption{{\bf Emergent hypernetworks with triplet interactions.} (a) The original ring network. (b) Hypernetwork recovered from original dynamics data. The state of a given node is influenced by triplet interactions. We show that as along as some sparsity is imposed when the model is obtained from data, only hypernetworks can be found. Our theory predicts the emergence of the hypernetwork determined by the original coupling function $h$, the network, and the resonance relations among the isolated frequencies.}
    \label{fig:networks}
\end{figure}

We fix $\lambda = 0.15 $,  $\omega_1 = 1.01$, $\omega_2 = 2.5$,  $\omega_3 = 1.5$,  $\omega_4 = 2.49$, and $\alpha = 0.18$.  Numerical integration of complex differential equations  is used to solve the differential equations for 10000-sec with 0.01-sec time-step. We discard the first 5000-sec points as transient, and we obtain a multivariate time series  $\{ z_1(t), z_2(t), z_3(t),z_4(t)  \}_{t=1}^{5000}$. 

Next, we aim at obtaining a model from the multivariate time series of $z$. Because $\alpha$ is
small and the isolated orbit is exponentially stable, the amplitude of each time series is slightly affected by $|z_k| \approx \sqrt{\lambda} + O(\alpha)$ as illustrated in Figure.~\ref{fig:4nodes_phi_1}~(a), and the dynamics is captured by the phases $\theta_k(t)$ of $z_k(t)$.   Therefore, we perform a polar decomposition $z_k(t) = r_k(t) e^{i \theta_k(t)}$ to get the unwrapped phase of each time series and obtain governing equations of the model from its phase dynamics.  

Each phase $\theta_k$  has a frequency close to   $\omega_k$, as illustrated in Figure~\ref{fig:4nodes_phi_1}~(b). This means that the growth of the phases  is almost linear with coupling terms as perturbations. 
In fact, the coupling terms generically contain fast variables such as phases $\theta$'s and slow variables involving the resonant combinations of phases such as $\theta_1 - \theta_2  + \theta_3$ that change slowly in time.
Therefore, we subtract the linear growth of the phases to analyse the effects of the coupling. To this end, we  introduce 
\begin{equation}\label{new_phases}
\vartheta_k(t) = \theta_k(t) - \Omega_k t, 
\end{equation} 
\noindent 
where $\Omega_k$ is obtained from data under the resonance condition $\Omega_1 -  \Omega_{i} +  \Omega_3 = 0$, with $i = 2,4$. In the new phases the coupling has the same magnitude as the frequency mismatch $\omega_1 - \omega_i + \omega_3$, with $i=2,4$. Finally,  we obtain a model for $\vartheta_k$.
We assume the model  
\[
\dot \vartheta_k = \varepsilon_k + H_k(\vartheta_1, \vartheta_2,\vartheta_3,\vartheta_4)
\]
\noindent
where $H_k = \sum [ c^k_p \sin \vartheta_p + d_p^k \cos \vartheta_p ]  + \sum [ c^k_{p,q} \sin (\vartheta_p - \vartheta_q) + d^k_{p,q} \cos (\vartheta_p - \vartheta_q ] 
+ \sum [ c^{k}_{pq} \sin (\vartheta_p + \vartheta_q - \vartheta_k) + d^{k}_{pq} \cos (\vartheta_p + \vartheta_q - \vartheta_k)] $. Note that this includes pairwise and triplet interactions. We solve for the coefficients to obtain the least square approximation and we impose sparsity by eliminating coefficients below a threshold $\tau = 10^{-4}$. The technique is discussed along with the package to perform the recovery as discussed in the main text.  The model recovery yields
\begin{eqnarray}
\label{recovered_eq}
\dot{\vartheta}_{1,3} &=& \varepsilon_{1,3} + r_{1,3}(\vartheta_1,\vartheta_2,\vartheta_3) + s_{1,3}(\vartheta_1,\vartheta_4, \vartheta_3) \\
\dot{\vartheta}_{2,4} &=&  \varepsilon_{2,4} + r_{2,4}(\vartheta_1,\vartheta_{2,4},\vartheta_3) \nonumber
\end{eqnarray}
where  $s$ and $r$ correspond to triplets in $H_k$ with nonzero coefficients. 

At first sight, the model recovery with triplets is remarkable because the original equations have only pairwise interactions. Nonetheless, a hypernetwork describes the data Figure~\ref{fig:networks}~(b). 
We show that when $\lambda \ll1$ and $\alpha \ll 1$  recovering a hypernetwork from data is  not a coincidence. As long as the coupling $h$ is nonlinear, by measuring the original variables of Eq. (\ref{Eq1}) and performing a sparse model recovery only hypernetworks can be found as they are normal forms of the original equations.

\begin{figure}[h]
 \centering
    \includegraphics[width=0.5\columnwidth]{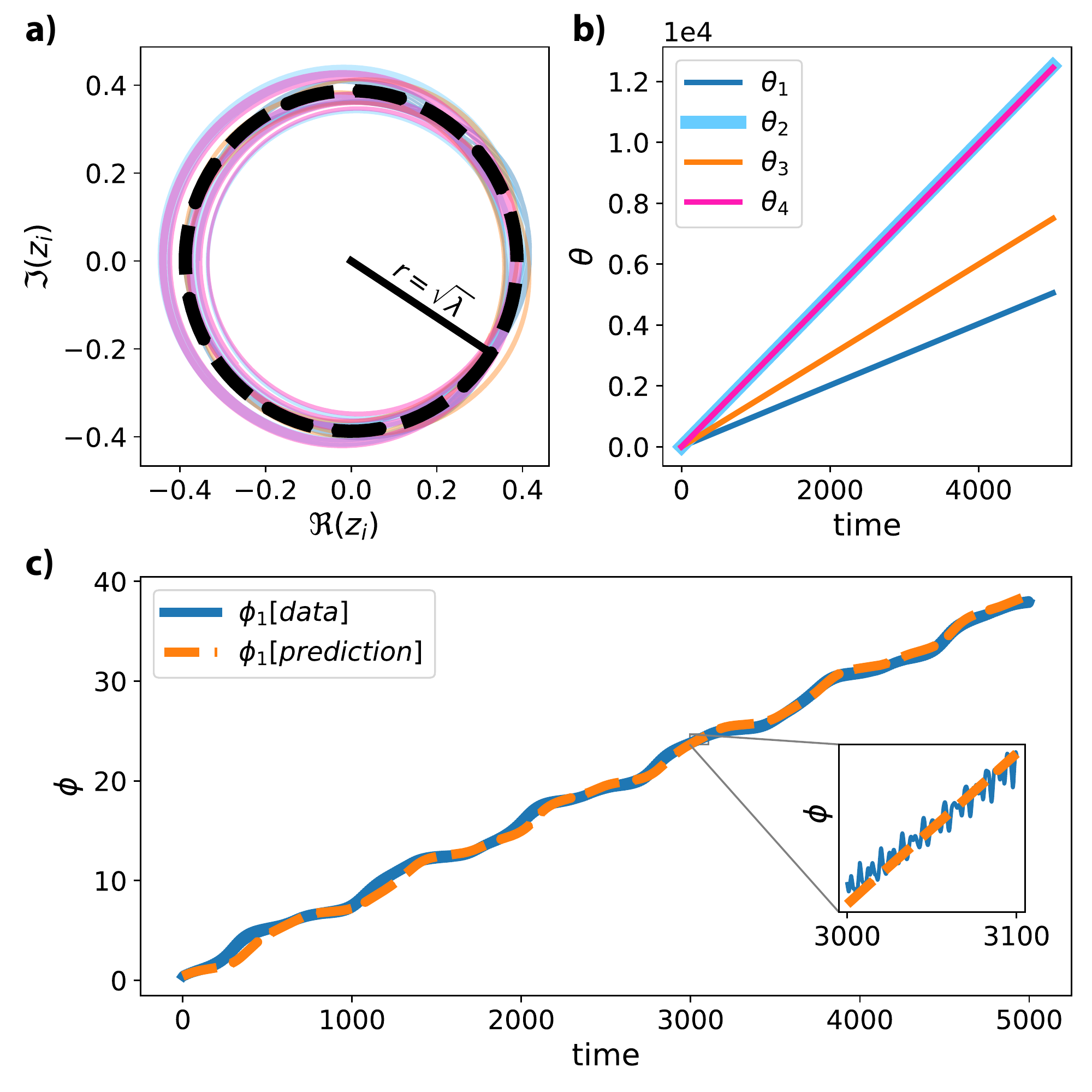}
    \caption{{\bf Time series and  emergent hypernetwork prediction.} The simulation was performed on a ring network (see Supplementary Fig.~\ref{fig:networks}~(a)). (a)   Amplitudes (solid lines) are slightly affected by the coupling and remaining close to a circle with radius $r = \sqrt{\lambda}$ (dashed circle). (b) Unwrapped phases $\theta_i$ growth. (c) Time series of the slow phase $\phi_1$ from data (solid) and the prediction of the emergent hypernetwork (dashed) capturing higher-order interactions  (see Supplementary Fig.~\ref{fig:networks}~(b)).
}
    \label{fig:4nodes_phi_1}
\end{figure}

\subsection{Emergent hypernetwork predicts data behaviour}
To illustrate prediction capabilities of emergent hypernetworks, we introduce the slow phases 
\begin{eqnarray}\label{slow_phases}
\phi_1 &=& \theta_1  - \theta_2 + \theta_3,\\ \nonumber
\phi_2 &=& \theta_1 - \theta_4 + \theta_3,
\end{eqnarray}
where the coupling strength is comparable to the frequency mismatch $\omega_1 - \omega_2 + \omega_3$. Our normal form theory predicts the emergent hypernetwork phase dynamics described in Eq. (15).   
We obtain the vector fields  of the coupled slow phases $\phi_1$ and $\phi_2$  analytically as described in the section above. 

Next, we simulate the vector fields obtained from first principles  using an adaptative Runge-Kutta method of 4th order. We treat the initial condition as unknown and perform a optimization to obtain the the initial condition that provides the minimum least square error between the data of the slow phase and the simulations and theory.  In Figure \ref{fig:4nodes_phi_1} c), we compare our predictions and slow phases estimated from data.
The theoretical prediction is in excellent agreement with the data with an error in the prediction of less than $5\%$ per cycle of the slow phase.

\section{Model recovery of a 3-path with coupling $h = (z^2 +  z)  \bar w$ and resonance $\omega_1 - \omega_2 + \omega_3 = 0$}\label{numerics4}
Now we consider the model for 3-nodes on a chain (Supplementary Fig.~\ref{fig:3nodes}) that reads as
 We consider the network ODE
\begin{align}\label{path}
\dot{z}_1 &= \gamma_1z_1 - \beta z_1|z_1|^2 + \alpha (z_1\overline{z_2} + z_1^2\overline{z_2} )\\ \nonumber
\dot{z}_2 &= \gamma_2z_2 - \beta z_2|z_2|^2 + \alpha ([z_2\overline{z_1} + z_2^2\overline{z_1} ] + [z_2\overline{z_3} + z_2^2\overline{z_3} ])\\ \nonumber
\dot{z}_3 &= \gamma_3z_3 - \beta z_3|z_3|^2 + \alpha (z_3\overline{z_2} + z_3^2\overline{z_2} ) \,  ,
\end{align}

\begin{figure}[!ht]
    \centering
    \includegraphics[width=0.8\columnwidth]{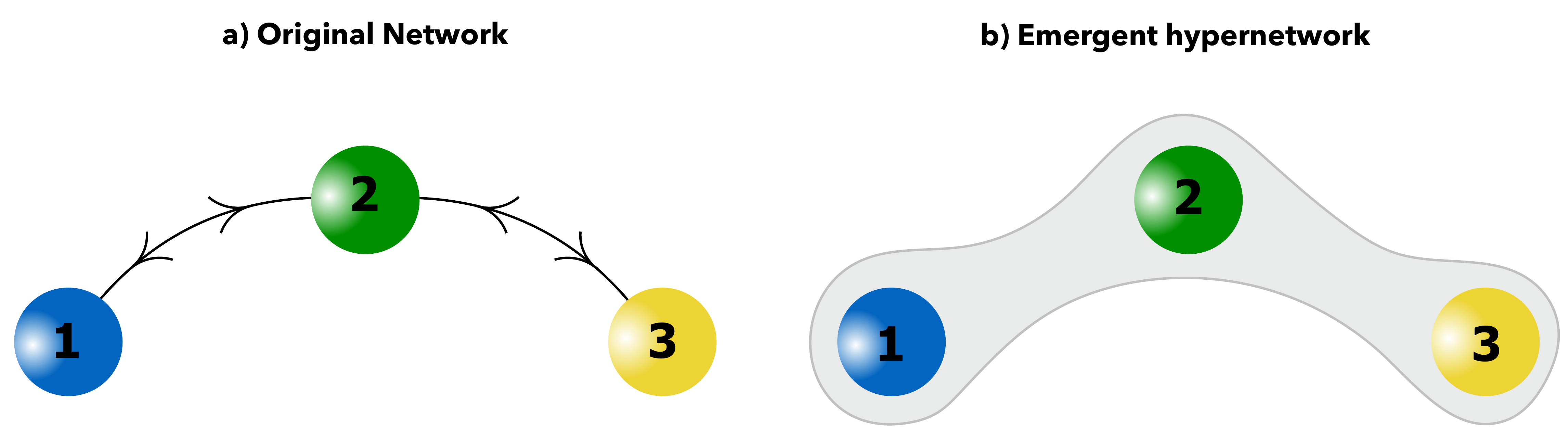}
    \caption{{\bf Emergent hypernetworks with triplet interaction.} Inset a) shows the original chain network. Each isolated node dynamics is close to a Hopf bifurcation. The pairwise coupling function $h$ is nonlinear given with $h = (z^2 + \bar z)w$. Inset b) shows the hypernetwork learnt from the phase dynamics data of the original dynamics. The state of a given node is influenced by the triple interaction of its own state in combination of the incoming links. Our theory also predicts the emergent of such hypernetwork. The hypernetwork  emerges as a combination of the original coupling function $h$, the network, and  the resonance relations of the isolated frequencies.}
    \label{fig:3nodes}
\end{figure}
To integrate Eq.~(\ref{path}) of the main manuscript for 3-node chain with  $\omega_1 = 1.01$, $\omega_2 = 2.5$ and $\omega_3 = 1.5$, we employed a wrapper of ODEPACK routine. Numerical integration for $\alpha = 0.18$ and $\delta = 0.01$ was performed for 10000s with 0.01 time step. We discard the first 5000s points as transient. Using the simulated phases $\theta_i$, we introduce new phases in Eq.~(3) of the main manuscript where $\Omega_1 = 1.0$, $\Omega_2 = 2.5$, $\Omega_3 =  1.5$. Applying the sequential thresholded least-squares method  on these new phases $\vartheta_i$ with thresholding parameter $\lambda = 10^{-4}$ we obtain

\begin{align}
\dot{\vartheta}_1 &= \,\,\,\,\,\,\,\,0.01 - 0.001\cos(\vartheta_1 - \vartheta_2 + \vartheta_3) \\
\dot{\vartheta}_2 &= -0.001  + 0.005 \cos(\vartheta_1 - \vartheta_2 + \vartheta_3) \\
\dot{\vartheta}_3 &= -0.001 - 0.001 \cos(\vartheta_1 - \vartheta_2 + \vartheta_3) 
\end{align}

Because the norms of the functions $r_k$ and $s_k$ are small, we introduce the slow phases 
\begin{eqnarray}\label{new_phi}
\phi &=& \theta_1  - \theta_2 + \theta_3 \nonumber
\end{eqnarray}

We then also perform a reconstruction for the slow phases $\phi$ using the same method  and obtain 

\begin{align}\label{eq:3nodes_recon}
\dot{\phi} = 0.010+ 0.001\sin(\phi) - 0.006\cos(\phi) 
\end{align}

We show the  model prediction and data for the slow phase in Figure \ref{fig:3nodes_phi}.
\begin{figure}[h]
    \centering
    \includegraphics[width=0.5\columnwidth]{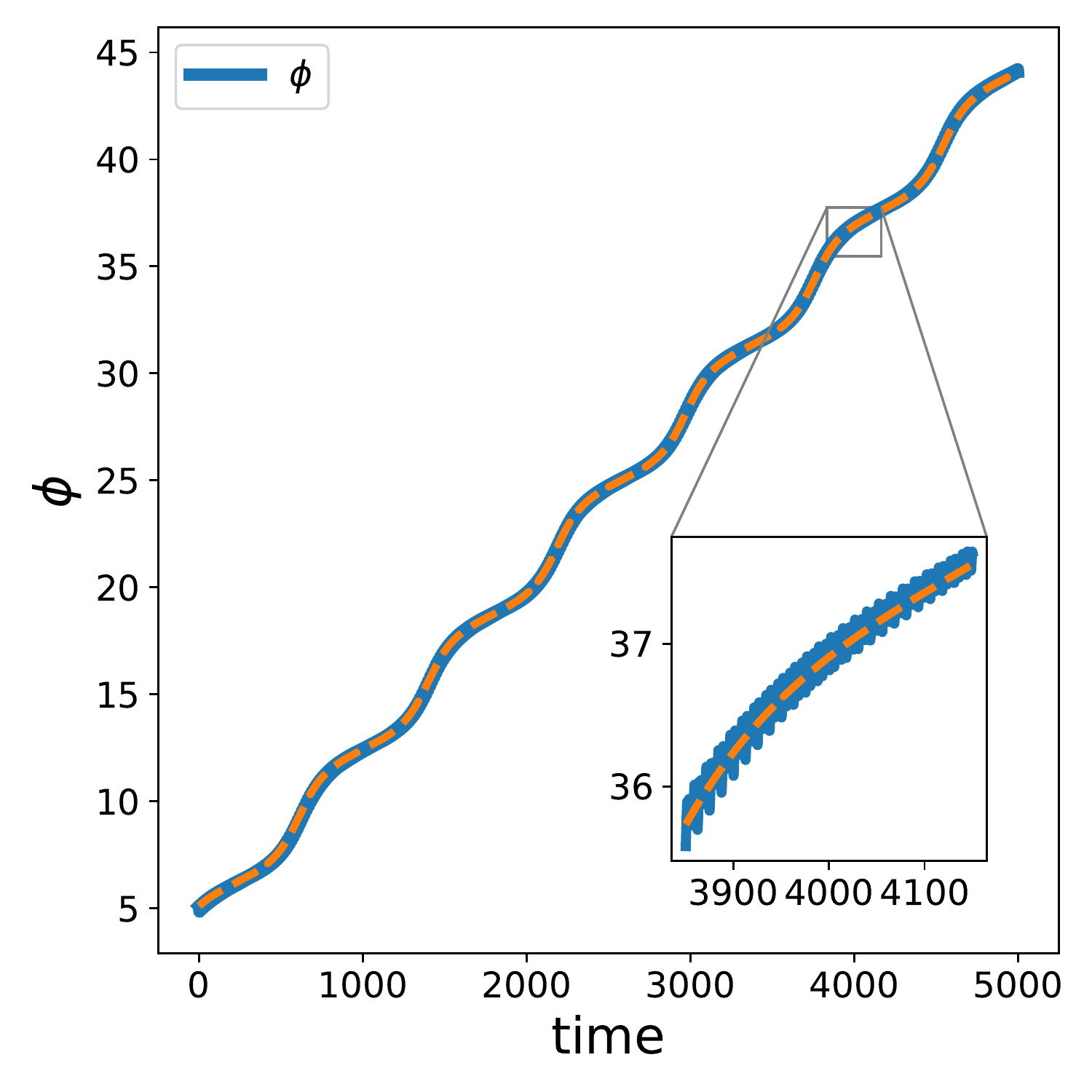}
    \caption{ New slow phase variable  $\phi$ (blue curve) were computed from the data  collected from the simulations of Eq.~(1) on a path. Then Eq.~\ref{eq:3nodes_recon} (orange dashed curve) reconstructed from data of $\phi$ using sequential thresholded least squares method.}
    \label{fig:3nodes_phi}
\end{figure}

\subsection{Emergent network explanation}

The normal-form for this system is given by
\begin{align}
\dot{u}_1 &= \gamma_1u_1 - \beta u_1|u_1|^2 - \alpha^2 \left(\frac{1}{\gamma_1 + \overline{\gamma}_2}\right)u_1^2\overline{u}_2u_3\\ \nonumber
\dot{u}_2 &= \gamma_2u_2 - \beta u_2|u_2|^2 -  \alpha^2 \left(\frac{2}{\gamma_2 + \overline{\gamma}_3} + \frac{2}{\gamma_2 + \overline{\gamma}_1} + \frac{1}{\overline{\gamma}_3} + \frac{1}{\overline{\gamma}_1}\right)u_2^2\overline{u}_1\overline{u}_3\\ \nonumber
\dot{u}_3 &= \gamma_3u_3 - \beta u_3|u_3|^2 -  \alpha^2 \left(\frac{1}{\gamma_3 + \overline{\gamma}_2}\right)u_3^2\overline{u}_2u_1\,  .
\end{align}
The interaction now becomes forth order in $u$. A phase reduction leads to the triplet interaction recovered numerically. 

\newpage

\section{6 nodes network examples with $h(z,w) = z \bar w$}\label{6n}
We consider the network presented in Figure \ref{fig:networks6n}a) with the coupling  $h(z,w) = z \bar w$ leading to 
\begin{figure}[!ht]
    \centering
    \includegraphics[width=0.65\columnwidth]{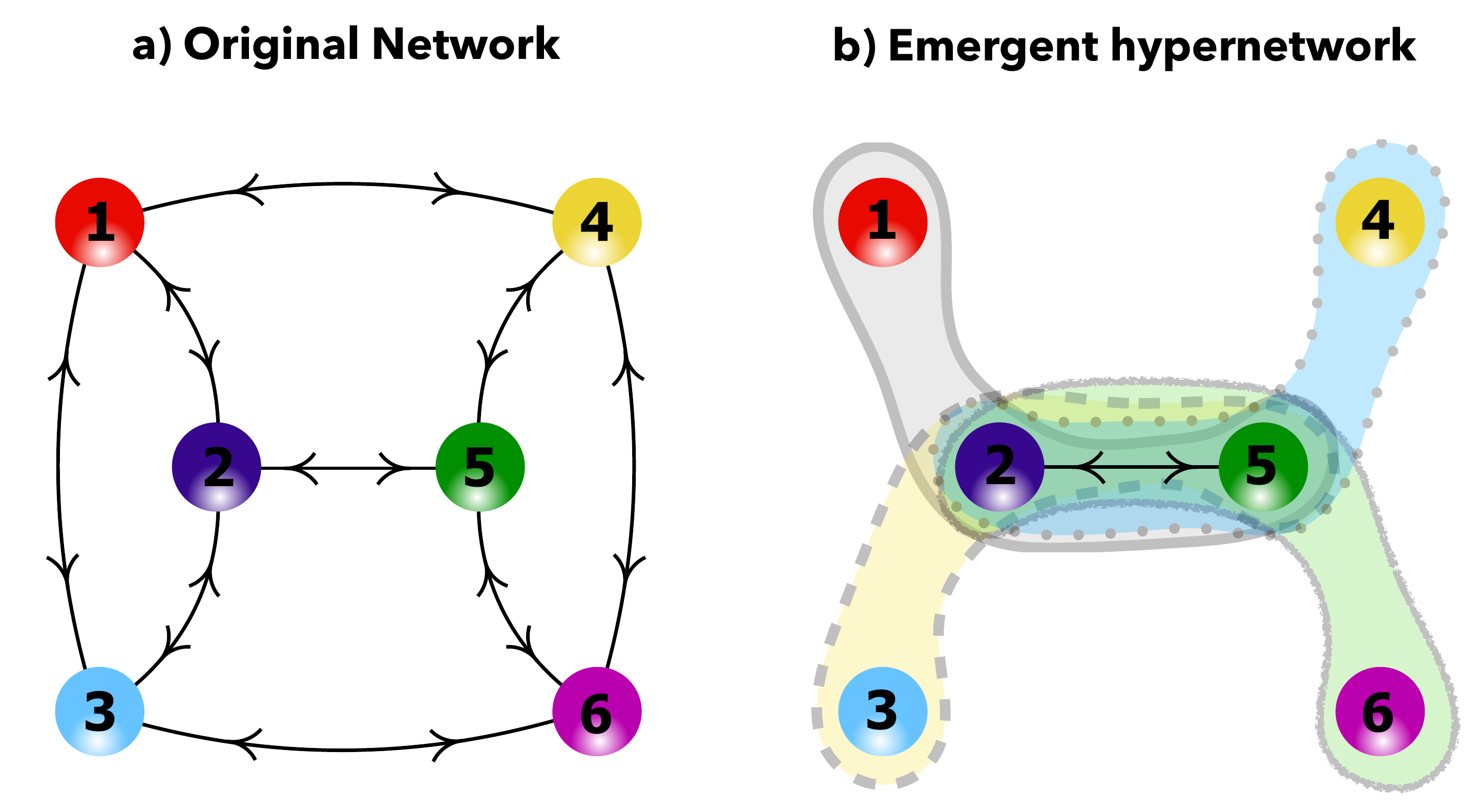}
    \caption{{\bf Emergent hypernetworks with triplet interaction.} Inset a) shows the original network. Each isolated node dynamics is close to a Hopf bifurcation. The pairwise coupling function $h(z,w) = z \bar w$. Inset b) shows the hypernetwork reconstructed from the phase dynamics of the original dynamics. The state of a given node is influenced by the triple interaction of its own state in combination of the incoming links. Our theory also predicts the emergent of such hypernetwork. The hypernetwork  emerges as a combination of the original coupling function $h(z,w) = z \bar w. $, the network, and  the resonance relations of the isolated frequencies.}
    \label{fig:networks6n}
\end{figure}

\begin{align}\label{exxx2}
\dot{z}_1 &= \gamma_1z_1 - \beta z_1|z_1|^2 + \alpha (z_1\overline{z}_2 + z_1\overline{z}_3 +  z_1\overline{z}_4)\\ \nonumber
\dot{z}_2 &= \gamma_2z_2 - \beta z_2|z_2|^2 + \alpha (z_2\overline{z}_1 + z_2\overline{z}_3 +  z_2\overline{z}_5)\\ \nonumber
\dot{z}_3 &= \gamma_3z_3 - \beta z_3|z_3|^2 + \alpha (z_3\overline{z}_1 + z_3\overline{z}_2 +  z_3\overline{z}_6)\\ \nonumber
\dot{z}_4 &= \gamma_4z_4 - \beta z_4|z_4|^2 + \alpha (z_4\overline{z}_5 + z_4\overline{z}_6 +  z_4\overline{z}_1)\\ \nonumber
\dot{z}_5 &= \gamma_5z_5 - \beta z_5|z_5|^2 + \alpha (z_5\overline{z}_4 + z_5\overline{z}_6 +  z_5\overline{z}_2)\\ \nonumber
\dot{z}_6 &= \gamma_6z_6 - \beta z_6|z_6|^2 + \alpha (z_6\overline{z}_4 + z_6\overline{z}_5 +  z_6\overline{z}_3)\, .
\end{align}
We will assume either one of
\begin{enumerate}
\item $\gamma_2 \approx \gamma_5$,  $\gamma_1 \approx \gamma_6$ and $\gamma_3 \approx \gamma_4$, with $\gamma_1 \not\approx \gamma_2$ (and hence $\gamma_1, \gamma_6 \not\approx \gamma_2, \gamma_5$), $\gamma_2 \not\approx \gamma_3$ and $\gamma_1 \not\approx \gamma_3$. 
\item $\gamma_2 \approx \gamma_5$ and $\gamma_i \not\approx \gamma_j$ for all $i,j \in \{1, \dots, 6\}$ with $i \not= j$ and $(i,j) \not= (2,5), (5,2)$.
\end{enumerate}
In each case, we may transform ODE \eqref{exxx2} into
\begin{align}\label{transform2}
\dot{v}_1 &= \gamma_1v_1 - \beta v_1|v_1|^2 - \epsilon\frac{v_1\overline{v}_2v_5}{\overline{\gamma}_2}  + \mathcal{O}(|\epsilon, v|^5)\\ \nonumber
\dot{v}_2 &= \gamma_2v_2 - \beta v_2|v_2|^2 - \epsilon\frac{v_2\overline{v}_5v_2}{\overline{\gamma}_5}  + \mathcal{O}(|\epsilon, v|^5)\\ \nonumber
\dot{v}_3 &= \gamma_3v_3 - \beta v_3|v_3|^2 - \epsilon\frac{v_3\overline{v}_2v_5}{\overline{\gamma}_2}  + \mathcal{O}(|\epsilon, v|^5)\\ \nonumber
\dot{v}_4 &= \gamma_4v_4 - \beta v_4|v_4|^2 - \epsilon\frac{v_4\overline{v}_5v_2}{\overline{\gamma}_5}  + \mathcal{O}(|\epsilon, v|^5)\\ \nonumber
\dot{v}_5 &= \gamma_5v_5 - \beta v_5|v_5|^2 - \epsilon\frac{v_5\overline{v}_2v_5}{\overline{\gamma}_2}  + \mathcal{O}(|\epsilon, v|^5)\\ \nonumber
\dot{v}_6 &= \gamma_6v_6 - \beta v_6|v_6|^2 - \epsilon\frac{v_6\overline{v}_5v_2}{\overline{\gamma}_5}  + \mathcal{O}(|\epsilon, v|^5)\, ,
\end{align}
where $\epsilon = \alpha^2$ leading to the emergent higher order network displayed in Figure \ref{fig:networks6n}b).

\section{Model recovery and normal form representation}\label{Data}
Let $x\in \mathbb{R}^m$ and consider 

\begin{eqnarray}\label{Rx}
\dot x &=& F(x)  
\end{eqnarray}

We assume for simplicity that $F: \mathbb{R}^m \rightarrow \mathbb{R}^m$ is a polynomial map and  

\begin{eqnarray}\label{Rx}
\dot x_i&=& \sum_{j=1}^k a_{i j} p_j(x),
\end{eqnarray}

\noindent
where $x_i$ is the $i$th coordinate of $x$ and $p_j$'s form a basis of homogeneous polynomials. 
Notice that in a network context, $x\in \mathbb{R}^m$ would represent the state vector of the network and $F$ would model isolated dynamics and interactions. Once a trajectory $x(t)$ and $\dot x(t)$ are known, we perform a model recovery as follows. Fix a sampling $h$ and introduce 

$$
 V = \left(
\begin{array}{cccc}
\dot x_1(0) & \dot x_2(0) & \dots & \dot x_m(0)\\
\dot x_1(h) & \dot x_2(h) & \dots & \dot x_m(h)\\
\vdots & \vdots& \ddots & \vdots \\
\dot x_1(T) & \dot x_2(T) & \dots & \dot x_m(T)\\
\end{array}
\right) \mbox{~ ~ and ~ ~} 
X = \left(
\begin{array}{cccc}
 x_1 (0) & x_2 (0)& \dots & x_m (0)\\
 x_1 (h) & x_2 (h)& \dots & x_m (h)\\
\vdots & \vdots& \ddots & \vdots \\
x_1 (T) & x_2 (T)& \dots & x_m (T)\\
\end{array}
\right)
$$

along with 

\[
\Phi(X) =
\left(
\begin{array}{cccc}
p_1(x(0)) 	   & p_2(x(0)) & \cdots & p_k(x(0)) \\
p_1(x(h)) 	   & p_2(x(h)) & \cdots & p_k(x(h)) \\
\vdots & \vdots & \ddots & \vdots \\ 
p_1(x(T)) 	   & p_2(x(T)) & \cdots & p_k(x(T)) \\
\end{array}
\right) \nonumber
\]

Let $v_i$ be the $i$th column of $V$ and $\xi_i = (a_{i1}, a_{i2}, \dots, a_{ik})^*$. Here, $^*$ denotes the transpose. Then by construction

\begin{equation}\label{fund}
\Phi (X) \xi_i = v_i
\end{equation}
\noindent
and if for large $T$ the operator $\Phi$ is full rank the solution of Eq. (\ref{fund}) is unique. 
Solving this equation for all coordinates, we recover the differential equation. 

In data, however, due to numerical round-off errors or noise Eq. (\ref{fund}) is perturbed and one seeks for solutions allowing a small error $\| \Phi (X) \xi_i - v_i \| <\varepsilon_0$ but  under a model simplification such as imposing that some coefficients of $\xi_i$ are zero, that is, looking for sparse solutions. 

The sparse model recovery of the coefficients $\xi_i$ is the problem 

$$
\min_{q \in \mathbb{R}^k} \| q \|_0 \mbox{~ ~ subjected to ~ ~} \| \Phi(X) q -  v_i \|\le\lambda
$$
\noindent
for a suitably chosen $\lambda>0$.

Now we are ready to prove the following

\begin{thr}
Consider Eq. (\ref{Rx}) and the following assumptions
\begin{itemize}
\item[](H0) Eq. (\ref{Rx}) is generic (coefficients $a_{ij}$ are non vanishing)
\item[](H1) Eq (\ref{Rx}) has a normal form
\begin{eqnarray}
\dot y &=& G(y) + R(y)  
\end{eqnarray}
where $G$ contains no non-resonant terms and  $\|R(y)\| = O(y^{d+1})$ for some large $d$. Moreover the coordinates of $y$ have the expansion 
\[
\dot y_i 
= \sum_{j=1}^k b_{ij} p_j(y) + R_i(y)
\]

\item[](H2) The trajectories $\{x(t)\}_{t=0}^T$ as well as $\{\dot x(t)\}_{t=0}^T$
are given with $T$ sufficiently large and stay in a sufficiently small neighbourhood $V_{\varepsilon}$ of the origin such that 
$$
\sup_{t\in[0,T]} \| x(t) \|_{C^1} \le \varepsilon 
$$
for initial conditions in an open neighbourhood of the origin. 

\item[](H3) The operator $\Phi$ is full rank. 
\end{itemize}
Then there exist  $\lambda = \lambda (\varepsilon,d)>0$  such the solution to the sparse recovery problem 
$$
\min_{q \in \mathbb{R}^k} \| q \|_0 \mbox{~ ~ subjected to ~ ~} \| \Phi(X) q -  v_i \|_2 \le \lambda 
$$
is the vector of coefficients of $(b_{i1},b_{i2}, \dots, b_{ik})^*$ of  the normal form of Eq. (\ref{Rx})
\end{thr}

\noindent
\begin{proof} 
We break the arguments into three steps:

{\it Step1: Approximations and Uniqueness solutions.} By normal form theory there are functions $Q_1$ and $Q_2$ such that 
\begin{equation}\label{Pnf}
x  = y + Q_1(y) \mbox{~~ and ~~} y = x+ Q_2(x)
\end{equation}
where $\| Q_1(y) \| = O(\|y\|^{2})$ and $\| Q_2(x) \| = O(\|x\|^2)$. 
Given a trajectory $y(t)$ we construct the matrix $Y$ in the same manner as $X$ and consider 
$$
 u_i = \left(
\begin{array}{c}
\dot y_i(0) \\
\dot y_i(h) \\
\vdots \\
\dot y_i (T)
\end{array}
\right) \mbox{~ and ~ }
 \rho_i (Y) = \left(
\begin{array}{c}
 R_i(y(0)) \\
 R_i(y(h)) \\
\vdots \\
 R_i(y(T))
\end{array}
\right)
$$ 
As the basis is formed by homogeneous polynomials, using Eq. (\ref{Pnf}) we conclude that there is $L$ such that 
$$
\| \Phi(X) - \Phi(Y)  \|_2 \le L \varepsilon^{2}
$$
  
By $(H3)$  $\Phi(X)$ is full rank and  for  $\varepsilon^{2}$ small enough, we conclude that $\Phi(Y)$ is also full rank since the rank is lower semicontinuous.  Next notice that the equation 
\begin{equation}\label{Y}
\Phi(Y) \zeta_i + \rho_i(Y) =  u_i
\end{equation}
also has a solution $\zeta_i = (b_{i1}, b_{i2}, \dots, b_{ik})^*$ by construction and it is unique since $\Phi(Y)$ is full rank. Furthermore, in $V_{\varepsilon}$ there is a constant $M$ such that 
$$
\| \rho_i(Y)\|_2 \le M \varepsilon^{2} ,\,\,\,\, \forall i \in \{1,\dots, m\}
$$

Using Eq. (\ref{Pnf})  we obtain  
\begin{eqnarray}\label{uv}
v_i =  u_i + z_i
\end{eqnarray}
where $z_i$ corresponds to terms as  $DP(y) \dot y$.  By $(H2)$ trajectories stay in the neighbourhood $V_{\varepsilon}$, thus,  there is a constant $C$ such that 
$$
\| z_i \|_2 \le C \varepsilon^{2},\,\,\,\, \forall i \in \{1,\dots, m\}
$$

{\it Step 2: A sparse solution.} Consider the unique solution $\zeta_i$ of Eq. (\ref{Y}) and let  $\sigma_i = \|\zeta_i \|_0$. Consider the set
$$
B_{\lambda, \sigma_i} = \{ q \in \mathbb{R}^k: \| q\|_0 \le \sigma_i \mbox{~and ~} 
\| \Phi(X)q - v_i \|_2 \le \lambda \}
$$
Now we claim that if $\lambda:= (L \| \zeta_i\|_2 + C+ M) \varepsilon^{d}$ then $\zeta_i \in B_{\lambda, \sigma_i}$. Indeed, consider 
\begin{eqnarray}
\| \Phi(X) \zeta_i - v_i \|_2 &=& \| \Phi(X) \zeta_i + \Phi(Y) \zeta_i - \Phi(Y) \zeta_i- v_i \|_2 \nonumber \\
&=& \| [ \Phi(X) - \Phi(Y) ] \zeta_i + \Phi(Y) \zeta_i- u_i - z_i +\rho_i - \rho_i \|_2 \nonumber \\
&=& \| [ \Phi(X) - \Phi(Y) ] \zeta_i- z_i - \rho_i \|_2 \nonumber \\
&\le & \| \Phi(X) - \Phi(Y) \|_2 \| \zeta_i \|_2 +  \|z_i\|_2 +  \|\rho_i \|_2 \nonumber \\
&\le& (L  \| \zeta_i \|_2 + M + C) \varepsilon^2 \nonumber
\end{eqnarray}

{\it Step 3: Uniqueness.}
Assume that there is $\eta \in B_{\lambda, \sigma_i}$ with $\| \eta\|_0 < \sigma_i$. Since $\Phi(X)$ is full rank, this implies that there is 
$\hat R$ such that $\| \hat R(x) \|_2 \le K \| x\|_2^{d+1}$ for some $K$ and $\dot x = \hat G(x) + \hat R(x)$. Thus, $\hat G$ has fewer coefficients than $G$, implying that either $G$ must have a non-resonant term or $(H0)$ was violated.  This contracts $(H1)$ and completes the proof. \\
\end{proof}

\begin{remk} Assumption $H2$ is natural in our context. Notice since the isolated system has a  limit cycle near the origin. Thus,  an open set of initial conditions is attracted to the cycles and stays for all times near the origin where we control the norm of solutions \cite{Stankovski_RMP_2017}. When coupling such dynamics to a network this behaviour persists. \\
\end{remk}

\begin{remk}
Assumption $H3$ is in general not restrictive. If solutions of $\dot x=F(x)$ are not degenerated such as all solutions converge to fixed points, then typically $\Phi(X)$ is full rank. In fact, if solutions converge to an attractor, we can adapt the basis to the dynamics such that in the adapted basis $\Phi^*(X) \Phi(X)$  is close to identity for large $T$ \cite{PhysicaD}. This implies that $\Phi(X)$ is close to orthogonal.  \\
\end{remk}

\begin{remk}
Another interesting case is when 
$$
\dot x = F(x) + U(t,x)
$$
and $U$ has fast oscillations. This happens typically in phase dynamics when we subtract the trends of linear frequencies. For example, consider 
\begin{eqnarray}
\dot \theta &=& 1  + \epsilon \sin(\theta - \phi ) + \epsilon \cos( \theta +\phi) \\
\dot \phi &=& 1+\delta  + \epsilon \sin(\phi-\theta ) + \epsilon \cos( \theta +\phi) 
\end{eqnarray}
where $\delta\ll1$. 
Subtracting the trend $\vartheta = \theta - t$ and $\varphi = \phi -t$ leads to 
\begin{eqnarray}
\dot \vartheta &=&  \epsilon \sin(\vartheta - \varphi ) + \epsilon \cos(\vartheta +\varphi - 2t ) \\
\dot \varphi &=& \delta  + \epsilon \sin(\varphi-\vartheta ) + \epsilon \cos(\vartheta +\varphi - 2 t) 
\end{eqnarray}
Since $\dot \varphi $ and $\dot \vartheta$ are $O(\delta)$ by the averaging Theorem, fast oscillating terms containing $\cos$ are averaged out and can be neglected in a time scale as $1/\delta$. Thus, also in this case when performing a model recovery with finite amount of data the function $U$ cannot be recovered. This also happens for our examples in the main text. Thus, sparsity and fast oscillations can contribute to the impossibility of recovering the original model. 
\end{remk}
 
\newpage

{

\section{Emergent hypernetworks in an  integrate-and-fire model}

\textbf{Integrate and fire model.} We used an autocatalytic integrate-and-fire model \cite{Kori2018} to simulate the
 behavior of four oscillators in a ring configuration with state variable $v_k$, and a parameter for each oscillator
  $p_k$ that determines whether the variable is increasing or decreasing. In the model, we introduce nonlinear 
  time-delayed coupling, and the oscillators are governed by the equations

\begin{equation}
\label{IF-equ}
\frac{dv_k}{dt} = \frac{p_{k}v_{k} - (1-p_{k})v_{k}B}{F_{k}} +p_{k}K \sum_{\l=1}^4 A_{k,l}(\tilde{v}_k+\tilde{v}_k^2)\tilde{v}_l(t-\tau)
\end{equation}
where $F_k$ is a rescaling factor that affects the natural frequency of $k$th oscillator, $K$ is the coupling strength, $\tilde{v}_k$ is the signal corrected for offset ($\tilde{v}_k=v_k-0.626$), $A_{k,l}$ is the adjacency matrix, and $\tau$ is the time delay.

When the variable $v_k$ reaches 1 from below, then $p_k$ is smoothly set to 0, and $v_k$ decreases. 
Similarly, when the variable $v_k$ reaches from above $A$, $p_k$ is set to 1, and the variable starts to increase.
We selected the threshold parameter $A=0.36$ and the timescale parameter $B=3.333$ so that only the one-cluster
 is stable with positive coupling. 
 Then we adjusted the parameter $F_k$ ($F_1$=4.950, $F_2$=1.955, $F_3$=3.177, $F_4$=1.970) of each oscillator to have a frequency ratio with respect oscillator 1 as $\omega_2/\omega_1 \approx 2.5$, $\omega_3/\omega_1 \approx 1.5$  and $\omega_4/\omega_1$$\approx 2.5$. Note that $F_k$ only affects the local dynamics of oscillator $k$ and not the coupling term. Figure ~\ref{fig:sim_iaf}~(a) shows the time series of the variable $v_k$ for $K$=0.234 and $\tau$=1.65 s. 
\\
\begin{figure}[!ht]
    \centering
    \includegraphics[width=1\columnwidth]{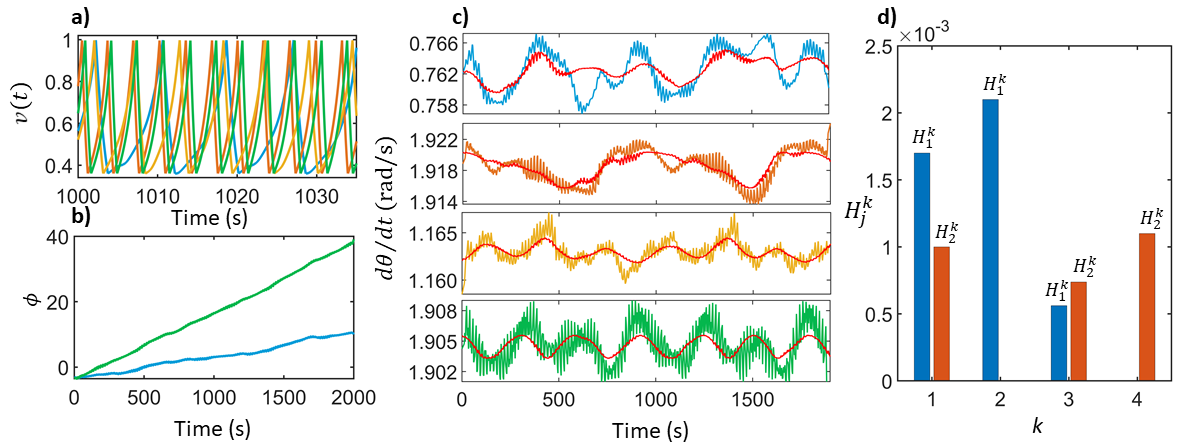}
    \caption{{\bf Simulations and network reconstruction with the integrate-and-fire model}  a) Time series of the $v_k$ variable. The blue, orange, yellow, and green lines correspond to oscillators one to four.  b) Time series of the slow phases, $\phi_1$ (blue) and $\phi_2$ (green) with coupling and delay. c) The instantaneous frequency and the 
    phase model fitted values (red) for oscillators 1 to 4 (corresponding from top to bottom). d) Coupling amplitudes for the four oscillators ($k$) from hypernetworks one (blue) and 2 (red).}
    \label{fig:sim_iaf}
\end{figure}

\textbf{Fitting of phase dynamics.} Similar to the experiments, we extract the phase of each oscillator using the peak-finding approach \cite{pikovsky2003synchronization} from the time series of the variable $v_k$. When there is coupling and delay, the triplet phase differences,  $\phi_j$, $j$=1, 2, show a phase slip behavior Supplementary Fig.~\ref{fig:sim_iaf}~(b). As described in the main text we used LASSO to fit the $\dot{\theta}_k$ values accoriding to equation ~(\ref{eq:exp}) with drifting in the natural frequencies. The time series of the $v_k$ variable (see Supplementary Fig. ~\ref{fig:sim_iaf}~(a)) showed a more nonlinear wave form, and we fitted the amplitudes of  sin and cos until the seond order harmonics ($C_{j,2}^k$ and $D_{j, 2}^k$). The $\dot{\theta}_k$ was filtered by a first order Savitzky-Golay filter for 125 s. Supplementary Fig.~\ref{fig:sim_iaf} ~(c) shows the corresponding fits for oscillator 1 to 4. (In the LASSO 
fit,  we used a regularization parameter that represented an error 40\% higher than the best fit). The fitted parameters are shown in Supplementary Fig. \ref{fig:coeff_tab}. 

\begin{figure}[!ht]
    \centering
    \includegraphics[width=0.6\columnwidth]{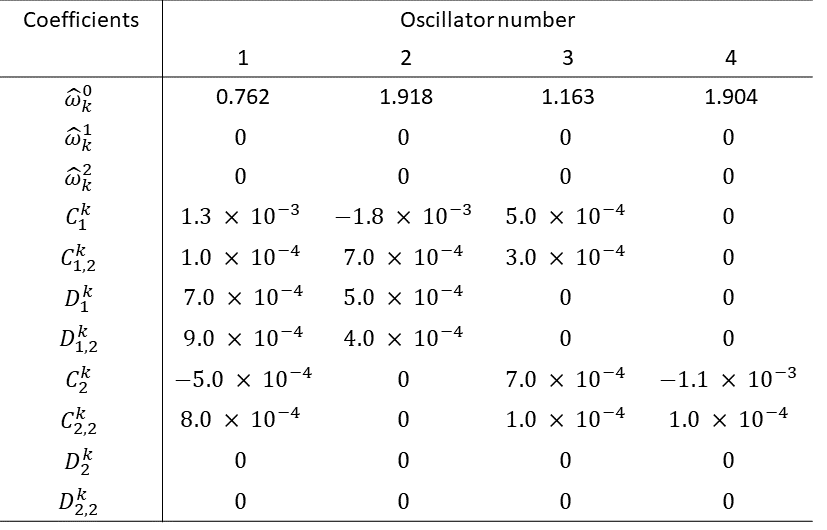}
    \caption{{\bf Hypernetwork fitting coefficients for phase dynamics of the integrate-and-fire model. }  }
    \label{fig:coeff_tab}
\end{figure}

The strength of the triplet interactions on oscillator $k$ is given by the amplitudes of the first and second harmonics
($H_j^k$); the amplitudes are shown in Supplementary Fig. \ref{fig:sim_iaf}d. In agreement with the experiments, the dynamics of oscillators 1 and 3 are impacted by both triplet interactions $\phi_1$ and $\phi_2$. For oscillator 1, the amplitudes are 1.7 $\times 10^{-3}$ and 1.0 $\times 10^{-3}$, and for oscillators 3 the amplitudes are
5.6 $\times 10^{-4}$ and 7.4 $\times 10^{-4}$ respectively. However, the dynamics of oscillator 2 and 4 are only impacted by $\phi_1$ (2.1 $\times 10^{-3}$) and $\phi_2$ (1.1 $\times 10^{-3}$). 

We conclude that in an integrate-and-fire model, the phase dynamics of the oscillators coupled in a ring can be described by an emergent hypernetwork.

\section{Mean field interaction}

We consider the system
\begin{align}\label{theODE00pluslinsasd}
\dot{z}_k &= \gamma_kz_k - \beta_k z_k|z_k|^2 + \alpha \sum_{\ell =1}^n A_{k\ell}(z_{\ell} + \bar{z}_{\ell}z_k^2)\, ,
\end{align}
where $A$ is the 4-ring network, with nodes labelled $1$ through $4$ along the ring. These frequencies satisfy the resonance conditions $\omega_1 + \omega_3 \approx 2\omega_2$ and $\omega_2 + \omega_4 \approx 2\omega_1$. A priori, it is unclear what the behavior of the system \eqref{theODE00pluslinsasd} will look like. To elucidate this, we conjugate the system by a transformation designed to get rid of the third order coupling terms in $\alpha$. To this end, we define new coordinates
 \begin{equation}
w_k = z_k - \alpha \sum_{\ell =1}^n \frac{A_{k\ell}}{\gamma_k + \bar{\gamma}_{\ell}}z_k^2\bar{z}_{\ell} \, .
 \end{equation}
 This causes new terms in $\alpha$ to appear, related to the $\beta_k z_k|z_k|^2$ terms in Equation \eqref{theODE00pluslinsasd}. We therefore perform another coordinate transformation
  \begin{equation}
u_k = w_k - \alpha Q_k(w) \, ,
 \end{equation}
 for some suitably chosen polynomials $Q_k$. We get equations for $\dot{u}_k$, which involve, among others, combinations of the linear and non-linear terms in the coupling 
 \begin{equation}
  \alpha \sum_{\ell =1}^n A_{k \ell}(z_{\ell} + \bar{z}_{\ell}z_k^2)\, .
 \end{equation}
We then discard non-resonant terms in $\alpha^2$, which leaves the equations 
\begin{align}\label{RedZ}
\dot{u}_1 &= \gamma_1u_1 - \beta_1 u_1|u_1|^2 + \alpha(u_2 + u_4) + \alpha^2\frac{u_{2}^2\bar{u}_{3}}{\gamma_{2} + \bar{\gamma}_3} + \text{h.o.t.}\\ \nonumber
\dot{u}_2 &= \gamma_2u_2 - \beta_2 u_2|u_2|^2 + \alpha(u_1 + u_3) + \alpha^2\frac{u_{1}^2\bar{u}_{4}}{\gamma_{1} + \bar{\gamma}_4} + \text{h.o.t.}\\ \nonumber
\dot{u}_3 &= \gamma_3u_3 - \beta_3 u_3|u_3|^2 + \alpha(u_2 + u_4) + \alpha^2\frac{u_{2}^2\bar{u}_{1}}{\gamma_{2} + \bar{\gamma}_1} + \text{h.o.t.}\\ \nonumber
\dot{u}_4 &= \gamma_4u_4 - \beta_4 u_4|u_4|^2 + \alpha(u_1 + u_3) + \alpha^2\frac{u_{1}^2\bar{u}_{2}}{\gamma_{1} + \bar{\gamma}_2}+ \text{h.o.t.}\, .
\end{align}
See Section \ref{Normal_form_calculations} for more details on these normal form calculations. If we ignore the (non-resonant) terms  $\alpha(u_2 + u_4)$ and $\alpha(u_1 + u_3)$ in Equation \eqref{RedZ}, then averaging yields the emergent phase dynamics. To this end, we set $\varphi_1 := \phi_1-2\phi_2+\phi_3$ and $\varphi_2 := \phi_2-2\phi_1+\phi_4$ for the slow phases. The different monomials in Equation \eqref{RedZ} then yield terms in the phase equations according to:
\begin{itemize}
\item for node 1,  $\frac{u_{2}^2\bar{u}_{3}}{\gamma_{2} + \bar{\gamma}_3}$ gives terms involving sin/cos of $\varphi_1$;
\item for node 2,  $\frac{u_{1}^2\bar{u}_{4}}{\gamma_{1} + \bar{\gamma}_4}$ gives terms involving sin/cos of $\varphi_2$;
\item for node 3,  $\frac{u_{2}^2\bar{u}_{1}}{\gamma_{2} + \bar{\gamma}_1}$ gives terms involving sin/cos of $\varphi_1$;
\item for node 4,  $\frac{u_{1}^2\bar{u}_{2}}{\gamma_{1} + \bar{\gamma}_2}$ gives terms involving sin/cos of $\varphi_2$.
\end{itemize}

\subsection{Frequency shifts}

The linear terms $\alpha(u_2 + u_4)$ and $\alpha(u_1 + u_3)$ nevertheless have an effect on the emergent dynamics, in the following way. Whereas the natural frequencies of the uncoupled system (i.e. for $\alpha = 0$) are given by $\omega_1, \dots, \omega_4$, they are in general given by the imaginary part of the eigenvalues of the perturbed matrix 
$$U = \lambda I + i\Omega + \alpha A$$ 
where, $\Omega$ is the diagonal matrix with entries $\omega_1, \dots, \omega_4$, and $A$ is the adjacency matrix of the network. Eigenvalue perturbation then gives augmented frequencies of the form $\omega_k + O(\alpha^2)$. Note that the frequency perturbation is again of order $\alpha^2$. This can be explained by a linear transformation bringing the perturbed system $\lambda I + i\Omega + \alpha A$ to that of the form $\lambda I + i\Omega + \alpha^2 B$, similar to our techniques for non-linear terms.

Therefore, whenever $\alpha>0$ the frequencies will shift providing a frequency mismatch between the slow phases 
\begin{align}
\varphi_1 = \theta_1 -2 \theta_2  +\theta_3\\
\varphi_2 = \theta_2 -2 \theta_1 + \theta_4
\end{align}
namely, they will be modelled as 
\begin{align}\label{eq:mf_fitting}
\dot{\varphi}_{1,2} = \varepsilon_{1,2} + G_{1,2}(\varphi_1) +H_{1,2}(\varphi_2).
\end{align}
where $\varepsilon_{1,2}=O(\alpha^2)$. 
\subsection{Model Recovery}
We integrate Eq.~(\ref{theODE00pluslinsasd}) with  $\Omega_1 = 2, \Omega_2 = 3, \Omega_3 = 4$ and $\Omega_4 = 1$, by employing a wrapper of ODEPACK routine. Numerical integration was performed for 25000s with 0.01 time step. We discard the first 5000s points as transient.

We apply sparse regression using PySINDy Python package \cite{kaptanoglu2022} with the Lasso optimizer on the phases $\theta_i$ considering the slow phases $\varphi_{1,2}$ with a penalty term $\lambda = 5\times10^{-3}$, we obtain

\begin{align}
\dot{\theta}_1 &=2.001 + 0.018 \cos(\varphi_1) \\
\dot{\theta}_2 &=2.999 - 0.015 \cos(\varphi_2) \\
\dot{\theta}_3 &=3.992 - 0.011 \cos(\varphi_1) \\
\dot{\theta}_4 &=1.008 + 0.011 \cos(\varphi_2). 
\end{align}

To recover the slow phase dynamics of $\varphi_1$ and $\varphi_2$  we apply the Lasso method  with a penalty term $\lambda = 10^{-5}$ after applying a rolling window averaging process using window size of 100s to smooth the fast oscillations to have better fit on slow phases. The obtained equation after the Lasso approach reads as
\begin{align}\label{eq:3nodes_recon}
\dot{\varphi_1} &= -0.008 + 0.002 \sin(\varphi_1) + 0.001 \cos(\varphi_2) \\ 
\dot{\varphi_2} &= \,\,\,\,\, 0.008 - 0.001 \cos(\varphi_1) + 0.002 \sin(\varphi_2). 
\end{align}
The theory and the fitting are also in a perfect agreement for this mean-field case.

\subsection{Normal Form Calculations}\label{Normal_form_calculations}
Here we consider the case where we have both (non-resonant) linear coupling terms as well as higher order ones. More precisely, we consider the system
\begin{align}\label{theODE00pluslin}
\dot{z}_k &= \gamma_kz_k - \beta_k z_k|z_k|^2 + \alpha \sum_{\ell =1}^n c_{k,\ell}z_{\ell} + \alpha H_k(z) \, ,
\end{align}
where $H_k$ has only terms of degree $3$ and higher. Later, we will set $H_k(z) =  \sum_{\ell =1}^n c_{k,\ell}z_k^2\bar{z_{\ell}}$. We assume that corresponding functions $P_k$ exist that solve
\begin{align}\label{theODE00pluslin1}
\Gamma P_k - \gamma_k P_k = H_k\, .
\end{align}
In particular, when $H_k(z) =  \sum_{\ell =1}^n c_{k,\ell}z_k^2\bar{z_{\ell}}$ we assume that $\omega_k \not= \omega_{\ell}$ whenever $c_{k,\ell} \not= 0$, so that we may define
\begin{align}\label{theODE00pluslin2}
P_k(z) = \sum_{\ell =1}^n \frac{c_{k,\ell}}{\gamma_k + \bar{\gamma}_{\ell}}z_k^2\bar{z}_{\ell} \, .
\end{align}
With slight abuse of notation, the sum in Equation \eqref{theODE00pluslin2} is taken over all $\ell$ such that $c_{k,\ell} \not= 0$. As before, we consider the coordinate transformation $w_k = z_k - \alpha P_k(z)$, which gives 
\begin{equation}\label{expresssss122}
z_k = w_k + \alpha P_k(w) + \alpha^2 [P_k||P](w) + \mathcal{O}(|\alpha|^3|w|^7)\, ,
\end{equation}
by lemmas \ref{lemzinw0} and \ref{fullinverse0}. A calculation as before reveals that 
\begin{align}
\dot{w}_k &= \gamma_kw_k - \beta_k w_k|w_k|^2 + \alpha \sum_{\ell =1}^n c_{k,\ell}w_{\ell} + \alpha (\gamma_kP_k(w) - \Gamma P_k(w) + H_k(w)) \\ \nonumber
&+  \alpha(L^1_k(w) - L^2_k(w)) + \alpha^2[\gamma_kP_k - \Gamma P_k + H_k || P](w) \\ \nonumber
&+\alpha^2\sum_{\ell =1}^n c_{k,\ell}P_{\ell}(w) - \alpha^2[P_k|| (\dots, \sum_{\ell =1}^n c_{j,\ell}w_{\ell}, \dots )](w) - \alpha^2[P_k || H](w) \\ \nonumber
&+ \mathcal{O}(|\alpha|^2|w|^7 + |\alpha|^3|w|^5) \\ \nonumber
&= \gamma_kw_k - \beta_k w_k|w_k|^2 + \alpha \sum_{\ell =1}^n c_{k,\ell}w_{\ell} +  \alpha(L^1_k(w) - L^2_k(w))  \\ \nonumber
&+\alpha^2\sum_{\ell =1}^n c_{k,\ell}P_{\ell}(w) - \alpha^2[P_k|| (\dots, \sum_{\ell =1}^n c_{j,\ell}w_{\ell}, \dots )](w) - \alpha^2[P_k || H](w) \\ \nonumber
&+ \mathcal{O}(|\alpha|^2|w|^7 + |\alpha|^3|w|^5) 
\end{align}
where in the last step we have used Equation \eqref{theODE00pluslin1}, and where we again set
\begin{align}
L^1_k(w) &:=  [P_k|| (\dots, \beta w_j|w_j|^2, \dots)](w) \text{ and }\\ \nonumber
L^2_k(w) &:=  [\beta_k w_k |w_k|^2|| P](w)\, .
\end{align}
Next, we perform a second transformation $u_k = w_k - \alpha Q_k(w)$, where $Q_k$ solves 
\begin{align}\label{theODE00pluslin12w}
\Gamma Q_k - \gamma_k Q_k = L^1_k - L^2_k\, ,
\end{align}
and where
\begin{equation}\label{invssofw}
    w_k = u_k + \alpha Q_k(u) + \mathcal{O}(|\alpha|^2|u|^9)\, .
\end{equation}
This gives
\begin{align}
\dot{u}_k &= \gamma_ku_k - \beta_k u_k|u_k|^2 + \alpha \sum_{\ell =1}^n c_{k,\ell}u_{\ell} + \alpha (\gamma_kQ_k(u) - \Gamma Q_k(u) + L^1_k(u) - L^2_k(u)) \\ \nonumber
&+  \alpha^2 \sum_{\ell =1}^n c_{k,\ell}Q_{\ell}(u) - \alpha^2[Q_k|| (\dots, \sum_{\ell =1}^n c_{j,\ell}u_{\ell}, \dots )](u)\\ \nonumber
&+\alpha^2\sum_{\ell =1}^n c_{k,\ell}P_{\ell}(u) - \alpha^2[P_k|| (\dots, \sum_{\ell =1}^n c_{j,\ell}w_{\ell}, \dots )](u) - \alpha^2[P_k || H](u) \\ \nonumber
&+ \mathcal{O}(|\alpha||u|^7 + |\alpha|^3|u|^5) \, .
\end{align}
By Equation \eqref{theODE00pluslin12w} we therefore get
\begin{align}\label{lastofabstra}
\dot{u}_k &= \gamma_ku_k - \beta_k u_k|u_k|^2 + \alpha \sum_{\ell =1}^n c_{k,\ell}u_{\ell}  \\ \nonumber
&+  \alpha^2 \sum_{\ell =1}^n c_{k,\ell}Q_{\ell}(u) - \alpha^2[Q_k|| (\dots, \sum_{\ell =1}^n c_{j,\ell}u_{\ell}, \dots )](u)\\ \nonumber
&+\alpha^2\sum_{\ell =1}^n c_{k,\ell}P_{\ell}(u) - \alpha^2[P_k|| (\dots, \sum_{\ell =1}^n c_{j,\ell}w_{\ell}, \dots )](u) - \alpha^2[P_k || H](u) \\ \nonumber
&+ \mathcal{O}(|\alpha||u|^7 + |\alpha|^3|u|^5) \, .
\end{align}
We now return to the special case of  $H_k(z) =  \sum_{\ell =1}^n c_{k,\ell}z_k^2\bar{z_{\ell}}$. By assumption, the terms in $\alpha$ are resonant. Of the terms in $\alpha^2$, only 
\begin{align}\label{lastofabstra2}
\alpha^2\sum_{\ell =1}^n c_{k,\ell}P_{\ell}(u) - \alpha^2[P_k|| (\dots, \sum_{\ell =1}^n c_{j,\ell}w_{\ell}, \dots )](u)
\end{align}
is third order in $u$ (as opposed to fifth order). A direct calculation shows that
\begin{align}\label{lastofabstra3}
&\sum_{\ell =1}^n c_{k,\ell}P_{\ell}(u) - [P_k|| (\dots, \sum_{\ell =1}^n c_{j,\ell}w_{\ell}, \dots )](u) \\ \nonumber
&= \sum_{\ell =1}^n \sum_{p =1}^n \frac{c_{k,\ell} c_{\ell,p}}{\gamma_{\ell} + \bar{\gamma}_p}u_{\ell}^2\bar{u}_{p} - \sum_{\ell =1}^n \sum_{p =1}^n \frac{c_{k,\ell} c_{\ell,p}}{\gamma_{k} + \bar{\gamma}_{\ell}}u_{k}^2\bar{u}_{p} - 2\sum_{\ell =1}^n \sum_{p =1}^n \frac{c_{k,\ell} c_{k,p}}{\gamma_{k} + \bar{\gamma}_{\ell}}u_{k}\bar{u}_{\ell}u_{p}  \, .
\end{align}
From these the resonant terms can be selected, which leads to a hypernetwork description of the dynamics.

}


\end{document}